\documentclass[12pt]{amsart}
\usepackage[margin=0.5in]{geometry}
\usepackage{graphicx,epstopdf,color}
\usepackage{amssymb,amscd,amsthm,amsxtra}

\usepackage[font=footnotesize]{caption}

\usepackage{mathrsfs}
\renewcommand{\mathcal}{\mathscr}

\numberwithin{equation}{section}

\newtheorem{theorem}{Theorem}[section]
\newtheorem{lemma}[theorem]{Lemma}
\newtheorem{proposition}[theorem]{Proposition}
\newtheorem{remark}[theorem]{Remark}
\newtheorem{corollary}[theorem]{Corollary}

\newcommand{\R}{\mathbb R}
\newcommand{\N}{\mathbb N}

\newcommand{\Z}{\mathbb Z}
\newcommand{\C}{\mathbb C}

\renewcommand{\le}{\leqslant}

\renewcommand{\ge}{\geqslant}
\renewcommand{\epsilon}{\varepsilon}
\newcommand{\eps}{\varepsilon}

\newcommand{\per}{\,{\rm Per}_s}

\begin{document}

\author[S. Dipierro]{Serena Dipierro}
\address[Serena Dipierro]{School of Mathematics and Statistics,
University of Melbourne,
813 Swanston Street, Parkville VIC 3010, Australia,
and
Fakult\"at f\"ur Mathematik,
Institut f\"ur Analysis und Numerik,
Otto-von-Guericke-Universit\"at Magdeburg,
Universit\"atsplatz 2, 39106 Magdeburg, Germany}
\email{serena.dipierro@ed.ac.uk}

\author[O. Savin]{Ovidiu Savin}
\address[Ovidiu Savin]{Department of Mathematics, Columbia University,
2990 Broadway,
New York NY 10027, USA}
\email{savin@math.columbia.edu}

\author[E. Valdinoci]{Enrico Valdinoci}
\address[Enrico Valdinoci]{School of Mathematics and Statistics,
University of Melbourne,
813 Swanston Street, Parkville VIC 3010, Australia,
Weierstra{\ss} Institut f\"ur Angewandte
Analysis und Stochastik, Hausvogteiplatz 5/7, 10117 Berlin, Germany,
Dipartimento di Matematica, Universit\`a degli studi di Milano,
Via Saldini 50, 20133 Milan, Italy, and
Istituto di Matematica Applicata e Tecnologie Informatiche,
Consiglio Nazionale delle Ricerche,
Via Ferrata 1, 27100 Pavia, Italy.}
\email{enrico@math.utexas.edu}

\title{Boundary behavior of nonlocal minimal surfaces}
\thanks{The first author has been supported by EPSRC grant  EP/K024566/1
``Monotonicity formula methods for nonlinear PDEs'',
ERPem
``PECRE Postdoctoral and Early Career Researcher Exchanges'' and Alexander
von Humboldt Foundation.
The second author has been supported by
NSF grant DMS-1200701.
The third author has been supported by ERC grant 277749 ``EPSILON Elliptic
Pde's and Symmetry of Interfaces and Layers for Odd Nonlinearities''
and PRIN grant 201274FYK7
``Aspetti variazionali e
perturbativi nei problemi differenziali nonlineari''.}

\begin{abstract}
We consider the behavior of the nonlocal minimal surfaces
in the vicinity of the boundary. By a series of detailed
examples, we show that nonlocal minimal surfaces
may stick at the boundary of the domain, even 
when the domain is smooth and convex. This is a purely
nonlocal phenomenon, and it is in sharp contrast with
the boundary properties of the classical minimal surfaces.

In particular, we show stickiness phenomena
to half-balls when the datum outside the ball is
a small half-ring and to the side of a two-dimensional box when
the oscillation between the datum on the right and on the left
is large enough.

When the fractional parameter is small,
the sticking effects may become more and more evident.
Moreover, we show that lines in the plane are unstable
at the boundary: namely, small compactly supported perturbations
of lines cause the minimizers in a slab to stick at the boundary,
by a quantity that is proportional to a power of the perturbation.

In all the examples, we present concrete estimates
on the stickiness phenomena.
Also, we construct a family of compactly supported barriers
which can have independent interest.
\end{abstract}

\subjclass[2010]{49Q05, 35R11, 53A10}
\keywords{Nonlocal minimal surfaces, boundary regularity, barriers.}
\maketitle

\section{Introduction}

It is well known (see e.g. \cite{HS, DS}) that the classical minimal surfaces
do not stick at the boundary. Namely, if~$\Omega$ is a convex
domain and~$E$ is a set that minimizes the perimeter among
its competitors in~$\Omega$, then~$\partial E$ is transverse to~$\partial\Omega$
at their intersection points.

In this paper we show that the situation for the nonlocal
minimal surfaces is completely different.
Indeed, we prove that nonlocal interactions can favor stickiness at
the boundary for minimizers of a fractional perimeter.

The mathematical framework in which we work 
was introduced in~\cite{CRS} and is the following. Given~$s\in(0,1/2)$
and an open set~$\Omega\subseteq\R^n$,
we define the $s$-perimeter of a set~$E\subseteq\R^n$ in~$\Omega$ as
$$ \per(E,\Omega) := L(E\cap\Omega, E^c)+L(\Omega\setminus E,
E\setminus\Omega),$$
where~$E^c:=\R^n\setminus E$ and, for any
disjoint sets~$F$ and~$G$, we use the notation
$$ L(F,G):=\iint_{F\times G} \frac{dx\,dy}{|x-y|^{n+2s}}.$$
We say that~$E$ is $s$-minimal in~$\Omega$ if~$\per(E,\Omega)<+\infty$
and~$\per(E,\Omega)\le \per(F,\Omega)$ among all the sets~$F$
which coincide with~$E$ outside~$\Omega$.

With a slight abuse of language, when~$\Omega$
is unbounded, we say that~$E$ is $s$-minimal
in~$\Omega$ if it is $s$-minimal
in any bounded open subsets of~$\Omega$ (for a more
precise distinction between $s$-minimal sets
and locally $s$-minimal sets see e.g.~\cite{luca}).
\medskip

Problems related to the $s$-perimeter naturally arise
in several fields, such as the motion by nonlocal mean curvature
and the nonlocal Allen-Cahn equation, see e.g.~\cite{SOU, GAMMA}.
Also, the $s$-perimeter can be seen as a fractional interpolation
between the classical perimeter (corresponding to the case~$s\to1/2$)
and the Lebesgue measure (corresponding to the case~$s\to0$),
see e.g.~\cite{MAZYA, BREZIS, CV, MARTIN, DFPV}.
\medskip

The field of nonlocal minimal surfaces
is rich of open problems and surprising examples (see e.g.~\cite{LAWSON})
and the interior regularity theory of the nonlocal
minimal surfaces has been established in the plane and when the
fractional parameter is close enough to~$1/2$
(see~\cite{CV-REG, SV-REG}), but, as far as we know,
the boundary behavior of the nonlocal minimal surfaces
has not been studied till now.
\medskip

We show in this paper that the boundary datum
is not, in general, attained continuously. Indeed,
nonlocal minimal surfaces may stick at the boundary
and then detach from the boundary in a~$C^{1,\frac{1}{2}+s}$-fashion.
We will give concrete examples of this stickiness phenomenon
with explicit (and somehow optimal) estimates.
In particular, we will present stickiness phenomena to half-balls,
when the domain is a ball and the datum is a small half-ring,
and to the sides of a two-dimensional box, when the datum is
small on one side and large on the other side.
\medskip

Moreover, we study how small perturbations with compact support may affect
the boundary behavior of a given nonlocal minimal surface.
Quite surprisingly, these perturbations may produce stickiness effects
even in the case of flat objects and in low dimension.
For instance, adding a small perturbation to a half-space
in the plane produces a sticking effect,
with the size of the sticked portion proportional to a power
of the size of the perturbation.
We now present and discuss these results in further detail.

\subsection*{Stickiness to half-balls}

For any~$\delta>0$, we let
\begin{equation}\label{Kdelta}
K_\delta := \big( B_{1+\delta}\setminus B_1\big)\cap \{x_n<0\}.\end{equation}
We define~$E_\delta$ to be the set minimizing~$\per(E,B_1)$
among all the sets~$E$ such that~$E\setminus B_1=K_\delta$.

\begin{figure}
    \centering
    \includegraphics[height=5.9cm]{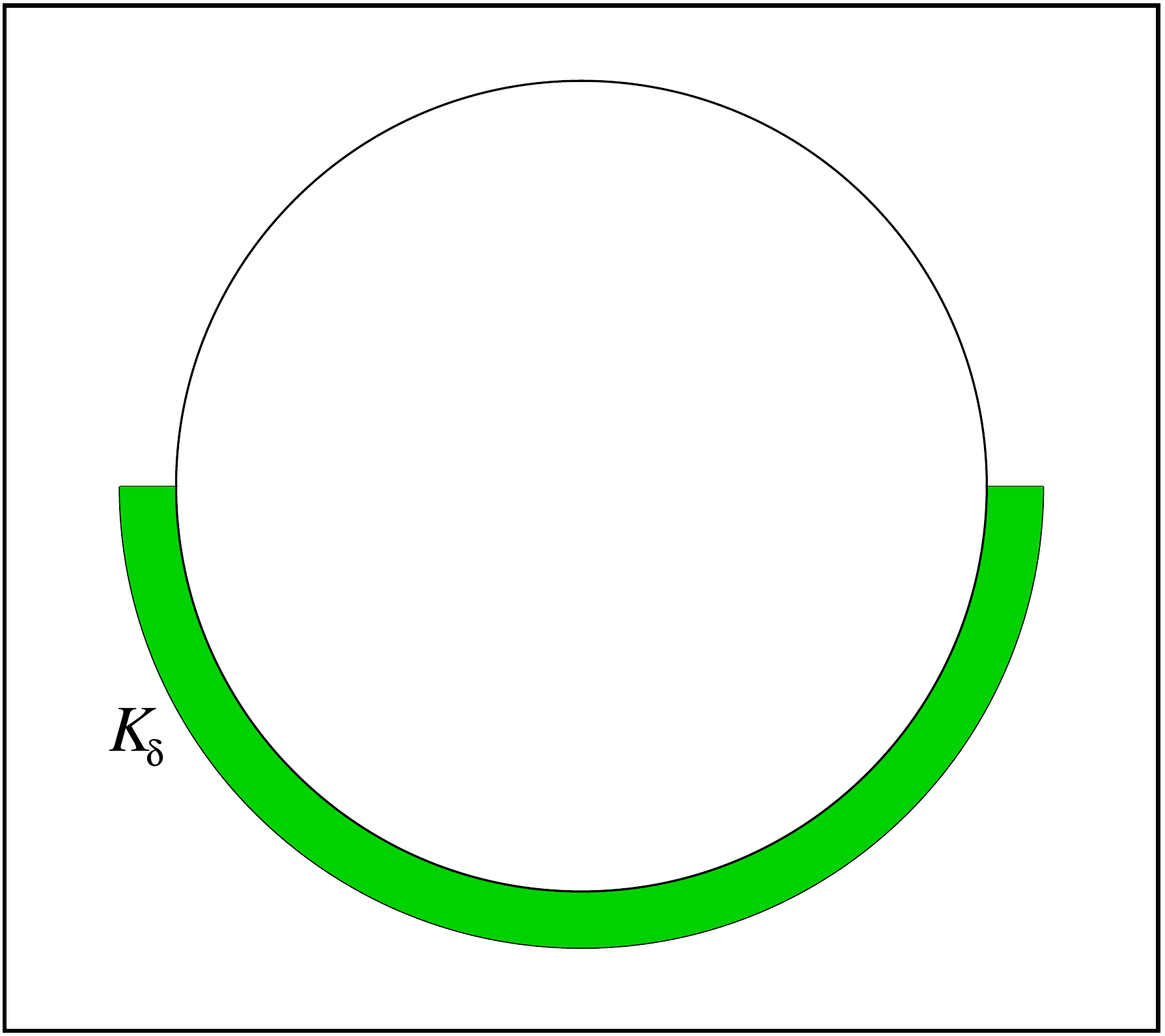}
    \caption{The stickiness property in Theorem~\ref{89=THM}.}
    \label{NESS}
\end{figure}

Notice that, in the local setting, the minimizer of the perimeter functional
that takes~$K_\delta$ as boundary value at~$\partial B_1$
is the flat set~$B_1\cap \{x_n<0\}$ (independently of~$\delta$).
The picture changes dramatically in the nonlocal framework,
since in this case the nonlocal minimizers stick at~$\partial B_1$
if~$\delta$ is suitably small, see Figure~\ref{NESS}. The formal statement of this
feature is the following:

\begin{theorem}\label{89=THM}
There exists~$\delta_o>0$,
depending on $s$ and $n$, such that
for any~$\delta\in(0,\delta_o]$ we have that
$$ E_\delta=K_\delta.$$
\end{theorem}

\subsection*{Stickiness to the sides of a box}

Given a large~$M>1$ we
consider the $s$-minimal set~$E_M$ in~$(-1,1)\times\R$
with datum outside~$(-1,1)\times\R$ given by
the jump
\begin{equation}\label{jump}
\begin{split}
& J_M:=J^-_M \cup J^+_M,\\
{\mbox{where }}\quad
& J^-_M:= (-\infty,-1]\times (-\infty,-M)
\\{\mbox{and }}\quad&J^+_M:= [1,+\infty)\times (-\infty,M).
\end{split}\end{equation}
We prove that, if~$M$ is large enough,
the minimal set~$E_M$ sticks at the boundary
(see Figure~\ref{REG:2}).
Moreover, the stickiness region gets close to the origin,
up to a power of~$M$.
The precise result is the following:

\begin{figure}
    \centering
    \includegraphics[width=9cm]{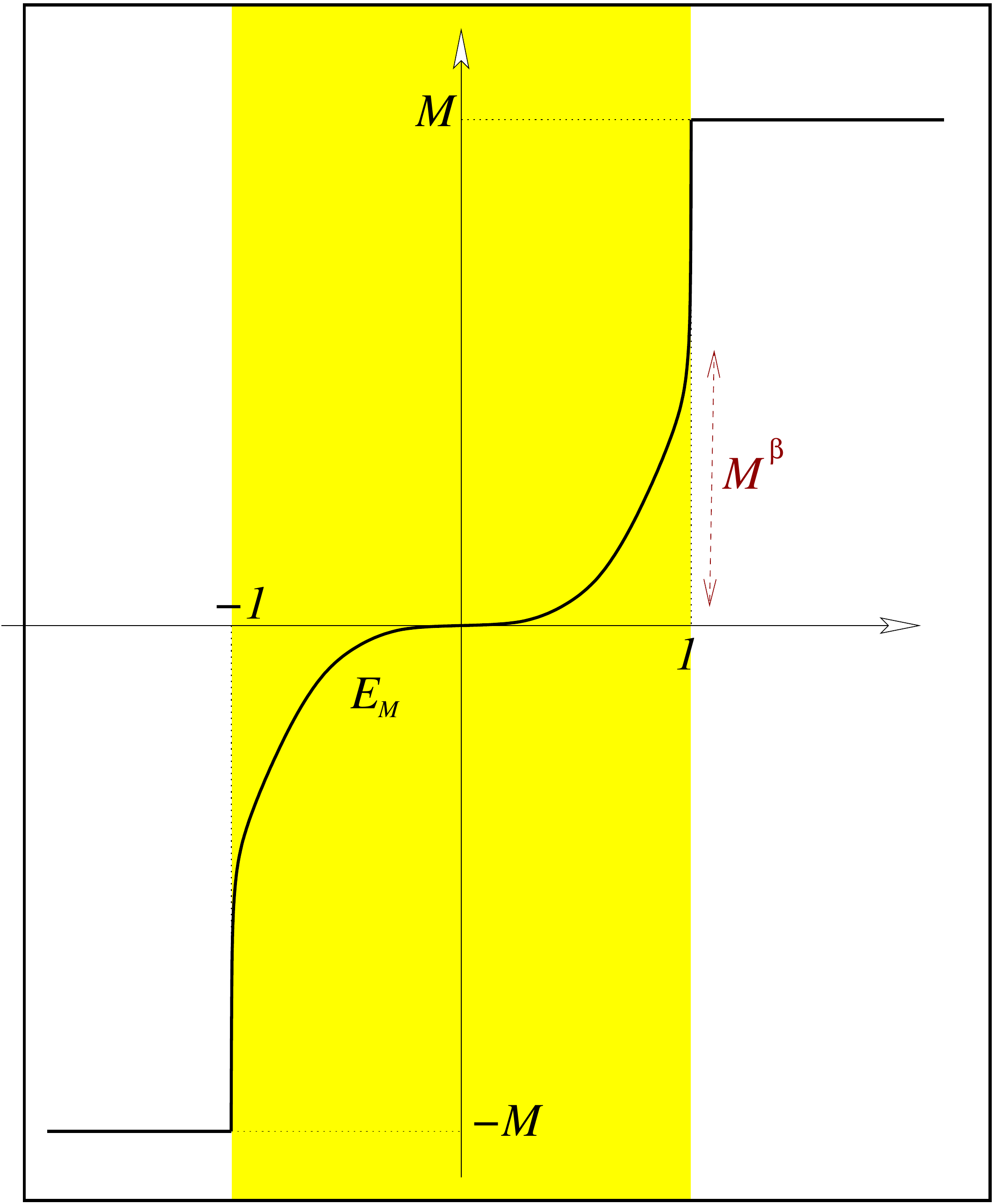}
    \caption{The stickiness property in Theorem~\ref{CYL-THEOR},
with~$\beta:={\frac{1+2s}{2+2s}}$.}
    \label{REG:2}
\end{figure}

\begin{theorem}\label{CYL-THEOR}
There exist~$M_o>0$ and~$C_o\ge C_o'>0$, depending on~$s$, such that
if~$M\ge M_o$ then
\begin{eqnarray*}
&& [-1,1)\times [C_o M^{\frac{1+2s}{2+2s}},\,M]\subseteq E_M^c 
\\{\mbox{and }}&& (-1,1]\times [-M,\,-C_o M^{\frac{1+2s}{2+2s}}]\subseteq E_M.
\end{eqnarray*}
Also, the exponent~${\frac{1+2s}{2+2s}}$ above is optimal.
For instance, if
either~$[-1,1)\times [b M^{\frac{1+2s}{2+2s}},\,M]\subseteq E_M^c$
or~$(-1,1]\times [-M,\,-C_o M^{\frac{1+2s}{2+2s}}]\subseteq E_M$
for some~$b\ge0$, then~$b\ge C_o'$.
\end{theorem}

\subsection*{Stickiness as~$s\to0^+$}

The stickiness properties of nonlocal minimal surfaces
are a purely nonlocal phenomenon and they become more evident
for small values of~$s$. To provide a confirming example,
we consider the boundary value given by
a sector in~$\R^2$ outside~$B_1$, i.e. we define
\begin{equation}\label{SECTOR-3456fgg}
\Sigma := \{ (x,y)\in\R^2\setminus B_1 {\mbox{ s.t. }}
x>0 {\mbox{ and }} y>0\}.\end{equation}
We show that as~$s\to0^+$ the $s$-minimal set in~$B_1$
with datum~$\Sigma$ sticks to~$\Sigma$,
and, more precisely, this stickiness already occurs
for a small~$s_o>0$ (see Figure~\ref{SIMP--1}).

\begin{figure}
    \centering
    \includegraphics[height=5.9cm]{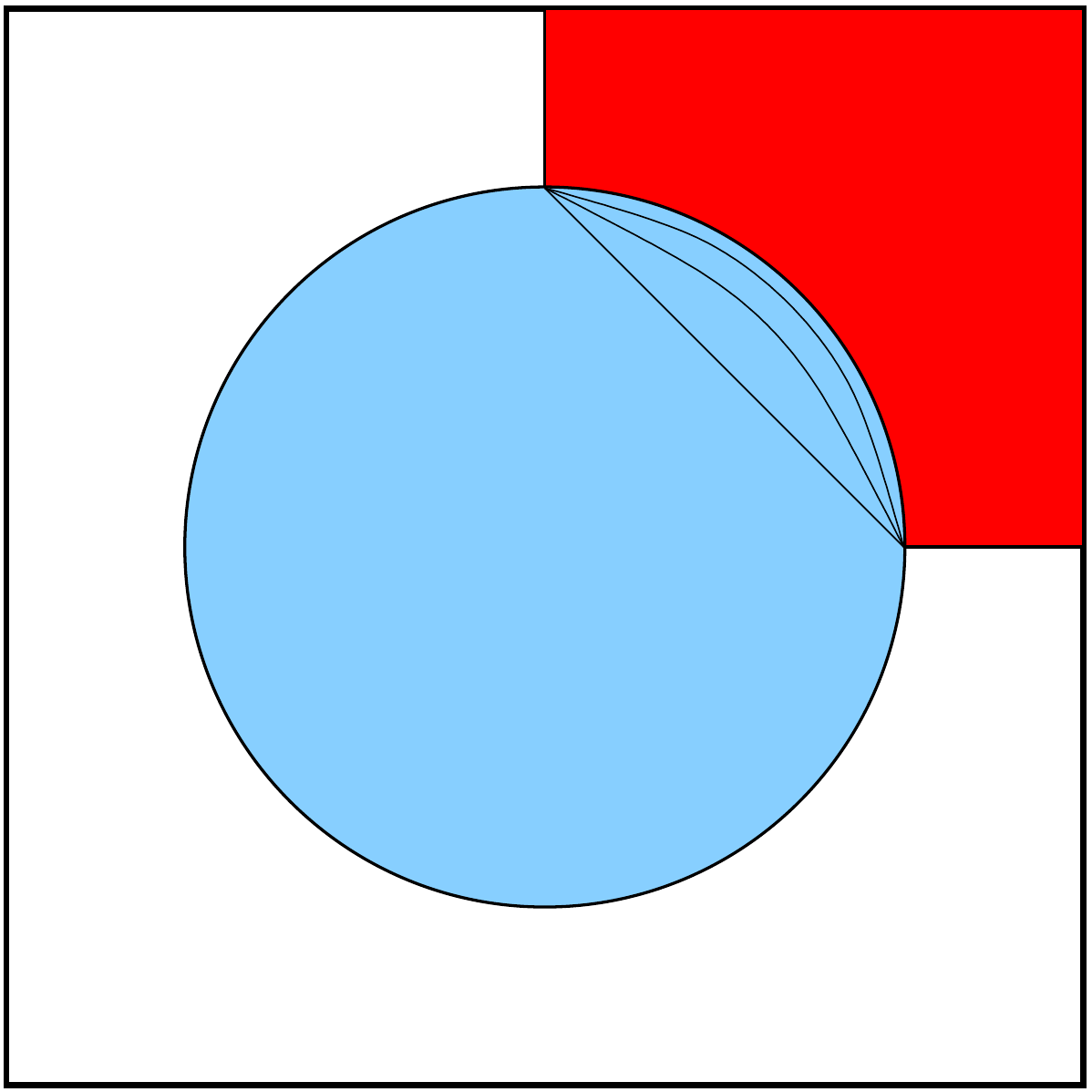}
    \includegraphics[height=5.9cm]{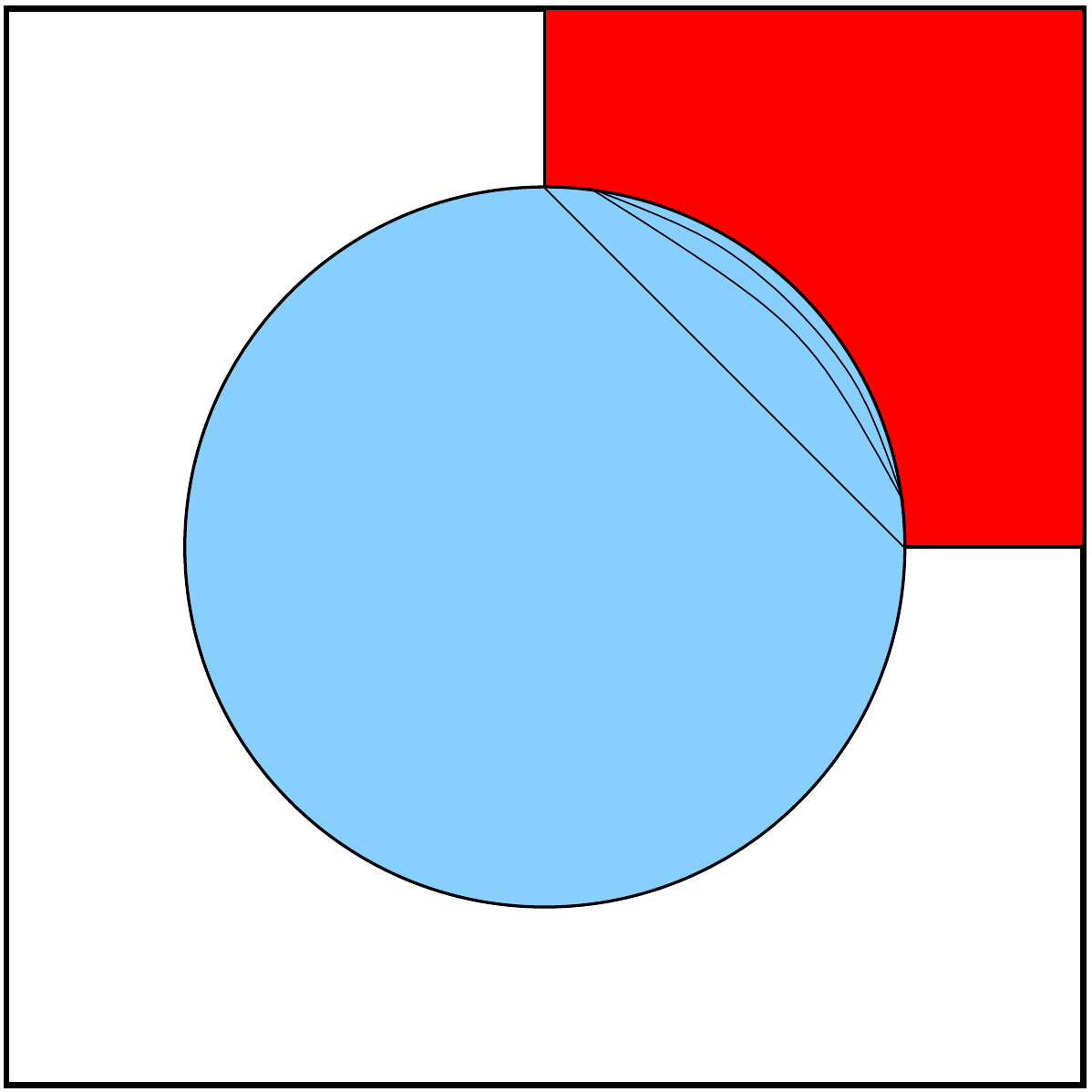}
    \caption{The stickiness property in Theorem~\ref{SECTOR}.}
    \label{SIMP--1}
\end{figure}

\begin{theorem}\label{SECTOR}
Let~$E_s$ be
the $s$-minimizer of~$\per(E,B_1)$
among all the sets~$E$ such that~$E\setminus B_1=\Sigma$.

Then, there exists~$s_o>0$ such that for any~$s\in(0,s_o]$
we have that~$E_s=\Sigma$.
\end{theorem}

\subsection*{Instability of the flat fractional minimal surfaces}

Rather surprisingly, one of our results states that
the flat lines are ``unstable'' fractional minimal surfaces,
in the sense that an arbitrarily small and compactly supported
perturbation can cause a boundary stickiness phenomenon.
We are also able to give a quantification
of the size of the stickiness in terms of the size
of the perturbation: namely the size of the stickiness
is bounded from below by the size of the perturbation
to the power~$\frac{2+\epsilon_0}{1-2s}$,
for any fixed~$\epsilon_0$ arbitrarily small (see Figure~\ref{CAMB}).
We observe that this power tends to~$+\infty$
as~$s\to1/2$, which is consistent with the fact that
classical minimal surfaces do not stick.
The precise result that we obtain is the following:

\begin{figure}
    \centering
    \includegraphics[height=5.9cm]{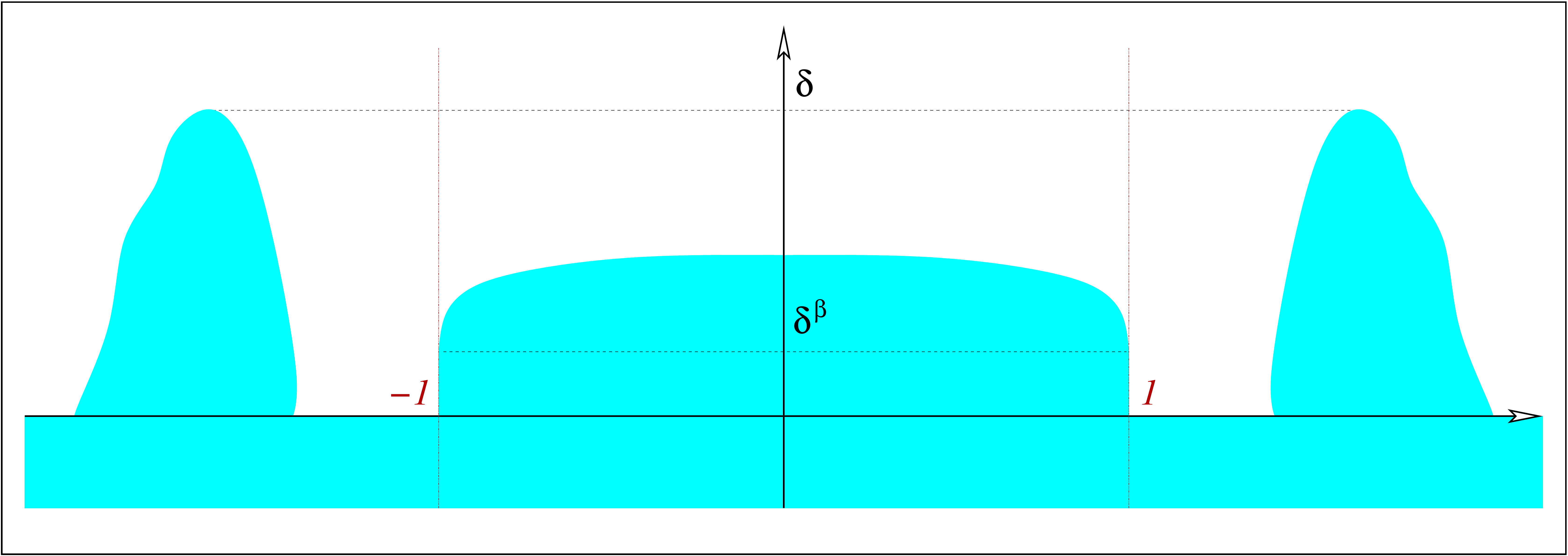}
    \caption{The stickiness/instability property in Theorem~\ref{UNS},
with~$\beta:=\frac{2+\epsilon_0}{1-2s}$.}
    \label{CAMB}
\end{figure}

\begin{theorem}\label{UNS}
Fix~$\epsilon_0>0$ arbitrarily small.
Then, there exists~$\delta_0>0$, possibly depending on~$\epsilon_0$,
such that for any~$\delta\in(0,\delta_0]$ the following statement holds true.

Assume that~$F\supset H\cup F_-\cup F_+$, where~$H:=\R\times(-\infty,0)$,
$F_-:=
(-3,-2)\times [0,\delta)$ and
$F_+:= (2,3)\times [0,\delta)$. Let~$E$ be the 
$s$-minimal set in~$(-1,1)\times\R$ among all the sets that coincide
with~$F$ outside~$(-1,1)\times\R$. Then
$$ E\supseteq (-1,1)\times (-\infty, \delta^{\frac{2+\epsilon_0}{1-2s}} ].$$
\end{theorem}

The proof of Theorem~\ref{UNS} is rather delicate
and it is based on the construction of suitable auxiliary barriers,
which we believe are interesting in themselves.
These barriers are used to detach a portion of the set
in a neighborhood of the origin and their construction
relies on some compensations of nonlocal integral terms.
As a matter of fact, the compactly supported barriers are
obtained by glueing other auxiliary barriers with polynomial growth
(the latter barriers are somehow ``self-sustaining solutions''
and can be seen as the geometric counterparts of
the $s$-harmonic function~$x_+^s$).
\medskip

Though quite surprising at a first glance, the sticking
effects that we present in this paper
have some (at least vague)
heuristic explanations.
Indeed, first of all,
the contribution to the fractional mean curvature
which comes from far may bend a nonlocal minimal surface
towards the boundary of the domain: then, the points in
the vicinity of the domain may end up receiving
a contribution which is incompatible with the vanishing
of the fractional mean curvature, due to some transverse
intersection between the datum and the domain itself,
thus forcing these points to stick at the boundary.

Another heuristic explanation of the stickiness phenomenon
comes from the different fractional scalings
that the problem exhibits at different scales.
On the one hand, vanishing of
the fractional mean curvature corresponds to a $s$-harmonicity
property 
(i.e. a harmonicity with respect to the fractional operator~$(-\Delta)^s$)
for the characteristic function of the $s$-minimal set,
with~$s\in(0,1/2)$.
If the boundary of the set is the graph
of a smooth function~$u$, this gives an equation for~$u$
whose linearization corresponds to~$(-\Delta)^{\frac{1}{2}+s}$,
which would correspond, roughly speaking, to a regularity theory
of order~$C^{\frac{1}{2}+s}$ at the boundary. On the other hand,
nonlocal minimal surfaces detach from free boundaries in
a~$C^{1,\frac{1}{2}+s}$-fashion (see~\cite{PRO}), which suggests that the linearized
equation of the graph is not a good approximation for the boundary
behavior.
\medskip

The rest of the paper is organized as follows.
In Section~\ref{S:HALF-BALL}, we discuss the case of the stickiness to
a half-ball and we prove Theorem~\ref{89=THM}.
Then, Section~\ref{SEC:BOX} considers the case
of a two-dimensional box with high oscillating datum,
providing the proof of Theorem~\ref{CYL-THEOR}.
The asymptotics as~$s\to0$ is presented in Section~\ref{SEC:to0}.

The second part of the paper is devoted to the proof of
Theorem~\ref{UNS}. In particular, Sections~\ref{SEC:BAR:PTWL}, \ref{sec:GROW:R}
and~\ref{SEC:CMP:SUPP} are devoted to the construction
of the auxiliary barriers. More precisely,
in Section~\ref{SEC:BAR:PTWL} we construct barriers with a linear
growth, by superposing straight lines with slowly varying slopes;
then, in Section~\ref{sec:GROW:R}, we 
glue the barrier with linear growth with a power-like function
(this is needed to obtain sharper estimates on the size of the glueing)
and in Section~\ref{SEC:CMP:SUPP} we adapt this construction
to build barriers that are compactly supported.

This will allow us to prove Theorem~\ref{UNS}
in Section~\ref{INST:SEC}. The paper ends with an appendix
that contains a simple, but general, symmetry property,
and an alternative proof of an integral identity.

\section{Stickiness to half-balls}\label{S:HALF-BALL}

This section is devoted to the analysis of the stickiness
phenomena to the half-ball, caused by a small half-ring
as external datum. The main goal of this part is to prove
Theorem~\ref{89=THM}. For this, we 
take~$K_\delta$ as in~\eqref{Kdelta}, i.e.
$$ K_\delta := \big( B_{1+\delta}\setminus B_1\big)\cap \{x_n<0\}$$
and~$E_\delta$ to be the set minimizing~$\per(E,B_1)$
among all the sets~$E$ such that~$E\setminus B_1=K_\delta$.

We make some auxiliary observations.
First of all, we check that the~$s$-perimeter of~$K_\delta$
(and then of the minimizer)
must be small if so is~$\delta$:

\begin{lemma}\label{89=IO}
For any~$\eps>0$ there exists~$\delta_\eps>0$ such that
for any~$\delta\in(0,\delta_\eps]$ we have that
$$ \per ( K_\delta,B_1)\le \eps.$$
\end{lemma}

\begin{proof} We have
$$ \per ( K_\delta,B_1) = L(B_1,K_\delta)\le
\iint_{B_1\times \big( B_{1+\delta}\setminus B_1\big)}
\frac{dx\,dy}{|x-y|^{n+2s}}.$$
Now we observe that
\begin{equation}\label{CVFI}
(0,+\infty)\ni \iint_{B_1\times \big( B_{2}\setminus B_1\big)}
\frac{dx\,dy}{|x-y|^{n+2s}}
= \lim_{\delta\to0^+}
\iint_{B_1\times \big( B_{2}\setminus B_{1+\delta}\big)}
\frac{dx\,dy}{|x-y|^{n+2s}}.\end{equation}
Indeed, the first integral in~\eqref{CVFI} is finite,
see for instance Lemma~11 in~\cite{CV} (applied here with~$\eps:=1$,
$\Omega:=B_2$ and~$F:=B_1$).
As a consequence of~\eqref{CVFI}, for any~$\eps>0$ 
there exists~$\delta_\eps>0$ such that
for any~$\delta\in(0,\delta_\eps]$ we have
$$ \left| \iint_{B_1\times \big( B_{2}\setminus B_{1}\big)}
\frac{dx\,dy}{|x-y|^{n+2s}}-
\iint_{B_1\times \big( B_{2}\setminus B_{1+\delta}\big)}
\frac{dx\,dy}{|x-y|^{n+2s}}\right|\le\eps,$$
which gives the desired result.
\end{proof}

Next result proves that the boundary of the minimal set~$E_\delta$
can only lie in the neighborhood of~$\partial B_1$, if~$\delta$
is small enough. More precisely:

\begin{lemma}\label{89=QW}
For any~$\eps\in(0,1)$ there exists~$\delta_\eps>0$ such that
for any~$\delta\in(0,\delta_\eps]$ we have that
$$ (\partial E_\delta)\cap B_{1-\eps} =\varnothing. $$
\end{lemma}

\begin{proof} 
We observe that it is enough to prove the desired claim for small~$\eps$
(since this would imply the claim for bigger~$\eps$).
The proof is by contradiction.
Suppose that there exists~$p\in (\partial E_\delta)\cap B_{1-\eps}$.
Then~$B_{\eps/2}(p)\subset B_1$ and so, by the Clean Ball Condition
(see Corollary 4.3 in~\cite{CRS}), there exist~$p_1$, $p_2\in B_1$ such that
$$ B_{c\eps}(p_1)\subset E\cap B_{\eps/2}(p)\qquad
{\mbox{ and }}\qquad
B_{c\eps}(p_2)\subset E^c\cap B_{\eps/2}(p),$$
for a suitable constant~$c>0$.
In particular, both~$B_{c\eps}(p_1)$ and~$B_{c\eps}(p_2)$ lie inside~$B_1$,
and if~$x\in B_{c\eps}(p_1)$ and~$y\in B_{c\eps}(p_2)$ then~$|x-y|\le\eps$.
As a consequence
$$ \per (E_\delta,B_1)\ge L\big( B_{c\eps}(p_1),B_{c\eps}(p_2)\big)
\ge \frac{|B_{c\eps}(p_1)|\,|B_{c\eps}(p_2)|}{\eps^{n+2s}} = c_o \eps^{n-2s},$$
for some~$c_o>0$.
On the other hand, by Lemma~\ref{89=IO} (used here with~$\eps^n$
in the place of~$\eps$),
we have that~$\per (E_\delta,B_1)\le
\per ( K_\delta,B_1)\le\eps^n$ provided that~$\delta$
is suitably small with respect to~$\eps$. As a consequence,
we obtain that~$\eps^n \ge c_o \eps^{n-2s}$,
which is a contradiction if~$\eps$ is small enough.
\end{proof}

The statement of Lemma~\ref{89=QW} can be better specified,
as follows:

\begin{corollary}\label{89=QW-2}
For any~$\eps\in(0,1)$ there exists~$\delta_\eps>0$ such that
for any~$\delta\in(0,\delta_\eps]$ we have that
$$ E_\delta\cap B_{1-\eps} =\varnothing. $$
\end{corollary}

\begin{proof} Without loss of generality, we
may suppose that~$\eps\in(0,1/2)$.
The proof is by contradiction.
Suppose that~$ E_\delta\cap B_{1-\eps} \ne\varnothing$.
Then, by Lemma~\ref{89=QW}, we have that~$ B_{1-\eps}\subseteq
E_\delta$. Moreover, if we set
$$ H := \big( B_{2}\setminus B_1\big)\cap \{x_n>0\},$$
we have that~$H\subseteq E_\delta^c$.
As a consequence, 
$$ \per(E_\delta,B_1)\ge L(B_{1-\eps}, H)\ge L(B_{1/2},H)\ge c,$$
for some~$c>0$. This is in contradiction with Lemma~\ref{89=IO}
and so it proves the desired result.
\end{proof}

With this, we are in the position of completing the proof
of Theorem~\ref{89=THM}:

\begin{proof}[Proof of Theorem~\ref{89=THM}]
We need to show that~$E_\delta\cap B_1=\varnothing$.
By contradiction, suppose not. Then there exists
\begin{equation}\label{TY67}
p\in E_\delta\cap B_1.\end{equation} 
By Corollary~\ref{89=QW-2},
we know that
\begin{equation}\label{IOP89}
{\mbox{$B_r\subset E_\delta^c$ if $r\in(0,1-\eps)$.}}\end{equation}
We enlarge~$r$ till $B_r$ hits~$\partial E_\delta$. That is, by~\eqref{TY67},
there exists~$\rho \in[1-\eps,1)$ such that~$B_{\rho}\subset E_\delta^c$
and there exists~$q\in (\partial B_{\rho})\cap (\partial E_\delta)$
(see Figure~\ref{HB1}).

\begin{figure}
    \centering
    \includegraphics[height=5.9cm]{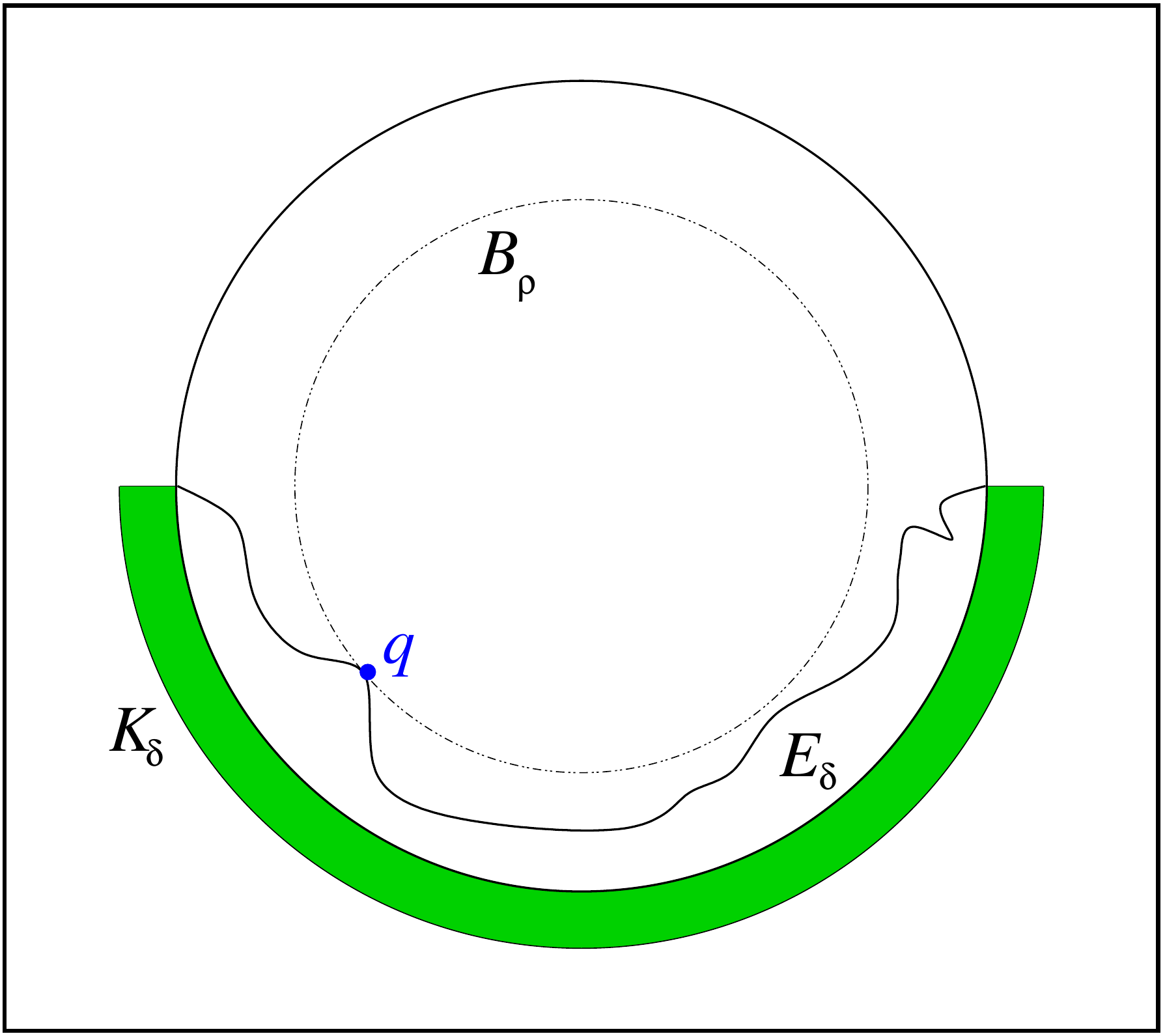}
    \includegraphics[height=5.9cm]{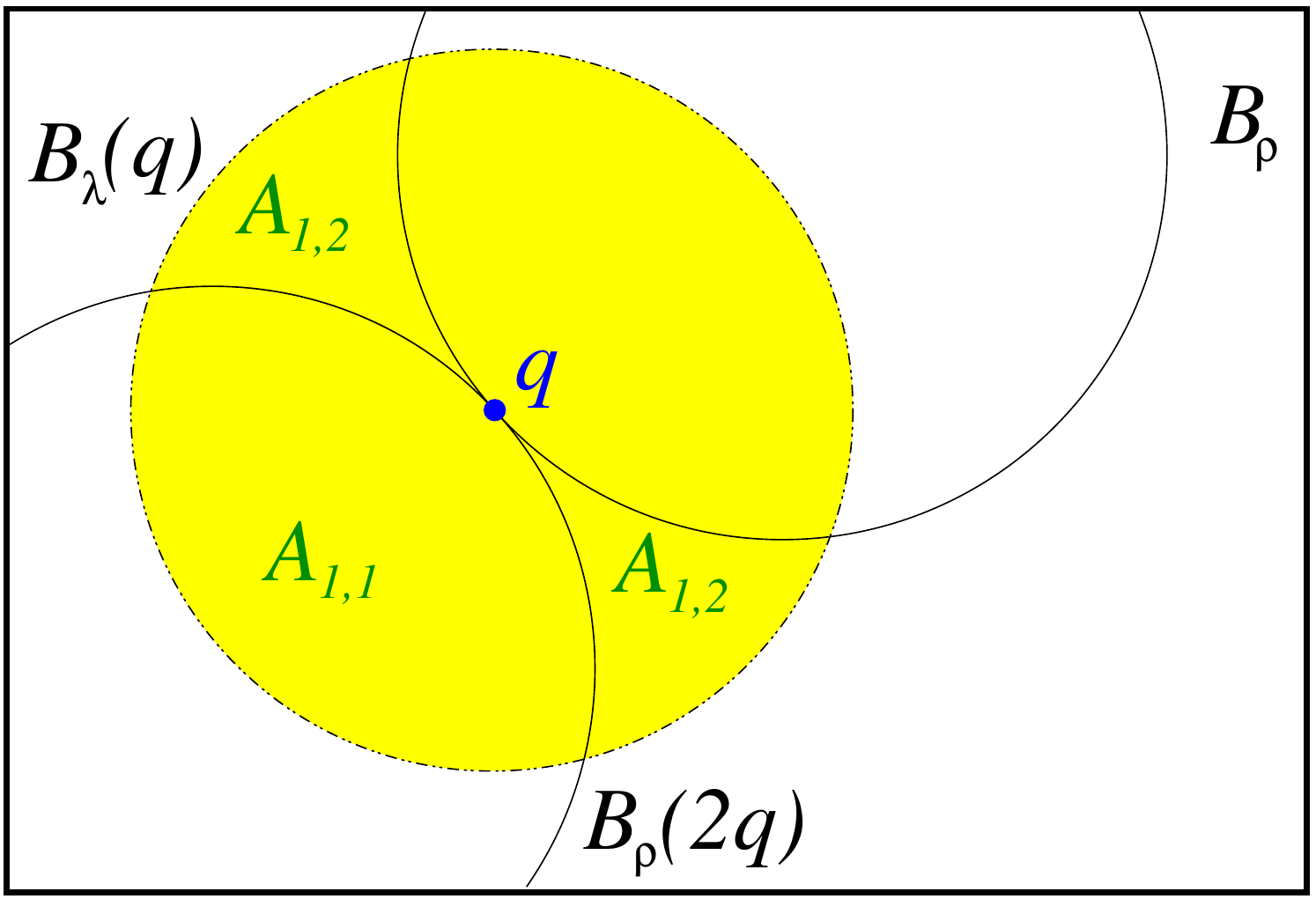}
    \caption{Touching the set~$E_\delta$ coming from the origin.}
    \label{HB1}
\end{figure}

Therefore, using the Euler-Lagrange equation in the viscosity sense
(see Theorem~5.1 in~\cite{CRS}), we conclude that
\begin{equation}\label{89=001}
\int_{\R^n} \frac{\chi_{E^c_\delta}(y)- \chi_{E_\delta}(y)}{|q-y|^{n+2s}}\,dy\le0.\end{equation}
By \eqref{IOP89}, we know that
$$ E_\delta\subseteq (B_1\setminus B_{\rho})\cup K_\delta
\subseteq B_{1+\delta}\setminus B_{\rho}$$
and so
\begin{equation}\label{89=002}
\int_{\R^n} \frac{\chi_{E_\delta}(y)-\chi_{E^c_\delta}(y)}{|q-y|^{n+2s}}
\le \int_{B_{1+\delta}\setminus B_{\rho}} \frac{dy}{|q-y|^{n+2s}}
-\int_{B_{\rho}} \frac{dy}{|q-y|^{n+2s}}.\end{equation}
In addition, if~$y\in B_{1/2}$, then~$|q-y|\le |q|+|y|<2$ and so
\begin{equation} \label{09ip}
\int_{B_{1/2}} \frac{dy}{|q-y|^{n+2s}} \ge \tilde c,\end{equation}
for some~$\tilde c>0$. 

Now we define~$\lambda:=(\eps+\delta)^\frac{1}{2(n+2s)}$. We
notice that~$\lambda$ is small if so are~$\eps$ and~$\delta$,
and so~$B_\lambda(q)\subset B_{1/2}^c$. Then, formula~\eqref{09ip}
gives that
$$ \int_{B_{\rho}} \frac{dy}{|q-y|^{n+2s}}\ge
\tilde c+ \int_{B_\lambda(q)\cap B_{\rho}} \frac{dy}{|q-y|^{n+2s}}.$$
This, \eqref{89=001} and~\eqref{89=002}
give that
\begin{equation}\label{89=003}
\int_{B_{1+\delta}\setminus B_{\rho}} \frac{dy}{|q-y|^{n+2s}}-
\int_{B_{\lambda}(q)\cap B_{\rho}} \frac{dy}{|q-y|^{n+2s}}
\ge \tilde c.
\end{equation}
Now we define
$$ A_1:=\big( B_{1+\delta}\setminus B_{\rho} \big)\cap B_{\lambda}(q)
\quad{\mbox{ and }} \quad A_2:=
\big( B_{1+\delta}\setminus B_{\rho} \big)\setminus B_{\lambda}(q).$$
We notice that
$$ \int_{A_2} \frac{dy}{|q-y|^{n+2s}}\le\frac{|A_2|}{\lambda^{n+2s}}
\le \frac{|B_{1+\delta}\setminus B_{\rho}|}{\lambda^{n+2s}}
\le \frac{C\,(\eps+\delta)}{\lambda^{n+2s}} = C\,\sqrt{\eps+\delta},$$
for some~$C>0$. Hence, \eqref{89=003}
becomes
\begin{equation}\label{89=004}
\int_{A_1} \frac{dy}{|q-y|^{n+2s}}-
\int_{B_{\lambda}(q)\cap B_{\rho}} \frac{dy}{|q-y|^{n+2s}}
\ge \frac{\tilde c}{2}.\end{equation}
Now we set
$$ A_{1,1}:=A_1\cap B_{\rho}(2q)
\quad{\mbox{ and }} \quad A_{1,2}:= A_1\setminus B_{\rho}(2q),$$
see again Figure~\ref{HB1}.
We remark that~$B_{\rho}(2q)$ is tangent to~$B_{\rho}$ at the point~$q$,
and~$A_{1,1}\subseteq B_{\lambda}(q)\cap B_{\rho}(2q)$.
Therefore, by symmetry
\begin{equation}\label{0oiGH}
\int_{A_{1,1}} \frac{dy}{|q-y|^{n+2s}}\le
\int_{B_{\lambda}(q)\cap B_{\rho}(2q)} \frac{dy}{|q-y|^{n+2s}}
=
\int_{B_{\lambda}(q)\cap B_{\rho}} \frac{dy}{|q-y|^{n+2s}}.\end{equation}
Now we observe that~$A_{1,2}$ is trapped between~$B_{\rho}$
and~$B_{\rho}(2q)$, and it lies in~$B_\lambda(q)$
therefore (see e.g. Lemma~3.1 in~\cite{nostro})   
$$ \int_{A_{1,2}} \frac{dy}{|q-y|^{n+2s}}\le
C\rho^{-2s} \lambda^{1-2s}\le C\lambda^{1-2s} =C
\,(\eps+\delta)^{\frac{1-2s}{2(n+2s)}},$$
up to renaming constants.

The latter estimate and~\eqref{0oiGH} give
$$ \int_{A_1} \frac{dy}{|q-y|^{n+2s}}\le
\int_{B_{\lambda}(q)\cap B_{\rho}} \frac{dy}{|q-y|^{n+2s}}+
C\,(\eps+\delta)^{\frac{1-2s}{2(n+2s)}}.$$
By inserting this information into~\eqref{89=004},
we obtain~$2C\,(\eps+\delta)^{\frac{1-2s}{2(n+2s)}}\ge\tilde c$,
which leads to a contradiction by choosing~$\eps$ small enough
(and thus~$\delta\le\delta_\eps$ small).
\end{proof}

\section{Stickiness to the sides of a box}\label{SEC:BOX}

In this section, we discuss the stickiness properties
to the sides of a box with high oscillatory external data
and we prove Theorem~\ref{CYL-THEOR}. 
To this
goal, 
we recall that the set~$J_M$ has been defined 
in~\eqref{jump} and~$E_M$ is the $s$-minimal set
in~$(-1,1)\times\R$
with datum outside~$(-1,1)\times\R$ equal to~$J_M$.

We first establish an easier version of
Theorem~\ref{CYL-THEOR}, in which the sticking size is proved
to be at least of the order of the oscillation
(then, a refined estimate will lead to the proof of
Theorem~\ref{CYL-THEOR}).

\begin{proposition}\label{CYL}
There exist~$M_o>0$, $c_o\in(0,1)$, depending on~$s$, such that
if~$M\ge M_o$ then
\begin{eqnarray}
&& [-1,1)\times [c_o M,M]\subseteq E_M^c \label{SL-011}
\\{\mbox{and }}&& (-1,1]\times [-M,-c_o M]\subseteq E_M.
\label{SL-012}
\end{eqnarray}
\end{proposition}

\begin{proof} We denote coordinates in~$\R^2$ by~$x=(x_1,x_2)$.
We take~$\eps_o>0$, to be chosen conveniently small
in the sequel. Let~$t\in[0,\eps_o^2]$.
We considers balls of radius~$\eps_o M$ with center lying on the
straight line~$\{ x_2=(1-t)M\}$. The idea of the proof
is to slide a ball of this type from left to right till
we touch~$\partial E_M$. We will show that the touching point
can only occur along the boundary~$\{x_1=1\}$. 
Hence, by varying~$t\in[0,\eps_o^2]$, we obtain that~$[-1,1)\times
[(1-\eps_o^2) M,M]$ is contained in~$E_M^c$. This would complete
the proof of~\eqref{SL-011}
(and the proof of~\eqref{SL-012} is similar).

The details of the proof of \eqref{SL-011} are the following.
We fix~$t\in[0,\eps_o^2]$.
If~$x_1<-M-2$, then
the ball~$B_{\eps_o M} ( x_1, (1-t)M)$ lies in~$(-\infty,-2)\times\R$,
and so its closure is contained in~$E_M^c$.
Hence, we consider~$\ell\ge -M-2$ such that~$
\overline{B_{\eps_o M} ( \ell, (1-t)M)}\subseteq E_M^c$
for any~$x_1<\ell$ and there exists~$q=(q_{1},q_{2})\in
(\partial E_M)\cap (\partial B_{\eps_o M} ( \ell, (1-t)M))$.
The proof of~\eqref{SL-011} is complete if we show that
\begin{equation}\label{SL-013}
q_{1} \ge 1.
\end{equation}
To prove this, we argue by contradiction.
If not, then~$q_1\in [-1,1)$, therefore,
by the Euler-Lagrange inequality (see
Theorem~5.1 in~\cite{CRS}), 
\begin{equation}\label{89=001=BIS}
\int_{\R^2} \frac{\chi_{E^c_M}(y)- \chi_{E_M}(y)}{|q-y|^{2+2s}}\,dy
\le0.\end{equation}
Now we denote by~$z:=(\ell, (1-t)M))$ the center
of the touching ball.
We also consider the extremal point of the touching
ball on the right, that we
denote by~$p:= z+(\eps_o M,0)$.
We claim that
\begin{equation}\label{Y7}
|q_2-p_2|\le 8\sqrt{\eps_o M}.\end{equation}
To prove this, we observe that,
by construction,
both~$q$ and~$p$ lie in~$[-1,1]\times\R$, hence~$|q_1|$, $|p_1|\le1$,
consequently
\begin{equation}\label{q1p12}
|q_1-p_1|\le2.\end{equation}
Also, both~$q$ and~$p$ lie on the boundary
of the touching ball, namely~$|q-z|=\eps_o M=|p-z|$, therefore
\begin{eqnarray*}
&& 0=|q-z|^2 - |p-z|^2
= |q|^2 -2q\cdot z -|p|^2 +2p\cdot z
= (q-p)\cdot (q+p-2z)
\\ &&\qquad =
(q_1-p_1)(q_1+p_1-2z_1)+
(q_2-p_2)(q_2+p_2-2z_2)
\\ &&\qquad=
(q_1-p_1)(q_1-p_1+2\eps_o M)+
(q_2-p_2)(q_2-p_2)
\\ &&\qquad\ge
-4(1+\eps_o M)+|q_2-p_2|^2.
\end{eqnarray*}
This establishes~\eqref{Y7}, provided that~$M$
is large enough (possibly in dependence of~$\eps_o$).

Now we consider the symmetric ball to the touching ball,
with respect to the touching point~$q$. That is, we define~$\bar z:=
z+2(q-z)$ and consider the ball~$B_{\eps_o M}(\bar z)$.
We remark that
\begin{equation}
{\mbox{$B_{\eps_o M}(z)$ and~$B_{\eps_o M}(\bar z)$
are tangent to each other at~$q$.}}\end{equation}
We also claim that
\begin{equation}\label{Y8}
B_{\eps_o M}(\bar z)\cap
\left\{ x_2> \bar z_2+2\eps_o^2 M\right\}
\subseteq \{ x_2>M\}.
\end{equation}
To prove this, we observe that
$$ -\eps_o^2 M -16\sqrt{\eps_o M}+2\eps_o^2 M
= \eps_o^2 M \left( 1 -\frac{16}{\eps_o^{3/2}\sqrt{M}}\right)>0,$$
if~$M$ is large enough.
Hence, recalling~\eqref{Y7},
\begin{eqnarray*}
&& \bar z_2+ 2\eps_o^2 M =
z_2+2(q_2-z_2) + 2\eps_o^2 M
= (1-t)M+2(q_2-p_2) + 2\eps_o^2 M\\
&&\qquad\ge (1-\eps_o^2)M -16\sqrt{\eps_o M}+ 2\eps_o^2 M> M.
\end{eqnarray*}
This proves~\eqref{Y8}.

\begin{figure}
    \centering
    \includegraphics[height=5.9cm]{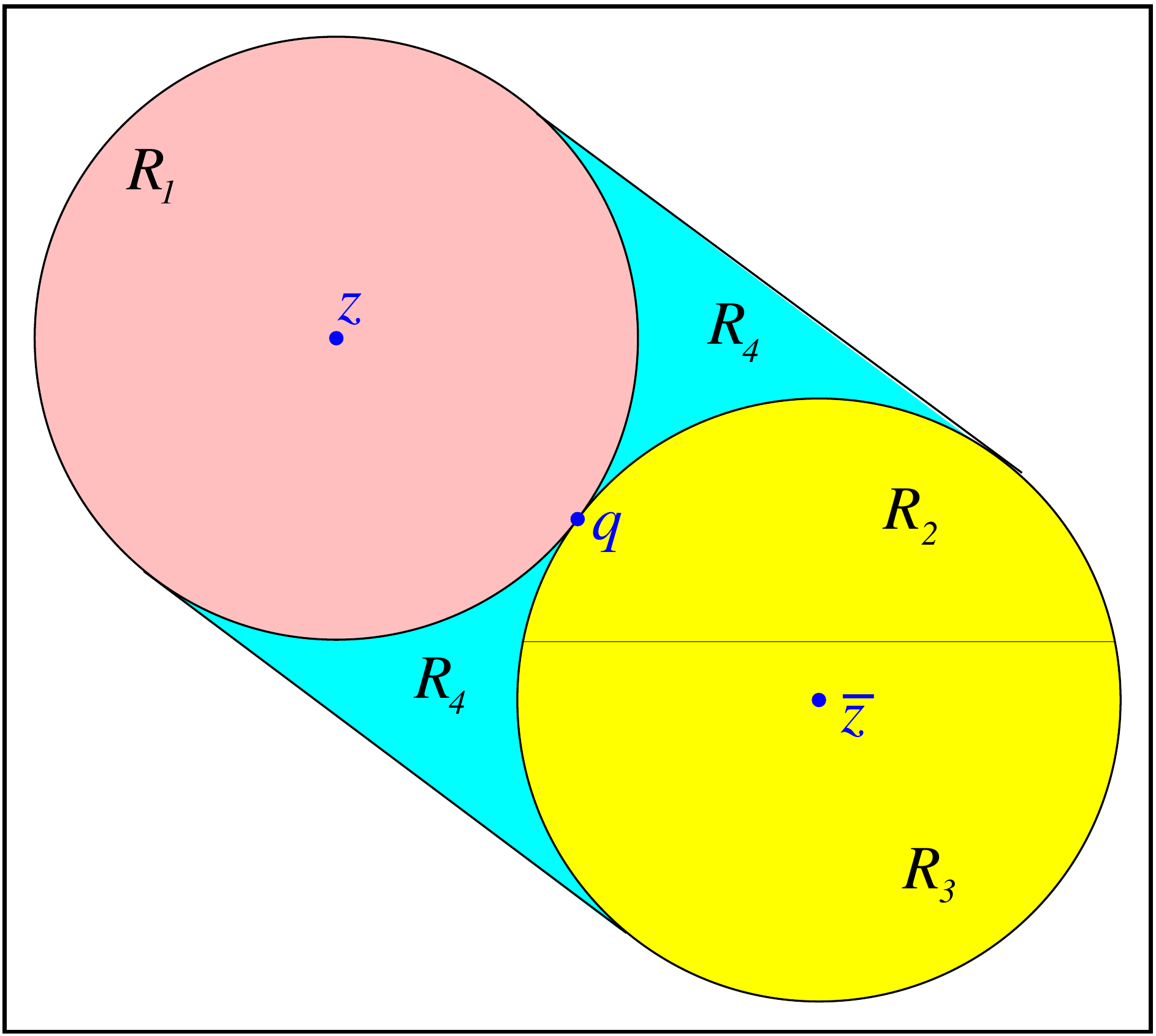}
    \caption{The partition of the plane needed for the proof
of Proposition~\ref{CYL}.}
    \label{REG}
\end{figure}

Now we decompose~$\R^2$ into five nonoverlapping
regions. Namely, we consider
\begin{eqnarray*}
&& R_1:=B_{\eps_o M}(z),\\
&& R_2:=B_{\eps_o M}(\bar z)\cap
\left\{ x_2> \bar z_2+ 2\eps_o^2 M\right\}\\
{\mbox{and }}&&
R_3:=
B_{\eps_o M}(\bar z)\cap
\left\{ x_2\le \bar z_2+ 2\eps_o^2 M\right\}.\end{eqnarray*}
Then we define~$D:=B_{\eps_o M}(z)\cup B_{\eps_o M}(\bar z)$,
$K$ the convex hull
of~$D$ and~$R_4:=K\setminus D$. Finally,
we set~$R_5:=\R^2 \setminus K$ and consider the partition of~$\R^2$
given by the regions~$R_1,\dots,R_5$.

We consider the contribution to the integral
in~\eqref{89=001=BIS}
given by these regions. The regions~$R_1$, $R_2$ and~$R_3$
will be considered together: namely, $R_1\subseteq E^c_M$,
and, by~\eqref{Y8}, also~$R_2\subseteq E^c_M$. Therefore,
by symmetry
\begin{equation}\label{Y9-K}
\int_{R_1\cup R_2\cup R_3}
\frac{\chi_{E^c_M}(y)- \chi_{E_M}(y)}{|q-y|^{2+2s}}\,dy\ge
\int_{R_1\cup R_2}
\frac{dy}{|q-y|^{2+2s}}
-
\int_{R_3}
\frac{dy}{|q-y|^{2+2s}} =
\int_{R_2}
\frac{2\,dy}{|q-y|^{2+2s}}
.\end{equation}
Now, for~$y\in R_2$, we consider the change of variable~$\tilde y=T(y)
:=(y-q)/(\eps_o M)$. We have that
\begin{equation}\label{INC0}\begin{split}
T(R_2) \,&=\,
B_{1}\left( \frac{q-z}{\eps_o M}\right)
\cap
\left\{ \tilde y_2> \frac{q_2-z_2}{\eps_o M}+2\eps_o
\right\} \\
&\supseteq\,
B_{1}\left( \frac{q-z}{\eps_o M}\right)
\cap
\left\{ \tilde y_2> 3\eps_o\right\},
\end{split}\end{equation}
where we used again~\eqref{Y7}
in the last inclusion (provided that~$\eps_o$ is sufficiently
small and~$M$ is sufficiently large, possibly in dependence
of~$\eps_o$).
 
Now we claim that
\begin{equation}\label{INC-1}
B_{\eps_o}(5\eps_o,5\eps_o)
\subseteq B_{1}\left( \frac{q-z}{\eps_o M}\right)
\cap
\left\{ \tilde y_2> 3\eps_o \right\}
.\end{equation}
To prove this, it is enough to take~$\eta\in B_{\eps_o}$
and show that
\begin{equation} \label{INC-2}
(5\eps_o,5\eps_o)+\eta\in 
B_1\left( \frac{q-z}{\eps_o M}\right).\end{equation}
For this, we use~\eqref{Y7} to observe that
\begin{equation}\label{PHG} 
|q_1-z_1|^2=|q-z|^2 -|q_2-z_2|^2 \ge 
(\eps_o M)^2-64\eps_o M.\end{equation}
Moreover, by~\eqref{q1p12},
$$ q_1-z_1=q_1-p_1+p_1-z_1=q_1-p_1+\eps_o M\ge \eps_o M-2>0.$$
Hence, \eqref{PHG}
gives that
$$ \frac{q_1-z_1}{\eps_o M} =\frac{|q_1-z_1|}{\eps_o M}
\ge \sqrt{1 -\frac{64}{\eps_o M}}
\ge 1 -\frac{128}{\eps_o M},$$
if~$M$ is large enough.
In particular
$$ \frac{q_1-z_1}{\eps_o M} - 5\eps_o -\eta_1
\ge 1 -\frac{128}{\eps_o M} -6\eps_o
\ge\frac12 -\frac{128}{\eps_o M}>0,$$
provided that~$\eps_o$ is small enough and~$M$ large enough
(possibly depending on~$\eps_o$). Therefore
$$ \left|
\frac{q_1-z_1}{\eps_o M} - 5\eps_o -\eta_1\right|
=
\frac{q_1-z_1}{\eps_o M} - 5\eps_o -\eta_1 \le
\frac{|q_1-z_1|}{\eps_o M} - 4\eps_o \le 1-4\eps_o.$$
In addition, by~\eqref{Y7},
$$
\left|
\frac{q_2-z_2}{\eps_o M} - 5\eps_o -\eta_2\right|\le
\frac{|q_2-z_2|}{\eps_o M} +6\eps_o\le 7\eps_o.$$
Therefore
\begin{eqnarray*}
&& \left|
\frac{q-z}{\eps_o M} - (5\eps_o,5\eps_o) -\eta\right|^2 \le \left(
1-4\eps_o\right)^2
+ \left( 7\eps_o\right)^2 \\
&&\qquad= 1 -8\eps_o + 16\eps_o^2+49\eps_o^2 
< 1 
\end{eqnarray*}
if~$\eps_o$ is small enough.
This establishes~\eqref{INC-2} and therefore~\eqref{INC-1}.

{F}rom~\eqref{INC0} and~\eqref{INC-1}, we see that
$$ T(R_2) \supseteq B_{\eps_o}(5\eps_o,5\eps_o)$$
and then
\begin{equation}\label{INC-12} \int_{R_2} \frac{dy}{|q-y|^{2+2s}}=
\frac{1}{(\eps_o M)^{2s}}\int_{T(R_2)}
\frac{d\tilde y}{|\tilde y|^{2+2s}}\ge
\frac{1}{(\eps_o M)^{2s}}\int_{
B_{\eps_o}(5\eps_o,5\eps_o) }
\frac{d\tilde y}{|\tilde y|^{2+2s}}.\end{equation}
Now, if~$\tilde y\in B_{\eps_o}(5\eps_o,5\eps_o)$
then~$|\tilde y|\le \eps_o+|(5\eps_o,5\eps_o)|\le 10\eps_o$,
and then~\eqref{INC-12} gives that
$$ \int_{R_2} \frac{dy}{|q-y|^{2+2s}}\ge
\frac{\tilde c}{\eps_o^{4s} M^{2s}},$$
for some~$\tilde c>0$.
By inserting this into~\eqref{Y9-K}
we conclude that
\begin{equation}\label{Y10-K}
\int_{R_1\cup R_2\cup R_3}
\frac{\chi_{E^c_M}(y)- \chi_{E_M}(y)}{|q-y|^{2+2s}}\,dy\ge
\frac{\tilde c}{\eps_o^{4s} M^{2s}}.\end{equation}
Moreover (see e.g. Lemma~3.1 in~\cite{nostro}
with~$R:=\eps_o M$ and~$\lambda:=1$),
we see that
\begin{equation}\label{Y10-K-2}
\left|\int_{R_4}
\frac{\chi_{E^c_M}(y)- \chi_{E_M}(y)}{|q-y|^{2+2s}}\,dy\right|\le
\int_{R_4}
\frac{dy}{|q-y|^{2+2s}}\le
\frac{C}{\eps_o^{2s} M^{2s}},\end{equation}
for some~$C>0$.
Furthermore, the distance from~$q$ to any point of~$R_5$
is at least~$\eps_o M$, therefore~$R_5\subseteq \R^2\setminus B_{\eps_o M}(q)$,
and
$$ \left|\int_{R_5}
\frac{\chi_{E^c_M}(y)- \chi_{E_M}(y)}{|q-y|^{2+2s}}\,dy\right|\le
\int_{\R^2\setminus B_{\eps_o M}(q)}
\frac{dy}{|q-y|^{2+2s}}=
\frac{\tilde C}{\eps_o^{2s} M^{2s}},$$
for some~$\tilde C>0$.

By combining the latter estimate with~\eqref{Y10-K} and~\eqref{Y10-K-2},
we obtain that
$$ \int_{\R^2}
\frac{\chi_{E^c_M}(y)- \chi_{E_M}(y)}{|q-y|^{2+2s}}\,dy\ge
\frac{1}{\eps_o^{2s} M^{2s}}\left( \frac{\tilde c}{\eps_o^{2s}}-
C-\tilde C\right)>0,$$
provided that~$\eps_o$ is suitably small.
This estimate is in contradiction with~\eqref{89=001=BIS}
and therefore the proof of~\eqref{SL-013} is complete.
\end{proof}

The result in Proposition~\ref{CYL} can be refined.
Namely, not only the optimal set~$E_M$ in Proposition~\ref{CYL}
sticks for an amount of order~$M$ is a box of side~$M$,
but it sticks up to an order of~$M^{\frac{1+2s}{2+2s}}$
from the origin, as the following
Proposition~\ref{CYL-RAFF} points out.
As a matter of fact, the exponent~${\frac{1+2s}{2+2s}}$
is sharp, as we will prove in the subsequent
Proposition~\ref{PR-3.2}.

\begin{proposition}\label{CYL-RAFF}
There exist~$M_o$, $C_o>0$, depending on~$s$, such that
if~$M\ge M_o$ then
\begin{eqnarray}
&& [-1,1)\times [C_o M^{\frac{1+2s}{2+2s}},\,M]\subseteq E_M^c \label{XSL-011}
\\{\mbox{and }}&& (-1,1]\times [-M,\,-C_o M^{\frac{1+2s}{2+2s}}]\subseteq E_M.
\label{XSL-012}
\end{eqnarray}
\end{proposition}

\begin{proof} 
We let~$\beta:={\frac{1+2s}{2+2s}}$.
We focus on the proof of~\eqref{XSL-011}
(the proof of~\eqref{XSL-012} is similar).
The proof is based on a sliding method: we will
consider a suitable surface and we slide it from left to
right in order to ``clean'' the portion
of space~$[-1,1)\times [C_o M^{\beta},M]$.
As a matter of fact, by Proposition~\ref{CYL},
it is enough to take care of~$[-1,1)\times [C_o M^{\beta},\,c_o M]$,
with~$c_o\in(0,1)$.

For this we fix any
\begin{equation}\label{q2whee-pre}
t\in [C_o M^\beta,\,c_o M]\end{equation}
and, for any~$\mu\in\R$,
we define
$$ S_\mu := B_{M^{2\beta}} (\mu-M^{2\beta},t)
\cap \{|x_2-t|<4M^\beta\}.$$
Notice that if~$\mu<-1$ then
$$S_\mu\subseteq (-\infty,-1)\times \{|x_2-t|<4M^\beta\}
\subseteq E_M^c.$$
Therefore we increase~$\mu$ till~$S_\mu$ touches~$\partial E_M$.
This value of $\mu$ will be fixed from now on.
We observe that Proposition~\ref{CYL-RAFF} is proved if we show that
$\mu=1$. So we assume by contradiction that~$\mu\in[-1,1)$.
By construction, we have that 
\begin{equation}\label{9uIIOk}
S_\mu\subseteq E_M^c
\end{equation}
and there exists~$q\in (\partial S_\mu)\cap(\partial E_M)$,
with~$q_1\in [-1,1)$.
We claim that
\begin{equation}\label{q2whee}
|q_2-t|\le 2 M^\beta.
\end{equation}
To prove this, we observe that~$|q_1-\mu+M^{2\beta}|\ge
M^{2\beta} -|q_1|-|\mu|\ge M^{2\beta}-2$. Moreover,
$q\in \partial S_\mu\subseteq\overline{
B_{M^{2\beta}} (\mu-M^{2\beta},t)}$, therefore
\begin{eqnarray*}
M^{4\beta} \ge \big|q-(\mu-M^{2\beta},t)\big|^2
\ge (M^{2\beta}-2)^2 + |q_2-t|^2
\ge M^{4\beta} -4M^{2\beta}+ |q_2-t|^2,
\end{eqnarray*}
from which we obtain~\eqref{q2whee}.

Now, using the Euler-Lagrange equation in the viscosity sense
(see Theorem~5.1 in~\cite{CRS}), we see that
\begin{equation}\label{89=001=E:L}
\int_{\R^n} \frac{\chi_{E^c_M}(y)- \chi_{E_M}(y)}{|q-y|^{2+2s}}\,dy\le0.
\end{equation}
We first estimate the contribution to the integral
above coming from~$B_{M^\beta}(q)$.
For this, we consider the symmetric point
of~$z:=(\mu-M^{2\beta},t)$ with respect to~$q$,
namely we set~$z':= z+2(q-z)$. We also consider the
ball~$B':=B_{M^{2\beta}}(z')$.
Notice that~$B_{M^{2\beta}}(z)$ and~$B'$ are tangent one to the other
at~$q$. We define~$A_1:= B_{M^{2\beta}}(z)\cap B_{M^\beta}(q)$,
$A_2:= B'\cap B_{M^\beta}(q)$ and~$A_3:= B_{M^\beta}(q)\setminus
(A_1\cup A_2)$.
Hence (see e.g.
Lemma~3.1 in~\cite{nostro},
used here with~$R:=M^{2\beta}$ and~$\lambda:=M^{-\beta}$),
we obtain that
\begin{equation}\label{89=001=E:L-1}
\left|\int_{A_3} \frac{\chi_{E^c_M}(y)- \chi_{E_M}(y)}{|q-y|^{2+2s}}\,dy
\right|\le
\int_{A_3} \frac{dy}{|q-y|^{2+2s}}\le C M^{-\beta(1+2s)}.
\end{equation}
Now we observe that
\begin{equation}\label{A1wh}
A_1\subseteq E_M^c.
\end{equation}
For this, let~$y\in A_1$.
Then~$|y-q|< M^\beta$. Therefore, recalling~\eqref{q2whee},
$$ |y_2-t|\le |y_2-q_2|+|q_2-t| < M^\beta + 2 M^\beta < 4M^\beta.$$
Since also~$y\in B_{M^{2\beta}}(z)$, we obtain that~$y\in S_\mu$.
Then we use~\eqref{9uIIOk}
and we finish the proof of~\eqref{A1wh}.

Then, we use~\eqref{A1wh} and a symmetry argument to see that
$$ \int_{A_1\cup A_2}\frac{\chi_{E^c_M}(y)- \chi_{E_M}(y)}{|q-y|^{2+2s}}\,dy
=
\int_{A_1} \frac{dy}{|q-y|^{2+2s}}+
\int_{A_2} \frac{\chi_{E^c_M}(y)- \chi_{E_M}(y)}{|q-y|^{2+2s}}\,dy
\ge 0.$$
This and~\eqref{89=001=E:L-1}
give that
$$ \int_{B_{M^\beta}(q)}
\frac{\chi_{E^c_M}(y)- \chi_{E_M}(y)}{|q-y|^{2+2s}}\,dy\ge- C M^{-\beta(1+2s)}.$$
Consequently,
by~\eqref{89=001=E:L},
\begin{equation}\label{89=001=E:L-2}
\int_{\R^n\setminus B_{M^\beta}(q)}
\frac{\chi_{E^c_M}(y)- \chi_{E_M}(y)}{|q-y|^{2+2s}}\,dy\le
-\int_{B_{M^\beta}(q)}
\frac{\chi_{E^c_M}(y)- \chi_{E_M}(y)}{|q-y|^{2+2s}}\,dy\le
C M^{-\beta(1+2s)}.
\end{equation}
Now we observe that
\begin{equation}\label{POL}
\{|x_1-q_1|\le 16\}\setminus B_{M^\beta}(q)
\subseteq \{|x_1-q_1|\le 16\}\times
\left\{|x_2-q_2|\ge \frac{M^\beta}{2}\right\}
.\end{equation}
To prove this, let~$y\in \{|x_1-q_1|\le 16\}\setminus B_{M^\beta}(q)$
and suppose, by contradiction, that~$|y_2-q_2|<M^\beta/2$.
Then
$$ |y-q|^2 \le 16^2+ \frac{M^{2\beta}}{4}< M^{2\beta}.$$
This would say that~$y\in B_{M^\beta}(q)$,
which is a contradiction, and so~\eqref{POL} is proved.

By~\eqref{POL}, we obtain that
\begin{eqnarray*}
&& \left|
\int_{\{|x_1-q_1|\le 16\}\setminus B_{M^\beta}(q)}
\frac{\chi_{E^c_M}(y)- \chi_{E_M}(y)}{|q-y|^{2+2s}}\,dy\right|
\le \int_{ \{|x_1-q_1|\le 16\}\times
\left\{|x_2-q_2|\ge \frac{M^\beta}{2}\right\} }
\frac{dy}{ |q-y|^{2+2s} } \\
&&\qquad\le \int_{q_1-16}^{q_1+16}
\left( \int_{ \{ |q_2-y_2|\ge M^\beta/2\}}
\frac{dy_2}{ |q_2-y_2|^{2+2s} }
\right)
\,dy_1 = CM^{-\beta(1+2s)},
\end{eqnarray*}
for some~$C>0$.

{F}rom this and~\eqref{89=001=E:L-2}, we obtain that
\begin{equation}\label{89=001=E:L-3}
\int_{ \{|x_1-q_1|> 16\} \setminus B_{M^\beta}(q)}
\frac{\chi_{E^c_M}(y)- \chi_{E_M}(y)}{|q-y|^{2+2s}}\,dy\le
C M^{-\beta(1+2s)},
\end{equation}
up to renaming~$C>0$.

Now we define~$H_1:=\{ x_1-q_1<-16\}$
and~$H_2:=\{x_1-q_1>16\}$.
Notice that~$H_1\subseteq \{ x_1< -15\}$ and~$H_2\subseteq \{x_1>15\}$.
Therefore~$H_1\cap \{x_2>-M\}\subseteq E^c_M$,
$H_1\cap \{x_2<-M\}\subseteq E_M$,
$H_2\cap \{x_2>M\}\subseteq E^c_M$
and~$H_2\cap \{x_2<M\}\subseteq E_M$.

Then, we define, for any~$i\in\{1,2\}$,
\begin{eqnarray*}
&& H_{i,1}:=H_i\cap \{ x_2>2q_2+M\},\\
&& H_{i,2}:=H_i\cap \{ x_2\in (M,2q_2+M]\},\\
&& H_{i,3}:=H_i\cap \{ x_2\in [-M,M]\},\\
&& H_{i,4}:=H_i\cap \{ x_2<-M\},\end{eqnarray*}
see Figure~\ref{HBQA-HSET}.

\begin{figure}
    \centering
    \includegraphics[height=7.9cm]{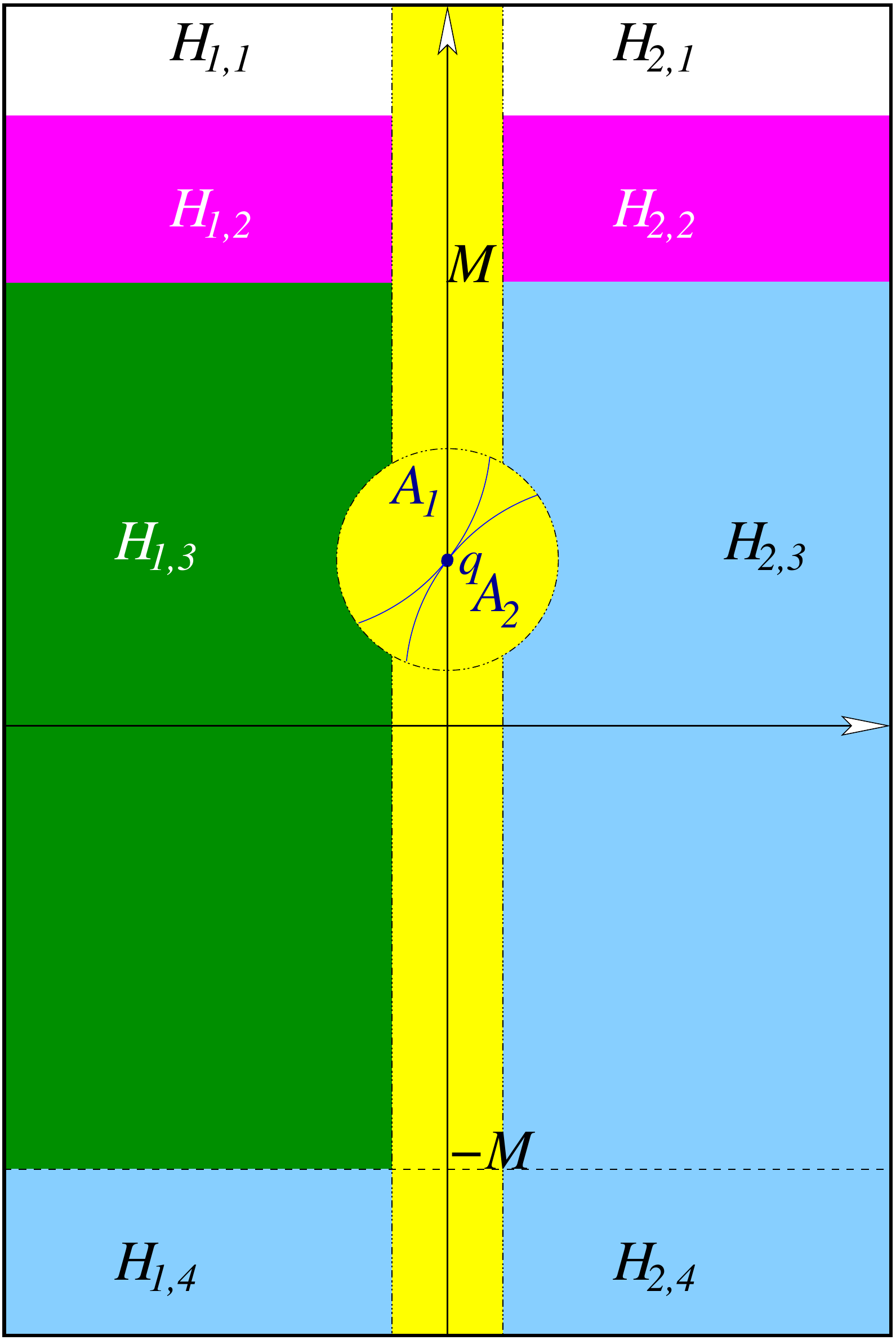}
    \caption{The geometry involved in the proof of Proposition~\ref{CYL-RAFF}.}
    \label{HBQA-HSET}
\end{figure}
 
By construction, $H_{i,1}\subseteq E^c_M$
and~$H_{i,4}\subseteq E_M$,
therefore, by up/down symmetry,
\begin{equation}\label{89=001=E:L-X1}
\int_{ (H_{1,1}\cup H_{1,4})\setminus B_{M^\beta}(q)}
\frac{\chi_{E^c_M}(y)- \chi_{E_M}(y)}{|q-y|^{2+2s}}\,dy =
0=
\int_{ (H_{2,1}\cup H_{2,4})\setminus B_{M^\beta}(q)}
\frac{\chi_{E^c_M}(y)- \chi_{E_M}(y)}{|q-y|^{2+2s}}\,dy
.\end{equation}
Moreover, $H_{1,3}\subseteq E^c_M$
and~$H_{2,3}\subseteq E_M$,
therefore, by left/right symmetry,
\begin{equation}\label{89=001=E:L-X2}
\int_{ (H_{1,3}\cup H_{2,3})\setminus B_{M^\beta}(q)}
\frac{\chi_{E^c_M}(y)- \chi_{E_M}(y)}{|q-y|^{2+2s}}\,dy =0.\end{equation}
Finally, we point out that~$H_{1,2}\cup H_{2,2}\subseteq E_M^c$
and (recalling~\eqref{q2whee} and~\eqref{q2whee-pre}) that
$$ B_{M^\beta}(q)\subseteq\{ x_2<q_2+M^\beta\}
\subseteq\{ x_2< t+3M^\beta\} \subseteq\{ x_2<M\}.$$
Therefore
\begin{equation}
\label{POljU}\begin{split}
& \int_{ (H_{1,2}\cup H_{2,2})\setminus B_{M^\beta}(q)}
\frac{\chi_{E^c_M}(y)- \chi_{E_M}(y)}{|q-y|^{2+2s}}\,dy
= \int_{ H_{1,2}\cup H_{2,2}} \frac{dy}{|q-y|^{2+2s}} \\
&\qquad\ge
\int_{ \{y_1-q_1\in (16,16+M),\; y_2\in (M,2q_2+M)} \frac{dy}{|q-y|^{2+2s}}
.\end{split}\end{equation}
Now we observe that if~$y_1-q_1\in (16,16+M)$ and~$y_2\in (M,2q_2+M)$,
then~$|q-y|\le C M$, for some~$C>0$. Then~\eqref{POljU} implies that
$$ \int_{ (H_{1,2}\cup H_{2,2})\setminus B_{M^\beta}(q)}
\frac{\chi_{E^c_M}(y)- \chi_{E_M}(y)}{|q-y|^{2+2s}}\,dy
\ge c q_2 M^{-1-2s} ,$$
for some~$c>0$.
As a consequence of~\eqref{q2whee}
and~\eqref{q2whee-pre}, we also know that~$q_2\ge t- 2 M^\beta
\ge (C_o-2)M^\beta\ge C_o M^\beta/2$, if~$C_o$ is taken suitably large.
Hence we obtain
\begin{equation}\label{po09uJ}
\int_{ (H_{1,2}\cup H_{2,2})\setminus B_{M^\beta}(q)}
\frac{\chi_{E^c_M}(y)- \chi_{E_M}(y)}{|q-y|^{2+2s}}\,dy \ge c C_o
M^{\beta-1-2s},\end{equation}
up to renaming~$c>0$.
Now we observe that
$$ \beta-1-2s = \frac{(1+2s) (1-2-2s)}{2+2s} = -\beta(1+2s),$$
so we can write~\eqref{po09uJ} as
\begin{equation*}
\int_{ (H_{1,2}\cup H_{2,2})\setminus B_{M^\beta}(q)}
\frac{\chi_{E^c_M}(y)- \chi_{E_M}(y)}{|q-y|^{2+2s}}\,dy \ge c C_o
M^{-\beta(1+2s)}.\end{equation*}
This, together with~\eqref{89=001=E:L-X1}
and~\eqref{89=001=E:L-X2}, gives that
$$ \int_{ \{|x_1-q_1|> 16\} \setminus B_{M^\beta}(q)}
\frac{\chi_{E^c_M}(y)- \chi_{E_M}(y)}{|q-y|^{2+2s}}\,dy
\ge c C_o M^{-\beta(1+2s)}.$$
By comparing this inequality with~\eqref{89=001=E:L-3}, we obtain that
$$ c C_o M^{-\beta(1+2s)}
\le C M^{-\beta(1+2s)},$$
which is a contradiction if~$C_o$ is large enough.
This completes the proof of Proposition~\ref{CYL-RAFF}.
\end{proof}

As a counterpart of Proposition~\ref{CYL-RAFF}, we show that
the stickiness to the boundary of the domain does not
get too close to the origin, as next result points out:

\begin{proposition}\label{PR-3.2}
In the setting of Proposition~\ref{CYL-RAFF}, suppose that
\begin{equation}\label{INCL-1}
[-1,1)\times [b M^{\frac{1+2s}{2+2s}},\,M]\subseteq E_M^c,\end{equation}
with~$p=(1,b M^{\frac{1+2s}{2+2s}})\in\partial E_M$,
for some~$b\ge0$.
Then~$b\ge C_o$,
for some~$C_o>0$, only depending on~$s$,
provided that~$M$ is large enough.
\end{proposition}

\begin{proof} For short, we set~$\beta:={\frac{1+2s}{2+2s}}$.
We remark that
\begin{equation}\label{8u-BETA}
1-\frac{\beta}{1+2s} = 1-\frac{1}{2+2s}=\beta.
\end{equation}
We argue by contradiction, supposing that
\begin{equation}\label{78UI-CON}
b\le C_o 
\end{equation}
for some~$C_o\in(0,1)$ that we can take conveniently small in the sequel.
By Lemma~\ref{SYMM-LEMMA} (used here with~$T(x):=-x$),
we have that~$E_M$ is odd with respect to the origin.
This and~\eqref{INCL-1} give that
\begin{equation}\label{INCL-2}
(-1,1]\times [-M,\,-b M^{\beta}]\subseteq E_M.\end{equation}
Now we let~$L:= M-b M^{\beta}$ and
we consider the cube~$Q$ of side~$2L$ that has the point~$p$
on its left side, namely
$$ Q:= (1,1+2L)\times (M-2L,M).$$
Notice that
\begin{equation}\label{PO09GH}
Q\subseteq E_M, \end{equation}
by the boundary datum of the problem.
We also take the symmetric reflection of~$Q$ with
respect to~$\{ x_1=1\}$, that is we set
$$ Q':= (1-2L,1)\times (M-2L,M).$$
We also set
$$ G:= (-1,1)\times (-3b M^{\beta}-2,\,-b M^{\beta}-1).$$
We claim that
\begin{equation}\label{INCL-3}
G\subseteq Q'.\end{equation}
Indeed, if~$x_1\in (-1,1)$
and~$x_2\in (-3b M^{\beta}-2,\,-b M^{\beta}-1)$,
then
$$ 1-2L =1-2M+2b M^{\beta} \le 1-2M+2M^{\beta}< -1<x_1,$$
since~$M$ is large. Also
$$ M-2L = -M+2b M^{\beta}
< -3b M^{\beta}-2 < x_2,$$
using again that~$M$ is large.
Accordingly, $x_1\in (1-2L,1)$ and~$x_2\in(M-2L,M)$,
which proves~\eqref{INCL-3}.

Now we claim that
\begin{equation}\label{INCL-4}
G\subseteq (-1,1]\times [-M,\,-b M^{\beta}].\end{equation}
Indeed, if~$x_2\in
(-3b M^{\beta}-2,\,-b M^{\beta}-1)$, then
$$ -M < -3 M^{\beta}-2
\le -3b M^{\beta}-2<x_2,$$
for large~$M$, and so~$x_2\in[-M,\,-b M^{\beta}]$,
which proves~\eqref{INCL-4}.

{F}rom~\eqref{INCL-2}, \eqref{INCL-3} and~\eqref{INCL-4}, we obtain that
\begin{equation*}
G\subseteq Q'\cap E_M.\end{equation*}
Using this and~\eqref{PO09GH}, by a symmetry argument we conclude that
\begin{equation}\label{INCL-7}
\int_{Q\cup Q'} \frac{\chi_{E_M}(y)- \chi_{E_M^c}(y)}{|p-y|^{2+2s}}\,dy
\ge \int_G  \frac{dy}{|p-y|^{2+2s}}.
\end{equation}
Now we recall that~$p=(p_1,p_2)=(1,b M^{\beta})$ and we
observe that if~$y\in G$ then
\begin{eqnarray*}
&& |p_2-y_2| = |b M^{\beta}-y_2|\ge
|y_2|-b M^{\beta}
\\&&\qquad
\ge b M^{\beta}+1-b M^{\beta}=1\ge 
\frac{|p_1|+|y_1|}{2} \ge \frac{|p_1-y_1|}{2}.
\end{eqnarray*}
Hence, $|p-y|\le C |p_2-y_2|$, for some~$C>0$ and thus~\eqref{INCL-7}
and the substitution~$t:=p_2-y_2$ give
\begin{equation}\label{INCL-77}
\begin{split}
& \int_{Q\cup Q'} \frac{\chi_{E_M}(y)- \chi_{E_M^c}(y)}{|p-y|^{2+2s}}\,dy
\ge C \int_G  \frac{dy}{|p_2-y_2|^{2+2s}}
\\ &\qquad= C\int_{-3b M^{\beta}-2}^{-b M^{\beta}-1}
\frac{dy_2}{|p_2-y_2|^{2+2s}}
= C \int_{2b M^{\beta}+1}^{4b M^{\beta}+2}
\frac{dt}{t^{2+2s}} =\frac{C}{(2b M^{\beta}+1)^{1+2s}},
\end{split}
\end{equation}
up to renaming~$C$.

\begin{figure}
    \centering
    \includegraphics[height=5.9cm]{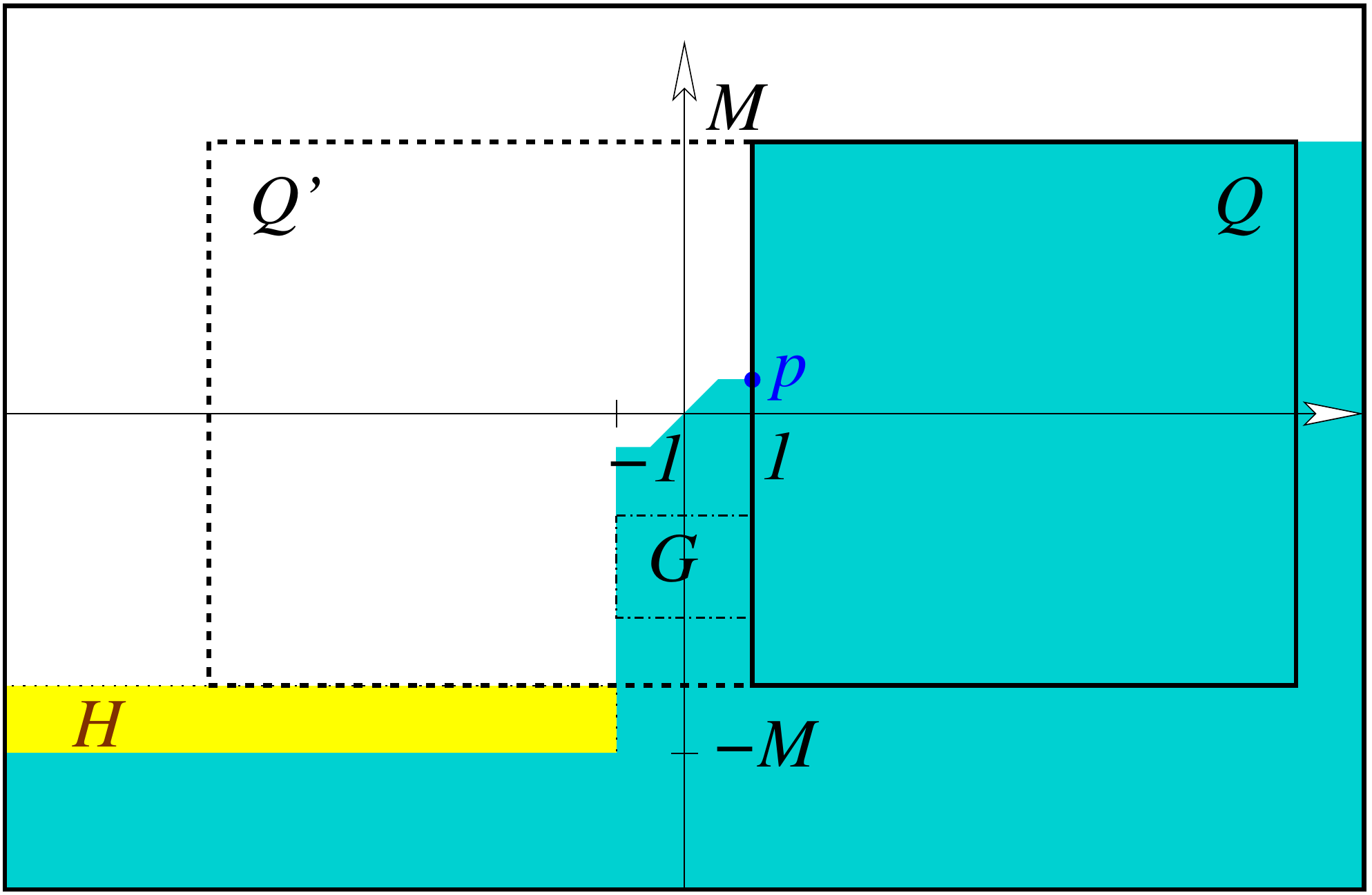}
    \caption{The geometry involved in the proof of Proposition~\ref{PR-3.2}.}
    \label{HBQA}
\end{figure}

Now we define
$$ H:=(-\infty,-1)\times (-M,M-2L),$$
see Figure~\ref{HBQA}. By construction, $H\subseteq E_M^c$.
We notice that the portion on the right of~$Q$ all belongs to~$E_M$,
while the portion on the left of~$Q'$ all belongs to~$E_M^c$,
that is
\begin{eqnarray*}
&&(-\infty,1-2L)\times(M-2L,M)\subseteq E_M^c\\
{\mbox{and }}&&(1+2L,+\infty)\times(M-2L,M)\subseteq E_M.
\end{eqnarray*}
Therefore, by symmetry, these contributions cancel and we have
\begin{equation}\label{INCL-889}
\int_{\R^2\setminus(Q\cup Q')}
\frac{\chi_{E_M}(y)- \chi_{E_M^c}(y)}{|p-y|^{2+2s}}\,dy
=
\int_{\{x_2>M\}\cup\{x_2<M-2L\}}
\frac{\chi_{E_M}(y)- \chi_{E_M^c}(y)}{|p-y|^{2+2s}}\,dy.
\end{equation}
Now we observe that~$\{x_2>M\}\subseteq E_M^c$
and~$\{x_2<M-2L\}\setminus H\subseteq E_M$, therefore, by symmetry,
\begin{equation}\label{INCL-889-B}
\int_{\{x_2>M\}\cup\{x_2<M-2L\}}
\frac{\chi_{E_M}(y)- \chi_{E_M^c}(y)}{|p-y|^{2+2s}}\,dy
=- 2 \int_{H}
\frac{dy}{|p-y|^{2+2s}}.\end{equation}
Now we observe that if~$y\in H$ then~$|y_2|\ge 2L-M$
and so
$$|y_2-p_2|\ge 2L-M-bM^\beta=M-3bM^\beta\ge M-3M^\beta\ge\frac{M}{2}$$
if~$M$ is large enough. Therefore
\begin{eqnarray*}
&& \int_{H} \frac{dy}{|p-y|^{2+2s}}\le C\int_{-\infty}^{1}
\left(\int_{-M}^{M-2L} \frac{dy_2}{\big( |p_1-y_1|^2 + M^2\big)^{\frac{2+2s}{2}}}
\right)\,dy_1 \\
&&\qquad=
C\,(M-L)\,\int_{-\infty}^{1}
\frac{dy_1}{\big( |1-y_1|^2 + M^2\big)^{\frac{2+2s}{2}}}
\\ &&\qquad\le
C\,(M-L)\,\left(
\int_{-\infty}^{-M}
\frac{dy_1}{|1-y_1|^{2+2s} } +
\int_{-M}^{1}
\frac{dy_1}{M^{2+2s}}
\right) \\
&&\qquad\le
C\,(M-L)\,M^{-1-2s} = C b M^{\beta-1-2s} \le CM^{\beta-1-2s}
,\end{eqnarray*}
for some~$C>0$ (possibly varying from line to line).
Using this, \eqref{INCL-889} and~\eqref{INCL-889-B},
we obtain that
\begin{equation}\label{INCL-8}
\int_{\R^2\setminus(Q\cup Q')}
\frac{\chi_{E_M}(y)- \chi_{E_M^c}(y)}{|p-y|^{2+2s}}\,dy
=
- 2 \int_{H}
\frac{dy}{|p-y|^{2+2s}}\ge -CM^{\beta-1-2s},\end{equation}
up to renaming~$C$.

Now we use the 
Euler-Lagrange equation in the viscosity sense at~$p$
and we obtain that
$$ \int_{\R^2}
\frac{\chi_{E_M}(y)- \chi_{E_M^c}(y)}{|p-y|^{2+2s}}\,dy\le0.$$
Combining this with~\eqref{INCL-77} and~\eqref{INCL-8}, we obtain
$$ 0\ge \frac{C}{(2b M^{\beta}+1)^{1+2s}}
-CM^{\beta-1-2s}.$$
That is, up to renaming constants,
$$ (2b M^{\beta}+1)^{1+2s}\ge c_* M^{1+2s-\beta},$$
for some~$c_*>0$. Using this
and~\eqref{8u-BETA}, we conclude that
$$ 2b M^{\beta}+1 \ge c_*^{\frac{1}{1+2s}} M^{1-\frac{\beta}{1+2s}} 
=c_o M^{\beta}.$$
Now we multiply by~$M^{-\beta}$
and we take~$M$ large enough, such that~$M^{-\beta}\le c_o/2$,
so we obtain
$$ 2b \ge -M^{-\beta} + c_o \ge \frac{c_o}{2}.$$
This is in contradiction with~\eqref{78UI-CON},
if we choose~$C_o$ small enough.
\end{proof}

As a combination of Propositions~\ref{CYL-RAFF}
and~\ref{PR-3.2}, we have the optimal statement
in Theorem~\ref{CYL-THEOR}.

\section{Stickiness as~$s\to0^+$}\label{SEC:to0}

This section contains the asymptotic properties as~$s\to0$
and the proof of Theorem \ref{SECTOR}.
For this, we recall that~$\Sigma$ has been defined in~\eqref{SECTOR-3456fgg}
as
$$ \Sigma := \{ (x,y)\in\R^2\setminus B_1 {\mbox{ s.t. }}
x>0 {\mbox{ and }} y>0\}$$
and~$E_s$
is the $s$-minimizer in~$B_1$
with datum~$\Sigma$ outside~$B_1$.

\begin{proof}[Proof of Theorem \ref{SECTOR}] First, we show that
\begin{equation}\label{UI}
E_s \subseteq \{ x+y=1\}.
\end{equation}
To prove it, we slide the half-plane~$h_t:=\{ x+y\le t\}$.
If~$t\le -3$, we have that~$h_t$ lies below~$\Sigma\cup B_1$
and so~$h_t\subseteq E_s^c$. Then we increase~$t$
until~$h_{t^*}$ intersects~$E_s$, with~$t_*\in [-3,1]$.
Notice that~\eqref{UI} is proved if we show that
\begin{equation}\label{UI2}
t_*=1.
\end{equation}
We prove this arguing by contradiction.
If not, there exists~$p\in B_1\cap (\partial E_s)\cap \{ x+y=t_*\}$.
Hence, using the
Euler-Lagrange equation in the viscosity sense
(see Theorem~5.1 in~\cite{CRS})
and the fact that~$h_{t_*}\subseteq E_s^c$, we obtain
$$ 0\ge
\int_{\R^2} \frac{\chi_{E^c_s}(y)- \chi_{E_s}(y)}{|p-y|^{2+2s}}\,dy
\ge
\int_{\R^2} \frac{\chi_{h_{t_*}}(y)- \chi_{h_{t_*}^c}(y)}{|p-y|^{2+2s}}\,dy
=0.$$
This shows that~$h_{t_*}$ must coincide with~$E_s^c$.
This is impossible, since~$E_s$ is not a half-plane outside~$B_1$.
Hence, we have proved~\eqref{UI2} and so~\eqref{UI}.

By~\eqref{UI}, we get that~$B_{\sqrt{2}/2}\subseteq E_s^c$.
So we can enlarge~$r\in [\sqrt{2}/2,1]$ till~$B_r$ touches~$E_s$.
We remark that Theorem~\ref{SECTOR} is proved
if we show that this touching property only occurs at~$r=1$.

Thus, we argue by contradiction and we suppose that there exists
\begin{equation}\label{r-below}
r\in[\sqrt{2}/2,1)\end{equation}
such that~$B_r\subseteq E_s^c$
and there exists~$q\in (\partial B_r)\cap(\partial E_s)$.
Then, by the Euler-Lagrange equation, we have that
\begin{equation}\label{ES-SEC-1}
\int_{\R^2} \frac{\chi_{E^c_s}(y)- \chi_{E_s}(y)}{|q-y|^{2+2s}}\,dy
\le0.\end{equation}
By construction,
\begin{equation}\label{ES-SEC-2}
E_s \subseteq 
\{ (x,y)\in\R^2\setminus B_r {\mbox{ s.t. }}
x>0 {\mbox{ and }} y>0\}.\end{equation}
Also, $0<q_1,q_2<1$.
Then we consider the translation by~$q$: namely we define~$F_s:=E_s-q$.
It follows from~\eqref{ES-SEC-2} that
\begin{equation}\label{ES-SEC-3}
F_s \subseteq
\{ (x,y)\in\R^2\setminus B_r(-q) {\mbox{ s.t. }}
x>-1 {\mbox{ and }} y>-1\}.\end{equation}
Also, by~\eqref{ES-SEC-1},
\begin{equation}\label{ES-SEC-4}
\int_{\R^2} \frac{\chi_{F^c_s}(y)- \chi_{F_s}(y)}{|y|^{2+2s}}\,dy
\le0.\end{equation}
Now we define~$D_r:=B_r(q)\cup B_r(-q)$ and we let
$K_r$ be the convex hull of~$D_r$.
Notice that
\begin{equation}\label{BALL_K}
B_r\subseteq K_r.\end{equation} We also define~$P_r:= K_r\setminus D_r$.
Since~$B_r(-q)\subseteq F_s^c$, by symmetry we obtain that
\begin{equation}\label{ES-SEC-6}
\int_{D_r} \frac{\chi_{F^c_s}(y)- \chi_{F_s}(y)}{|y|^{2+2s}}\,dy
\ge0.\end{equation}
Moreover
(see Lemma~3.1 in~\cite{nostro},
used here with~$\lambda:=1$) and~\eqref{r-below},
\begin{equation*}
\left|
\int_{P_r} \frac{\chi_{F^c_s}(y)- \chi_{F_s}(y)}{|y|^{2+2s}}\,dy\right|
\le 
\frac{C_1 r^{-2s} }{1-2s}\le \frac{C_2}{1-2s},
\end{equation*}
for suitable positive constants~$C_1$ and~$C_2$ that do not depend on~$s$.
Using this, \eqref{ES-SEC-4} and~\eqref{ES-SEC-6}
we obtain that
\begin{equation}\label{ES-SEC-10}
\begin{split}
0\,&\ge \int_{\R^2\setminus K_r}
\frac{\chi_{F^c_s}(y)- \chi_{F_s}(y)}{|y|^{2+2s}}\,dy
+\int_{D_r}
\frac{\chi_{F^c_s}(y)- \chi_{F_s}(y)}{|y|^{2+2s}}\,dy
+
\int_{P_r}
\frac{\chi_{F^c_s}(y)- \chi_{F_s}(y)}{|y|^{2+2s}}\,dy
\\ &\ge
\int_{\R^2\setminus K_r}
\frac{\chi_{F^c_s}(y)- \chi_{F_s}(y)}{|y|^{2+2s}}\,dy
-\frac{C_2}{1-2s}.\end{split}\end{equation}
Moreover, recalling~\eqref{BALL_K}
(and using again~\eqref{r-below}), we have that
$$ \left|\int_{B_{2r}\setminus K_r}
\frac{\chi_{F^c_s}(y)- \chi_{F_s}(y)}{|y|^{2+2s}}\,dy\right|\le
\int_{B_{2r}\setminus B_r}
\frac{dy}{|y|^{2+2s}}\le C_3,$$
for some~$C_3>0$ that does not depend on~$s$. Hence~\eqref{ES-SEC-10}
gives
\begin{equation}\label{ES-SEC-11}
0\ge 
\int_{\R^2\setminus B_{2r}}
\frac{\chi_{F^c_s}(y)- \chi_{F_s}(y)}{|y|^{2+2s}}\,dy
-C_3-\frac{C_2}{1-2s}.\end{equation}
Now we observe that~$B_r(-q)\subseteq B_{2r}$, since~$|q|=r$.
Consequently, recalling~\eqref{ES-SEC-3},
\begin{equation*}
F_s \setminus B_{2r} \subseteq
\{ (x,y)\in\R^2 \setminus B_{2r}
{\mbox{ s.t. }}
x>-1 {\mbox{ and }} y>-1\}.\end{equation*}
That is,~$F_s \setminus B_{2r} \subseteq A_1\cup A_2\cup A_3$, where
\begin{eqnarray*}
&& A_1 :=
\{ (x,y)\in\R^2 \setminus B_{2r}
{\mbox{ s.t. }} x>0 {\mbox{ and }} y\in(-1,1)\},\\
&& A_2 :=
\{ (x,y)\in\R^2 \setminus B_{2r}
{\mbox{ s.t. }} x\in(-1,1) {\mbox{ and }} y>0\}\\
{\mbox{and }}&& A_3:=
\{ (x,y)\in\R^2 \setminus B_{2r}
{\mbox{ s.t. }} x\ge1 {\mbox{ and }} y\ge1\}.
\end{eqnarray*}

\begin{figure}
    \centering
    \includegraphics[height=5.9cm]{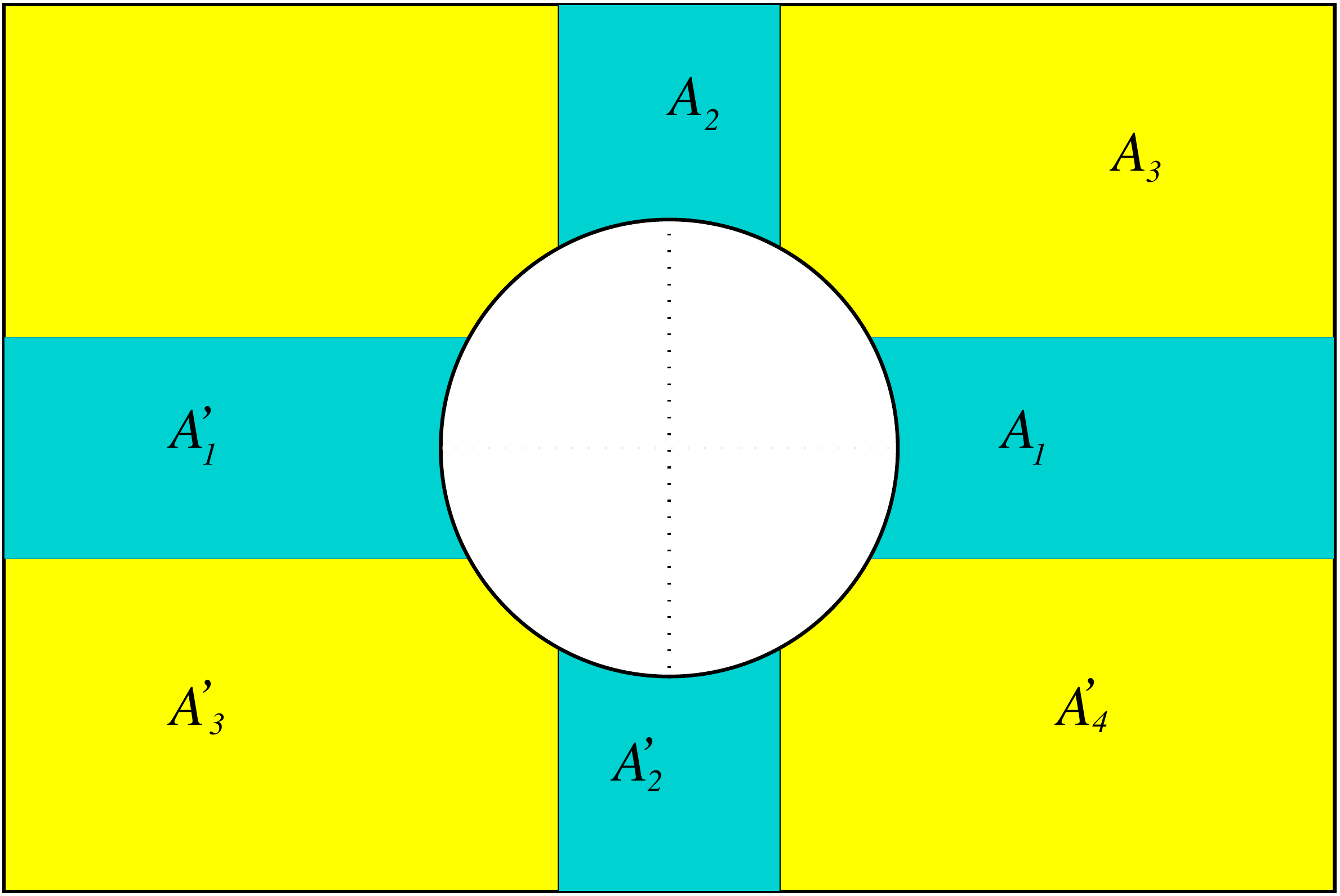}
    \caption{
The partition of the plane needed for the proof
of Theorem~\ref{SECTOR}.}
    \label{SIMP}
\end{figure}

On the other hand,~$F_s^c \setminus B_{2r} \supseteq A_1'\cup
A_2'\cup A_3'\cup A_4'$, where
\begin{eqnarray*}
&& A_1' :=
\{ (x,y)\in\R^2 \setminus B_{2r}
{\mbox{ s.t. }} x<0 {\mbox{ and }} y\in(-1,1)\},\\
&& A_2' :=
\{ (x,y)\in\R^2 \setminus B_{2r}
{\mbox{ s.t. }} x\in(-1,1) {\mbox{ and }} y<0\},\\
&& A_3' :=
\{ (x,y)\in\R^2 \setminus B_{2r}
{\mbox{ s.t. }} x\le-1 {\mbox{ and }} y\le-1\},\\
{\mbox{and }}&& A_4':=
\{ (x,y)\in\R^2 \setminus B_{2r}
{\mbox{ s.t. }} x\ge1 {\mbox{ and }} y\le-1\},
\end{eqnarray*}
see Figure~\ref{SIMP}.
After simplifying~$A_1$ with~$A_1'$, $A_2$ with~$A_2'$
and~$A_3$ with~$A_3'$, we obtain
\begin{equation}\label{78:90P}
\int_{\R^2\setminus B_{2r}}
\frac{\chi_{F^c_s}(y)- \chi_{F_s}(y)}{|y|^{2+2s}}\,dy
\ge \int_{A'_4} \frac{dy}{|y|^{2+2s}}.\end{equation}
Notice now that~$A'_4$ contains a cone with positive
constant opening with vertex at the origin,
therefore
$$ \int_{A'_4} \frac{dy}{|y|^{2+2s}}\ge
c_1 \int_{2r}^{+\infty} \frac{d\rho}{\rho^{1+2s}}
= \frac{c_2}{s\,r^{2s}} \ge \frac{c_3}{s},$$
where we have used again~\eqref{r-below}, and the positive
constants~$c_1$, $c_2$ and~$c_3$ do not depend on~$s$.
The latter estimate and~\eqref{78:90P} give that
$$ \int_{\R^2\setminus B_{2r}}
\frac{\chi_{F^c_s}(y)- \chi_{F_s}(y)}{|y|^{2+2s}}\,dy
\ge \frac{c_3}{s}.$$
Therefore, recalling~\eqref{ES-SEC-11},
$$ 0\ge \frac{c_3}{s}
-C_3-\frac{C_2}{1-2s}.$$
This is a contradiction if~$s\in(0,s_o)$ and~$s_o$
is small enough. Hence, we have completed the proof of
Theorem~\ref{SECTOR}.
\end{proof}

\section{Construction of barriers that are piecewise linear}\label{SEC:BAR:PTWL}

This part of the paper is devoted to the proof of
Theorem~\ref{UNS}.
The argument will rely on the construction of a series
of barriers, and the proof of Theorem~\ref{UNS}
will be completed in Section~\ref{INST:SEC}.

In this section, we construct barriers in the plane,
which are subsolutions of the fractional curvature equation when~$\{x_1>0\}$,
which possess a ``vertical'' portion along~$\{x_1=0\}$ and
which are built by joining linear functions
whose slope becomes arbitrarily close to being horizontal
(a precise statement will be given in Proposition~\ref{IT:BARR}).
For this scope,
we start with a simple
auxiliary observation to bound explicitly from below the
fractional curvature of an angle:

\begin{lemma}\label{ANGLE}
Let~$\ell\ge0$,
\begin{eqnarray*}
&&E_1:= (-\infty,0]\times (-\infty,0)\\
&& E_2:= \{ \ell x_2-x_1<0 ,\quad x_1>0\}\\
{\mbox{and }}&& E:=E_1\cup E_2.
\end{eqnarray*}
Then, for any~$p=(p_1,p_2)\in\partial E$ with~$p_2>0$,
\begin{equation}\label{SCD:k}
\int_{\R^2} \frac{\chi_E(y)-\chi_{E^c}(y)}{|y-p|^{2+2s}}\,dy\ge
\frac{c(\ell)}{|p|^{2s}}, \end{equation}
for a suitable nonincreasing function~$c:[0,+\infty)\to(0,1)$.

More precisely, for large~$\ell$, one has that~$c(\ell)\sim \bar c\ell^{-1}$,
for some~$\bar c>0$.
\end{lemma}

\begin{proof} Let~$\delta:=\arctan (1/\ell)\in \left(0,\frac{\pi}{2}\right]$.
By scaling, it is enough to prove~\eqref{SCD:k} when
\begin{equation}\label{SCD:k2}
|p|=4.\end{equation}

\begin{figure}
    \centering
    \includegraphics[height=5.9cm]{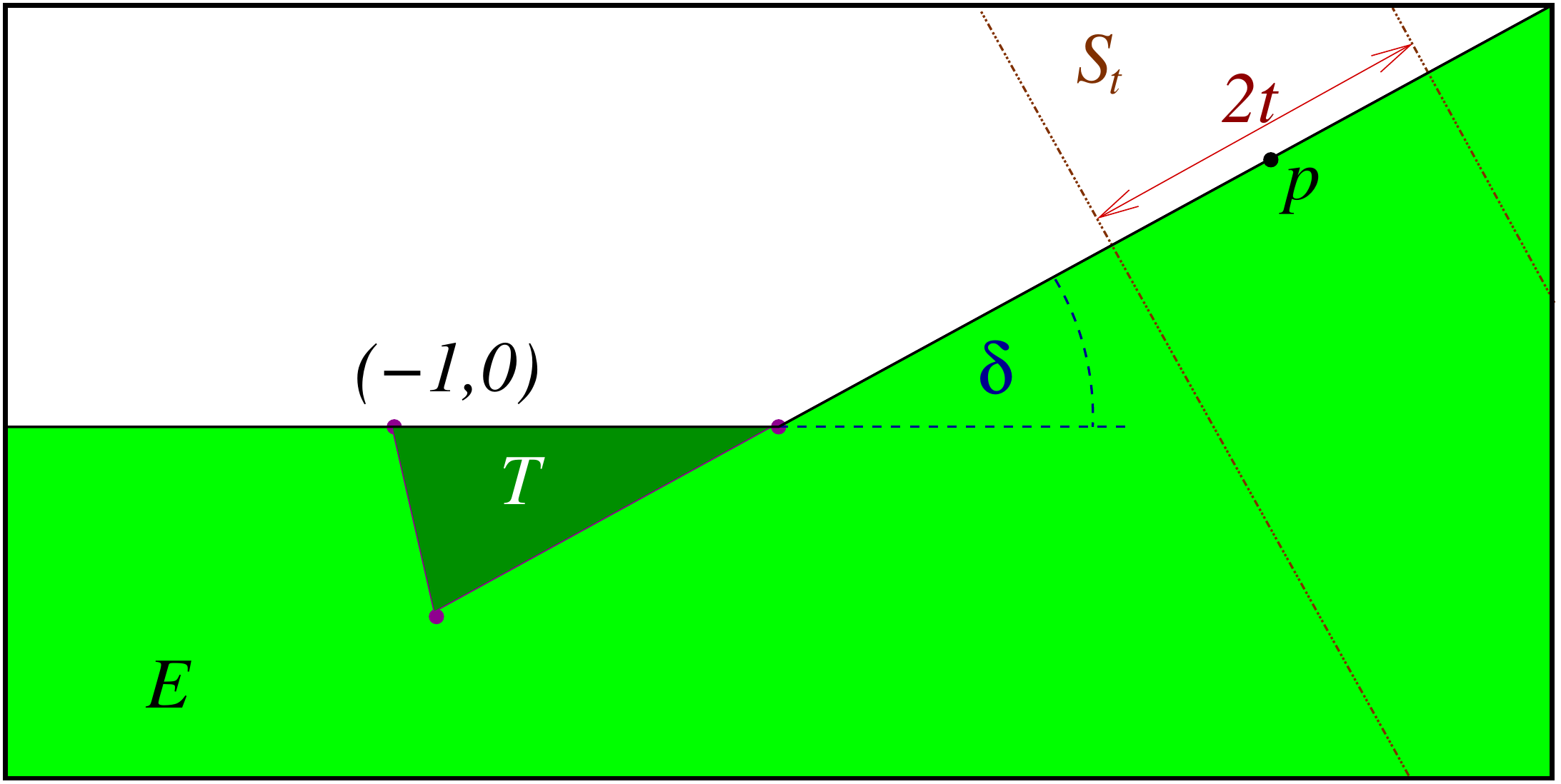}
    \caption{The proof of Lemma~\ref{ANGLE}.}
    \label{FsA}
\end{figure}

Now, for any~$t>0$, let~$S_t$ be the slab with boundary orthogonal
to the straight line~$\{ \ell x_2-x_1=0\}$
of width~$2t$, having~$p$ on its symmetry axis (see Figure~\ref{FsA}).
For small~$t$, the slab~$S_t$ does not contain the origin,
thus, the ``upper'' half of the slab is contained in~$E^c$,
while the ``lower'' half of the slab is contained in~$E$,
namely
\begin{equation*}
\int_{S_t} \frac{\chi_E(y)-\chi_{E^c}(y)}{|y-p|^{2+2s}}\,dy=0
.\end{equation*}
Enlarging~$t$,
the ``lower'' half of the slab is always contained in~$E$.
As for the ``upper'' half, we have that the triangle~$T$
with vertices~$(0,0)$, $(-\cos\delta,-\sin\delta)$, $(-1,0)$
lies in~$E$. Notice that
$$|T|=\frac{\sin\delta}{2}.$$
Also, if~$y\in T$ then~$|y|\le 2$ and so,
recalling~\eqref{SCD:k2},
$$ |y-p|\le |p|+2 \le 2|p|.$$
Consequently,
$$ \int_{T} \frac{\chi_E(y)-\chi_{E^c}(y)}{|y-p|^{2+2s}}\,dy
\ge \frac{ |T|}{ 2^{2+2s}\,|p|^{2+2s} }
= \frac{ \sin\delta }{ 2^{7+2s}\,|p|^{2s} },$$
which gives the desired result.
\end{proof}

The next result is the building block needed to
construct a barrier iteratively.
Roughly speaking, next result says that 
we can tilt a straight line towards infinity
by estimating precisely the effect of this modification
on the fractional curvature.

\begin{lemma}\label{PRE:BAR}
Let~$\ell\ge q\ge0$ and~$\delta:=\arctan (1/\ell)\in \left(0,\frac{\pi}{2}\right]$.
Let~$e:=(\ell-q,1)$.

Let~$\bar\tau\in C^\infty_0( B_1(e))$ with~$\bar\tau=1$ in~$B_{1/2}(e)$.

Let~$\tau_o\in C^\infty(\R)$ be such that
\begin{equation}
{\mbox{$\tau_o(t)=1$
if~$t\in \left[ \frac{\delta}{2},\frac{3\delta}{2}\right]$
and~$\tau_o(t)=0$ if~$t\in\R\setminus
\left[ \frac{\delta}{4},\frac{7\delta}{4}\right]$.}}\end{equation}
For any~$x\in\R^2$, let also~$\alpha(x)\in [0,2\pi)$ be
the angle between the vector~$x-e$ and the $x_1$-axis.
Let
\begin{equation}\label{PIOP11}
\tau(x):= \big(1-\bar\tau(x)\big)\,\tau_o\big( \alpha(x)\big).\end{equation}
For any~$\theta\in\R$, let~$R_\theta$ be the clockwise rotation by an angle~$\theta$, i.e.
$$ R_\theta(x)=R_\theta(x_1,x_2):=
\left(
\begin{matrix}
\cos\theta & \sin \theta\\
-\sin\theta & \cos\theta
\end{matrix}
\right) \,
\left(
\begin{matrix}
x_1\\ x_2 \end{matrix}
\right). $$
Let also
$$ \Psi_\theta(x):= R_{\tau(x)\theta} \,x.$$
Let~$E\subset\R^2$
be an epigraph such that
\begin{equation*}
\begin{split}
& E \cap \{x_1 < 0\} = (-\infty,0)\times(-\infty,0), \\
& E\supseteq \R\times (-\infty,0),\\
& E\cap \{x_2 >1\} = \{\ell x_2-x_1 -q <0\}\cap\{x_2 >1\}\\
{\mbox{and }}\;& E\cap \{x_1>\ell-q\} = \{\ell x_2-x_1 -q <0\}\cap\{x_1 >\ell-q\}.
\end{split}
\end{equation*}
Assume that, for any~$p\in \partial E\cap\{x_2>0\}$,
\begin{equation}\label{STI:STEp0}
\int_{\R^2} \frac{\chi_E(y)-\chi_{E^c}(y)}{|y-p|^{2+2s}}\,dy\ge
\frac{c}{|p|^{2s}}
\end{equation}
for some~$c\in(0,1)$.

Then, there exist nonincreasing functions~$\phi:[0,+\infty)\to
(0,1)$ and~$c_o:[0,+\infty)\to (0,c)$
such that for any~$\theta\in [0,\,\phi(\ell)]$ the following claim
holds true. Let~$F:=\Psi_\theta(E)$. Then, for any~$p\in (\partial F)
\cap\{x_2>0\}$,
\begin{equation}\label{VBhJ}
\int_{\R^2} \frac{\chi_F(y)-\chi_{F^c}(y)}{|y-p|^{2+2s}}\,dy\ge
\frac{c_o (\ell)}{|p|^{2s}}
.\end{equation}
More precisely, for large~$\ell$, one has that~$c_o(\ell)\sim
\bar c\min\{c,\ell^{-1}\}$,
for some~$\bar c>0$.
\end{lemma}

\begin{proof} First we point out that
\begin{equation}\label{DE:alpha}
|\nabla\alpha(x)|\le \frac{C}{|x-e|}
,\end{equation}
for some~$C>0$. Indeed, $\alpha(x)$ is identified by the two
conditions
\begin{equation}\label{cos-IU-1}
|x-e|\cos\alpha(x)=|x_1-\ell+q|\end{equation}
and~$|x-e|\sin\alpha(x)=|x_2-1|$.
Assume also that~$\sin^2\alpha(x)\ge 1/2$ (the case~$\cos^2\alpha(x)\ge 1/2$
is similar). Then we differentiate the relation~\eqref{cos-IU-1}
and we obtain
$$ \frac{x-e}{|x-e|}\cos\alpha(x)
-|x-e|\sin\alpha(x)\,\nabla\alpha(x)
=\frac{x_1-\ell+q}{|x_1-\ell+q|}\,(1,0).$$
Therefore
\begin{eqnarray*}
\frac{\sqrt{2}}2 |x-e|\,|\nabla\alpha(x)|
\le |x-e|\,|\sin\alpha(x)|\,|\nabla\alpha(x)|
= \left|
\frac{x-e}{|x-e|}\cos\alpha(x)-\frac{x_1-\ell+q}{|x_1-\ell+q|}\,(1,0)\right|
\le 2,
\end{eqnarray*}
which proves~\eqref{DE:alpha}.

Similarly, taking one more derivative, one sees that
\begin{equation}\label{DE:alpha:second}
|D^2\alpha(x)|\le \frac{C}{|x-e|^2}
.\end{equation}
Now, by~\eqref{PIOP11} and~\eqref{DE:alpha},
\begin{equation}\label{DE:alpha:2}
|\nabla\tau(x)|\le C\left( \chi_{B_1(e)\setminus B_{1/2}(e)}(x)+
\frac{\chi_{\R^2\setminus B_{1/2}(e)}(x)}{|x-e|}\right).
\end{equation}
Using~\eqref{DE:alpha:second}, one also obtains that
\begin{equation}\label{DE:alpha:2:second}
|D^2\tau(x)|\le C\left( \chi_{B_1(e)\setminus B_{1/2}(e)}(x)+
\frac{\chi_{\R^2\setminus B_{1/2}(e)}(x)}{|x-e|}\right).
\end{equation}
Let now
$$ \Phi_\theta(x):=
\Psi_\theta(x)-x= 
\left(
\begin{matrix}
\cos(\tau(x)\theta)-1 & \sin (\tau(x)\theta)\\
-\sin(\tau(x)\theta) & \cos(\tau(x)\theta)-1
\end{matrix}
\right) \,
\left(
\begin{matrix}
x_1\\ x_2 \end{matrix}
\right).$$
We claim that
\begin{equation}\label{DIF:1}
|D\Phi_\theta(x)|\le C\,(1+\ell)\,\theta,
\end{equation}
for some~$C>0$. To prove it, we consider the first coordinate of~$\Phi_\theta(x)$,
which is
\begin{equation}\label{first coo}
\big(\cos(\tau(x)\theta)-1\big)\,x_1 +
\sin(\tau(x)\theta)\,x_2,\end{equation}
since the computation with the second coordinate is similar.
We bound the derivative of~\eqref{first coo} by
\begin{equation}\label{first coo2}
\big|\cos(\tau(x)\theta)-1\big|+
\big|\sin(\tau(x)\theta)\big|
+ \theta\,\Big( \big|\sin(\tau(x)\theta)\big|+\big|\cos(\tau(x)\theta)\big|\Big)
|\nabla\tau(x)|\,|x|.\end{equation}
Thus, we bound~$\big|\cos(\tau(x)\theta)-1\big|\le C\theta^2$
and~$\big|\sin(\tau(x)\theta)\big|\le C\theta$ and we make use of~\eqref{DE:alpha:2},
to estimate the quantity in~\eqref{first coo2} by
\begin{equation}\label{first coo3}
C\theta \,\left(1+\frac{
\chi_{\R^2\setminus B_{1/2}(e)}(x)
\;|x|}{|x-e|} \right).\end{equation}
Now we observe that~$|e|=\sqrt{(\ell-q)^2+1}\le\sqrt{\ell^2+1}$, therefore
$$ |x|\le |x-e|+\sqrt{\ell^2+1}$$
and so, if~$|x-e|\ge1/2$,
$$ \frac{|x|}{|x-e|} \le 1 +2\sqrt{\ell^2+1}.$$
By inserting this information into~\eqref{first coo3}
we bound the first coordinate of~$\Phi_\theta(x)$ by~$C\,(1+\ell)\,\theta$.
This proves~\eqref{DIF:1}.

Similarly, making use of~\eqref{DE:alpha:2:second},
one sees that
\begin{equation}\label{DIF:2}
|D^2\Phi_\theta(x)|\le C\,(1+\ell)\,\theta.
\end{equation}
Notice also that, for any fixed~$x\in\R^2$,
we have that
\begin{equation}\label{NORM:1}
|\Psi_\theta(x)|= |R_{\tau(x)\theta} \,x|=|x|, \end{equation}
therefore
$$ \lim_{|x|\to+\infty} |\Psi_\theta(x)|=+\infty.$$
{F}rom this,~\eqref{DIF:1}, and the Global Inverse Function Theorem
(see e.g. Corollary 4.3 in~\cite{palais-1959-natural}),
we obtain that~$\Psi_\theta$ is a global diffeomorphism of~$\R^2$,
see Figure~\ref{FsB}.

\begin{figure}
    \centering
    \includegraphics[height=5.9cm]{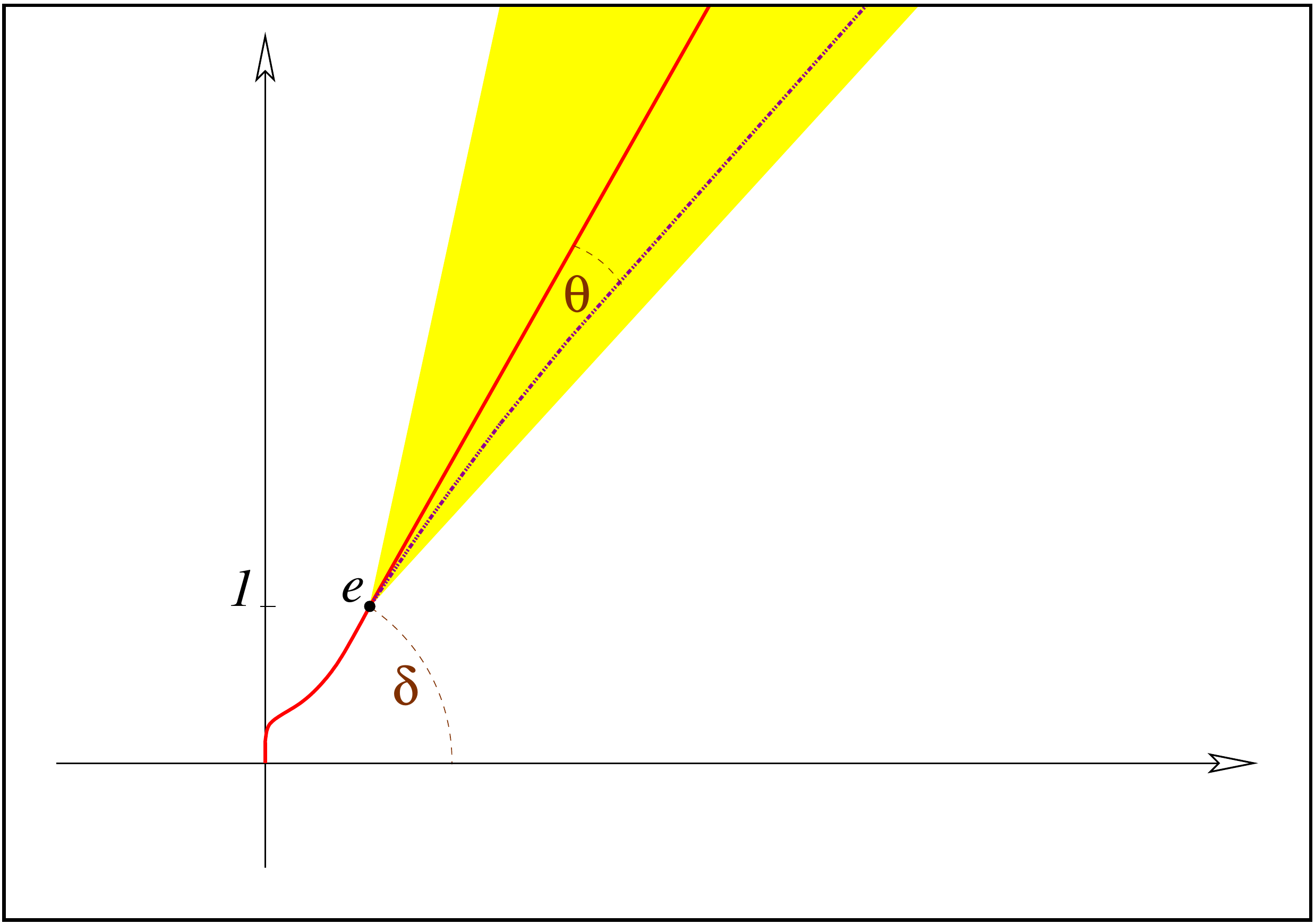}
    \caption{The diffeomorphism of~$\R^2$ in Lemma~\ref{PRE:BAR}.}
    \label{FsB}
\end{figure}

As a consequence, using~\eqref{DIF:1}, \eqref{DIF:2}
and the curvature estimates for diffeomorphisms (see Theorem~1.1
in~\cite{cozzi}), we conclude that
\begin{equation}\label{ST:MAR}
\int_{\R^2} \frac{\chi_F(y)-\chi_{F^c}(y)}{|y-p|^{2+2s}}\,dy
\ge \int_{\R^2} \frac{\chi_E(y)-\chi_{E^c}(y)}{|y-q|^{2+2s}}\,dy-
C\,(1+\ell)\,\theta,
\end{equation}
with~$q:=\Psi_\theta^{-1}(p)$, for any~$p\in(\partial F)\cap \{x_2>0\}$.

Now we claim that
\begin{equation}\label{p02:02}
{\mbox{if~$p\in\{x_2>0\}$ then~$\Psi_\theta^{-1}(p)\in\{x_2>0\}$.}}
\end{equation}
Suppose, by contradiction, that~$\Psi_\theta^{-1}(p)\in\{x_2\le0\}$.
Notice that~$\tau$ vanishes in~$\{x_2\le0\}$, therefore
$ \Psi_\theta$ is the identity in~$\{x_2\le0\}$.
As a consequence~$p=\Psi_\theta(\Psi_\theta^{-1}(p))=\Psi_\theta^{-1}(p)
\in\{x_2\le0\}$. This is a contradiction with our assumptions
and so it proves~\eqref{p02:02}.

Using~\eqref{STI:STEp0}, \eqref{NORM:1}, \eqref{ST:MAR}
and~\eqref{p02:02}, we have that
\begin{equation}\label{Go:09-01p1}
\int_{\R^2} \frac{\chi_F(y)-\chi_{F^c}(y)}{|y-p|^{2+2s}}\,dy
\ge \frac{c}{|q|^{2s}}
-C\,(1+\ell)\,\theta = \frac{c}{|p|^{2s}}
-C\,(1+\ell)\,\theta \ge \frac{c}{2\,|p|^{2s}},
\end{equation}
with~$q:=\Psi_\theta^{-1}(p)$, for any~$p\in(\partial F)\cap \{x_2>0\}\cap
B_{r_\theta}$, where~$r_\theta:={\left(\frac{c}{2C\,(1+\ell)\,\theta}\right)^{\frac{1}{2s}}}$
(we stress that~$r_\theta$ is large, for small~$\theta$,
according to the statement of Lemma~\ref{PRE:BAR}).

Now we take~$p\in\big((\partial F)\cap \{x_2>0\}\big)\setminus B_{r_\theta}$
and we observe that~$((\partial F)\setminus B_{r_\theta})\cap \{x_2>0\}$
coincides with a straight line of the form~$\lambda:=\{\ell_\theta x_2-x_1-q_\theta=0\}$,
with~$\ell_\theta\ge \ell$, $|\ell_\theta-\ell|$ as close
to zero as we wish for small~$\theta$, and~$q_\theta:=\ell_\theta-\ell+q$.
The intersections of the straight line~$\lambda$
with~$\{x_2=8\}$ and~$\{x_2=0\}$ occur at points~$x_1=8\ell_\theta-q_\theta$
and~$x_1=-q_\theta$, respectively.

Hence, we consider the triangle~$T$ with vertices~$(8\ell_\theta-q_\theta,0)$, $(8\ell_\theta-q_\theta,8)$
and~$(-q_\theta, 0)$. We observe that~$|T|=32\ell_\theta\le 32(1+\ell)$,
for small~$\theta$. Moreover, if~$y\in T$, then~$|y|\le 
C(1+\ell_\theta+q_\theta)\le C(1+\ell)$, up to renaming constants.
Therefore, if~$p\in B_{r_\theta}^c$ and~$y\in T$,
$$ |y-p|\ge |p|- C(1+\ell)\ge \frac{|p|}{2},$$
if~$\theta$ is small. Consequently,
\begin{equation}\label{13UI:0}
\int_{T} \frac{dy}{|y-p|^{2+2s}}\le \frac{C(1+\ell)}{|p|^{2+2s}}
\le \frac{C(1+\ell)}{r_\theta^2\,|p|^{2s}}.
\end{equation}
Now we define~$\tilde F:= F\cup T$. By Lemma~\ref{ANGLE},
$$ \int_{\R^2} \frac{\chi_{\tilde F}(y)-\chi_{\tilde F^c}(y)}{
|y-p|^{2+2s}}\,dy\ge
\frac{c(\ell_\theta)}{|p-(-q_\theta,0)|^{2s}}.$$
Using that~$\ell_\theta\le \frac{3\ell}{2}$ and that~$c(\cdot)$ is
nonincreasing, we see that~$c(\ell_\theta)\ge c\left(\frac{3\ell}{2}\right)$.
Moreover,
$$ |p-(-q_\theta,0)|\le |p|+q_\theta\le |p|+\ell+1\le 2|p|,$$
so we obtain that
$$ \int_{\R^2} \frac{\chi_{\tilde F}(y)-\chi_{\tilde F^c}(y)}{
|y-p|^{2+2s}}\,dy\ge
\frac{c\left(\frac{3\ell}{2}\right)}{2^{2s} \, |p|^{2s}}.$$
Exploiting this and~\eqref{13UI:0}, we obtain that,
for any~$p\in\big((\partial F)\cap \{x_2>0\}\big)\setminus B_{r_\theta}$,
\begin{equation}\label{Go:09-01p2}
\begin{split}
& \int_{\R^2} \frac{\chi_F(y)-\chi_{F^c}(y)}{|y-p|^{2+2s}}\,dy
\ge 
\int_{\R^2} \frac{\chi_{\tilde F}(y)-\chi_{\tilde F^c}(y)}{
|y-p|^{2+2s}}\,dy - \int_{T} \frac{dy}{|y-p|^{2+2s}}
\\ &\qquad\ge \frac{c\left(\frac{3\ell}{2}\right)}{2^{2s} \, |p|^{2s}} - \frac{C(1+\ell)}{r_\theta^2\,|p|^{2s}}
\ge \frac{c\left(\frac{3\ell}{2}\right)}{2^{1+2s} \, |p|^{2s}},
\end{split}\end{equation}
for small~$\theta$. Then, \eqref{VBhJ} follows by combining
\eqref{Go:09-01p1}
and~\eqref{Go:09-01p2}.
\end{proof}

By iterating Lemma~\ref{PRE:BAR} we can
construct the following barrier:

\begin{proposition}\label{IT:BARR}
Fix~$K\ge0$. Then there exist~$a_K\in(0,1)$,
$\ell_K\ge K$, $q_K\ge0$, $c_K\in(0,1)$,
a continuous function~$u_K :[0,+\infty)\to [0,+\infty)$
and a set~$E_K\subset\R^2$ with $(\partial E_K)\cap \{x_2>0\}$
of class~$C^{1,1}$
and such that:
\begin{itemize}
\item $u_K(x_2) =\ell_K\,x_2-q_K $
for any~$x_2\in [1,+\infty)$,
\item we have that
\begin{eqnarray*}
&& E_K \cap \{x_1 < 0\} = (-\infty,0)\times(-\infty,0), \\
&& E_K\supseteq \R\times (-\infty,0),\\
&& E_K\supseteq (0,+\infty)\times (-\infty,a_K],\\
&& E_K\cap \{x_2 >1\} = \{x_1> u_K(x_2), \quad x_2 >1\},\\
&& E_K\cap \{x_1>\ell_K-q_K\} =
\{x_1> u_K(x_2),\quad x_1 >\ell_K-q_K\}\end{eqnarray*}
and
$$ \int_{\R^2} \frac{\chi_{E_K}(y)-\chi_{E_K^c}(y)}{|y-p|^{2+2s}}\,dy
\ge \frac{c_K}{|p|^{2s}},$$
for any~$p\in (\partial E_K)\cap \{x_2>0\}$.
\end{itemize}
More precisely, for large~$K$, one has that~$c_K\sim \bar c\ell_K^{-1}$,
for some~$\bar c>0$. 

Moreover, one can also prescribe that
\begin{equation}\label{more}
q_K\le K^{-1}.\end{equation}
\end{proposition}

\begin{proof} We apply Lemma~\ref{PRE:BAR} iteratively for a large
(but finite) number of times, see Figure~\ref{FsC}.

\begin{figure}
    \centering
    \includegraphics[height=5.9cm]{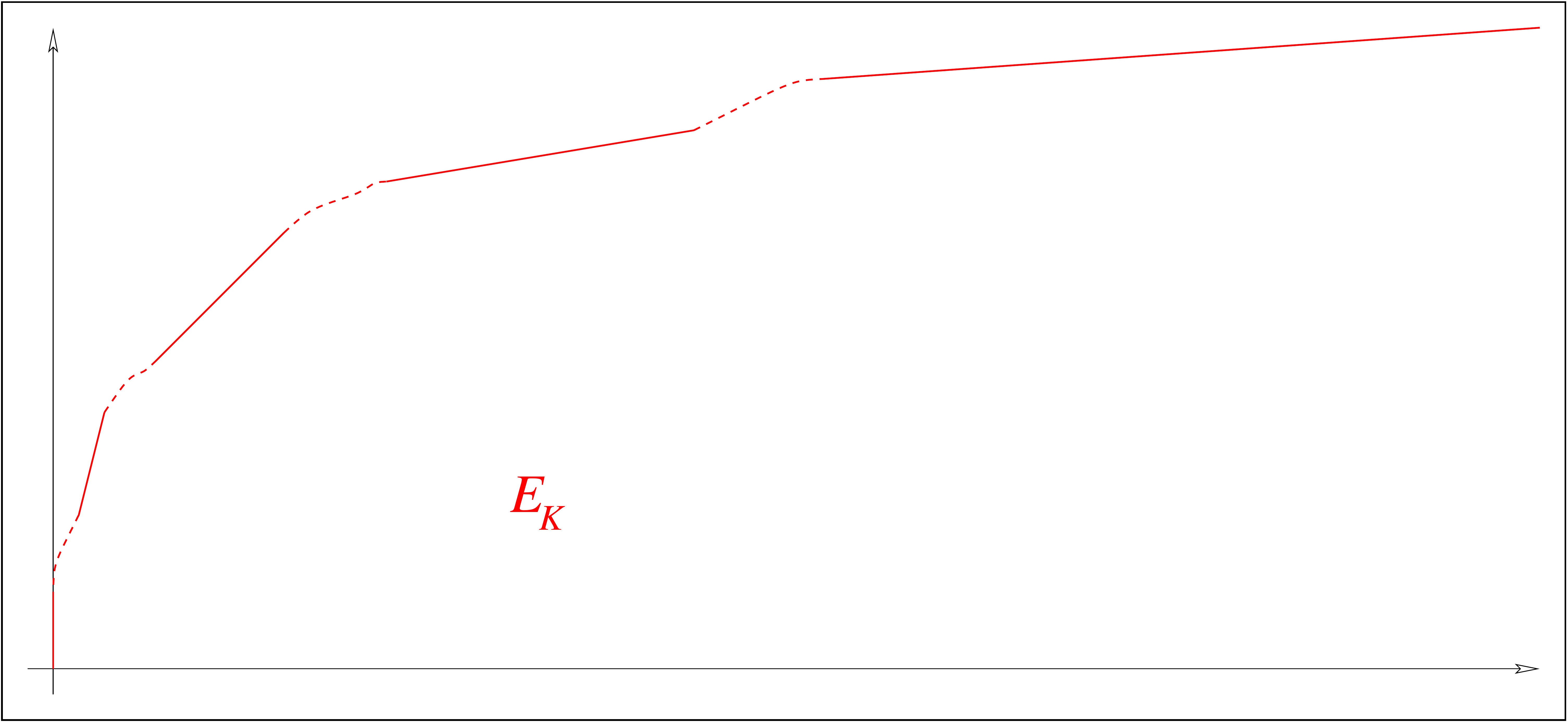}
    \caption{The barrier of Proposition~\ref{IT:BARR}.}
    \label{FsC}
\end{figure}

We start with~$u_0:=0$
and~$E_0:=\R^2\setminus\{ x_1\le0\le x_2\}$.
By Lemma~\ref{ANGLE} (used here with~$\ell:=0$)
we know that
$$ \int_{\R^2} \frac{\chi_{E_0}(y)-\chi_{E_0^c}(y)}{|y-p|^{2+2s}}\,dy
\ge \frac{c}{|p|^{2s}},$$
for some~$c>0$. Then we apply Lemma~\ref{PRE:BAR}
and we construct a set~$E_1$ whose boundary
coincides with~$\{x_2=0\}$ when~$\{x_1<0\}$
and with a straight line~$\{\ell_1 x_2-x_1-q_1=0\}$
when~$\{x_2>4\}$, whose fractional curvature satisfies the desired estimate
(as a matter of fact, we can take the new slope~$\ell_1$
as the one obtained by~$\phi(0)$
in Lemma~\ref{PRE:BAR}, thus~$\ell_1>0$).

Then we scale~$E_1$ by a factor~$\frac12$ and we apply once again
Lemma~\ref{PRE:BAR}, obtaining
a set~$E_2$ whose boundary
coincides with~$\{x_2=0\}$ when~$\{x_1<0\}$
and with a straight line~$\{\ell_2 x_2-x_1-q_2=0\}$
when~$\{x_2>4\}$, whose fractional curvature satisfies the desired estimate.
Notice again that~$\ell_2$ is obtained
in Lemma~\ref{PRE:BAR} by rotating clockwise
the straight line of slope~$\ell_1$ by an angle~$\phi(\ell_1)>0$,
hence~$\ell_2>\ell_1$.

Iterating this procedure, we obtain a sequence of 
increasing slopes~$\ell_j$ and sets~$E_j$
satisfying the desired geometric properties.
We stress that, for large~$j$, the slope~$\ell_j$
must become larger than the quantity~$K$ fixed
in the statement of Proposition~\ref{IT:BARR}.
Indeed, if not, say if~$\ell_j\le \ell_\star$ for some~$\ell_\star>0$,
at each step of the iteration we could
rotate the straight line by an angle of size larger than~$\phi(\ell_\star)$,
which is a fixed positive quantity 
(recall that~$\phi$ in 
Lemma~\ref{PRE:BAR} is nonincreasing): hence
repeating this argument many times we would make the slope become
bigger than~$\ell_\star$, that is a contradiction.

Thus, we can define~$j_o$ to be the first~$j$ for which~$\ell_j\ge K$.
The set~$E_{j_o}$ obtained in this way satisfies the desired properties,
with the possible exception of~\eqref{more}.
So, to obtain~\eqref{more}, we may suppose that~$q_{j_o}>K^{-1}$,
otherwise we are done, and
we scale the picture once again by a factor~$\mu:= K^{-1} q_{j_o}^{-1}\in(0,1)$.
In this way, the geometric properties of the set
and the estimates on the fractional curvature are preserved,
but the line~$\{ \ell_{j_o} x_2-x_1-q_{j_o}=0\}$
is transformed into the line~$\{ \ell_{j_o} x_2-x_1-\tilde q_{j_o}=0\}$,
with~$\tilde q_{j_o}:= \mu q_{j_o}$.
By construction, we have that~$\tilde q_{j_o}=K^{-1}$,
which gives~\eqref{more}.
\end{proof}

\section{Construction of barriers which grow
like~$x_1^{\frac{1}{2}+s+\epsilon_0}$}\label{sec:GROW:R}

In this section, we construct barriers in the plane,
which are subsolutions of the fractional curvature equation when~$\{x_1>0\}$,
which possess a ``vertical'' portion along~$\{x_1=0\}$ and
which grow like~$x_1^{\frac{1}{2}+s+\epsilon_0}$ at infinity
(here, $\epsilon_0>0$ is arbitrarily small).
This is a refinement of the barrier constructed in
Proposition~\ref{IT:BARR}, which grows linearly (with almost
horizontal slope). Roughly speaking, 
the difference with Proposition~\ref{IT:BARR} is
that the results obtained there have nice scaling properties
and an elementary geometry (since the barrier constructed
there is basically the junction of a finite number
of straight lines) but do not possess
an optimal growth at infinity. As a matter of fact,
the power obtained
here at infinity is dictated by the growth of
the functions that are harmonic with respect
to the fractional Laplacian~$(-\Delta)^{\gamma_0}$,
where
\begin{equation}
\gamma_0: =\frac{1}{2}+s.\end{equation}
As a matter of fact, this procedure
provides a good approximation
of the fractional mean curvature equation at points with
nearly horizontal tangent.
Namely, we set
$$ \gamma:= \frac{1}{2}+s+\epsilon_0=\gamma_0+\epsilon_0
\in\left(\frac12,1\right).$$
We will use the fact that~$\gamma>\gamma_0$ to construct
a subsolution of the $\gamma_0$-fractional Laplace equation.
More precisely, the main formula we need in this framework
is the following:

\begin{lemma}\label{TOR}
Let~$\epsilon_0\in(0,1-\gamma_0)$. We have that
$$ \frac{1}{2} \int_{\R}
\frac{ (1+t)^\gamma_+ + (1-t)^\gamma_+ -2 }{|t|^{2+2s}}\,dt
\ge c_\star \epsilon_0,$$
for some~$c_\star >0$.
\end{lemma}

\begin{proof} Let~$r\ge0$. By a Taylor expansion at~$r=1$, we have that
$$ r^{\frac{\gamma}{\gamma_0}} = 1 +{\frac{\gamma\,(r-1)}{\gamma_0}}
+{\frac{\gamma\,(\gamma-\gamma_0)\,
\xi^{ {\frac{\gamma}{\gamma_0}} -2}\,(r-1)^2 }{\gamma_0^2}},$$
for some~$\xi$ on the segment joining~$r$ to~$1$.
In particular, $\xi\le 1+r$.
Using this with~$r:=(1\pm t)^{\gamma_0}_+$, we obtain
$$ (1\pm t)^\gamma_+
= 1 +{\frac{\gamma\,\big((1\pm t)^{\gamma_0}_+-1\big)}{\gamma_0}}
+{\frac{\gamma\,(\gamma-\gamma_0)\,
\xi^{ {\frac{\gamma}{\gamma_0}} -2}\,
\big((1\pm t)^{\gamma_0}_+-1\big)^2 }{\gamma_0^2}},$$
for some~$\xi\in [0, 2+|t|]$.
Consequently, since
$$ {\frac{\gamma}{\gamma_0}} -2 ={\frac{\epsilon_0}{\gamma_0}} -1<0$$
we obtain
that
$$ \xi^{ {\frac{\gamma}{\gamma_0}} -2} \ge
(2+|t|)^{ {\frac{\gamma}{\gamma_0}} -2}\ge (2+|t|)^{-2}.$$
Accordingly,
$$ (1\pm t)^\gamma_+
\ge 1 +{\frac{\gamma\,\big((1\pm t)^{\gamma_0}_+-1\big)}{\gamma_0}}
+{\frac{\gamma\,(\gamma-\gamma_0)\,
\big((1\pm t)^{\gamma_0}_+-1\big)^2 }{\gamma_0^2
\,( 2+|t| )^{ 2 }
}},$$
and so
$$ (1+ t)^\gamma_+ + (1-t)^\gamma_+ -2
\ge {\frac{\gamma\,\big((1+ t)^{\gamma_0}_+
+(1-t)^{\gamma_0}_+
-2\big)}{\gamma_0}}
+{\frac{\gamma\,(\gamma-\gamma_0)\,\big[
\big((1+t)^{\gamma_0}_+-1\big)^2
+\big((1-t)^{\gamma_0}_+-1\big)^2 \big]}{\gamma_0^2
\,( 2+|t| )^{ 2 }
}}.$$
Hence, we set
$$ \phi(t):=
{\frac{
\big((1+t)^{\gamma_0}_+-1\big)^2
+\big((1-t)^{\gamma_0}_+-1\big)^2 }{ |t|^{2s}
\,( 2+|t| )^{ 2 }
}},$$
we use that~$\gamma=\gamma_0+\epsilon_0>\gamma_0$ and we conclude that
\begin{equation}\label{GH:SP0}
\int_{\R}
\frac{ (1+t)^\gamma_+ + (1-t)^\gamma_+ -2 }{|t|^{2+2s}}\,dt\ge
\frac{\gamma}{\gamma_0} \int_{\R}
\frac{ (1+t)^{\gamma_0}_+ + (1-t)^{\gamma_0}_+ -2 }{|t|^{2+2s}}\,dt
+\frac{\epsilon_0}{\gamma_0} \int_\R \phi(t)\,dt.
\end{equation}
Also, we know (see e.g.~\cite{getoor}) that~$(-\Delta)^{\gamma_0} t_+^{\gamma_0}=0$
for any~$t>0$, therefore, using this formula at~$t=1$
and noticing that~$1+2\gamma_0=2s$, we see that
$$ \int_{\R}
\frac{ (1+t)^{\gamma_0}_+ + (1-t)^{\gamma_0}_+ -2 }{|t|^{2+2s}}\,dt =0.$$
Using this and~\eqref{GH:SP0}, we obtain
$$ \int_{\R}
\frac{ (1+t)^\gamma_+ + (1-t)^\gamma_+ -2 }{|t|^{2+2s}}\,dt\ge
\frac{\epsilon_0}{\gamma_0} \int_\R \phi(t)\,dt,$$
which
implies the desired result.
\end{proof}

Throughout this section, we will consider~$m$ and~$\epsilon_0$
(to be taken appropriately small in the sequel, namely~$\epsilon_0>0$
can be fixed as small as one wishes, and then~$m>0$ is taken to
be small possibly in dependence of~$\epsilon_0$) and~$c_m\in\R$,
and let
\begin{equation}\label{DEF:v1}
v(x_1):=\frac{m\,(x_1+c_m)^\gamma_+ }{\gamma}.\end{equation}
The parameter~$c_m$ will be conveniently chosen in the sequel,
see in particular the following formula~\eqref{J56GFJJ},
but for the moment it is free.
Also, given~$p:=(p_1,p_2)$ with~$p_1\ge1-c_m$ and~$p_2=v(p_1)$,
we consider the tangent line at~$v$ through~$p$, namely
\begin{equation}\label{DEF:v2}
\Lambda(x_1):=v'(p_1) (x_1-p_1)+ v(p_1)
= m\,(p_1+c_m)^{\gamma-1} (x_1-p_1)+
\frac{m\,(p_1+c_m)^\gamma }{\gamma}.\end{equation}
We observe that the tangent line above
meets the $x_1$-axis at the point~$q=(q_1,0)$, with
\begin{equation}\label{q1q1}
q_1:= p_1-\frac{v(p_1)}{v'(p_1)} = p_1- \frac{p_1+c_m}{\gamma}.\end{equation}
We also consider the region~$A$ which lies above
the graph of~$v$ and below the graph of~$\Lambda$
and the region~$B$ which lies above  
the graph of~$\Lambda$ and below the $x_1$-axis, see Figure~\ref{AB:fig}.
More explicitly, we have
\begin{equation}\label{A:Bdef}\begin{split}
& A:= \{ (x_1,x_2) {\mbox{ s.t. }} x_1>q_1 {\mbox{ and }}
v(x_1)<x_2<\Lambda(x_1)\}\\ {\mbox{and }}\quad&
B:= \{ (x_1,x_2) {\mbox{ s.t. }} x_1<q_1 {\mbox{ and }}
\Lambda(x_1)<x_2<0\}.
\end{split}\end{equation}

\begin{figure}
    \centering
    \includegraphics[height=5.9cm]{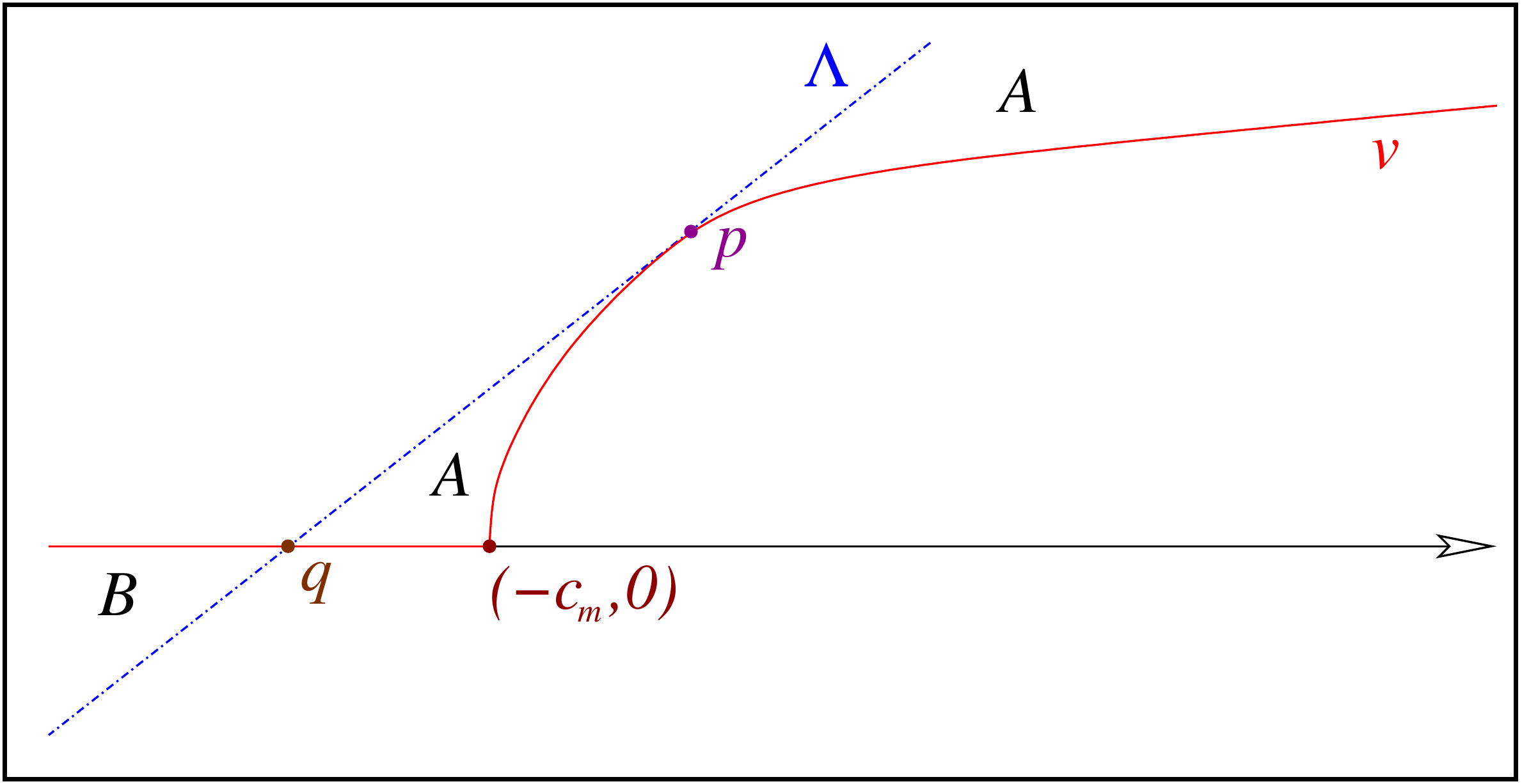}
    \caption{The sets involved in Section~\ref{sec:GROW:R}.}
    \label{AB:fig}
\end{figure}

The first technical result that we need is the following:

\begin{lemma}\label{TECH-LM1}
Let~$\epsilon_0\in(0,1-\gamma_0)$.
There exist~$c$, $c'\in(0,1)$ such that if~$m\in (0,c\epsilon_0]$
then 
\begin{equation}\label{T:SPP:1}
\int_{B} \frac{dy}{|y-p|^{2+2s}} -\int_{A} \frac{dy}{|y-p|^{2+2s}} \ge 
\frac{c' \,\epsilon_0\, m}{(p_1+c_m)^{\frac{1}{2}+s-\epsilon_0}},
\end{equation}
for any~$p:=(p_1,p_2)$ with~$p_1\ge1-c_m$ and~$p_2=v(p_1)$.
\end{lemma}

\begin{proof} First of all, we observe that~$|y-p|\ge |y_1-p_1|$,
therefore
\begin{equation}\label{FI6}
\int_{A} \frac{dy}{|y-p|^{2+2s}} \le 
\int_{A} \frac{dy}{|y_1-p_1|^{2+2s}}
= \int_{q_1}^{+\infty} \frac{\Lambda(y_1)-v(y_1)}{|y_1-p_1|^{2+2s}}\,dy_1 
=:H.\end{equation}
Recalling~\eqref{DEF:v1} and~\eqref{DEF:v2}, we have that
\begin{eqnarray*}
H &=&
\int_{q_1}^{+\infty} \frac{
m\,(p_1+c_m)^{\gamma-1} (y_1-p_1)+
\gamma^{-1} m\,(p_1+c_m)^\gamma_+ -\gamma^{-1}
m\,(y_1+c_m)^\gamma_+  }{|y_1-p_1|^{2+2s}}\,dy_1 
\\ &=& \frac{ m\,(p_1+c_m)^{\gamma} }{\gamma}
\int_{q_1}^{+\infty} \frac{
\gamma\,(p_1+c_m)^{-1}(y_1-p_1)+1-
(p_1+c_m)^{-\gamma}(y_1+c_m)^\gamma_+ }{|y_1-p_1|^{2+2s}}\,dy_1.
\end{eqnarray*}
Now we recall~\eqref{q1q1}
and use the change of variable from the variable~$y_1$
to the variable~$t$ given by
\begin{equation}\label{SU:id}
y_1+c_m = (p_1+c_m)(t+1).\end{equation}
In this way, we obtain that
\begin{equation*} H
=
\frac{ m}{\gamma \,(p_1+c_m)^{1+2s-\gamma}}
\int_{-\frac{1}{\gamma}}^{+\infty} \frac{
\gamma t+1-
(t+1)^\gamma_+ }{|t|^{2+2s}}\,dt=\frac{C_A\, m}{(p_1+c_m)^{1+2s-\gamma}},\end{equation*}
where
\begin{equation*} C_A := 
\int_{-\frac{1}{\gamma}}^{+\infty} \frac{
\gamma t+1-
(t+1)^\gamma_+ }{|t|^{2+2s}}\,dt .\end{equation*}
Therefore, recalling \eqref{FI6}, we conclude that
\begin{equation}\label{JKUFG}
\int_{A} \frac{dy}{|y-p|^{2+2s}} \le
\frac{C_A\, m}{(p_1+c_m)^{1+2s-\gamma}}.\end{equation}
Now we claim that
\begin{equation}\label{POg33h}
{\mbox{if $y\in B$, then~$|y_2-p_2|\le m |y_1-p_1|$.}}
\end{equation}
To prove this, we take~$y\in B$. Then~$\Lambda(y_1)< y_2<0$, therefore,
since~$p_2\ge0$, we have
$$ |y_2-p_2|=p_2-y_2\le p_2-\Lambda(y_1)=v(p_1)-
\Big( v'(p_1) (y_1-p_1)+ v(p_1) \Big) 
\le m(p_1+c_m)^{\gamma-1} |y_1-p_1|.$$
Now we have that~$p_1+c_m\ge 1$, by our assumptions.
Hence, since~$\gamma-1<0$,
we conclude that~$ |y_2-p_2|\le m|y_1-p_1|$,
thus proving~\eqref{POg33h}.

As a consequence of~\eqref{POg33h}, we have that if~$y\in B$
then~$|y-p|\le (1+Cm)|y_1-p_1|$, for some~$C>0$, and therefore
\begin{equation}\label{Y67P}
\int_{B} \frac{dy}{|y-p|^{2+2s}} \ge
(1-Cm) \int_{B} \frac{dy}{|y_1-p_1|^{2+2s}}=(1-Cm)\,I ,\end{equation}
up to renaming~$C>0$, where
$$ I:=\int_{B} \frac{dy}{|y_1-p_1|^{2+2s}}=
\int_{-\infty}^{q_1} \frac{-\Lambda(y_1)}{|y_1-p_1|^{2+2s}}\,dy_1.$$
Recalling the definition of~$H$ in~\eqref{FI6}, we have that
$$ J:= H-I =
\int_{q_1}^{+\infty} \frac{\Lambda(y_1)-v(y_1)}{|y_1-p_1|^{2+2s}}\,dy_1
+
\int_{-\infty}^{q_1} \frac{\Lambda(y_1)}{|y_1-p_1|^{2+2s}}\,dy_1.$$
Accordingly, since~$v(y_1)=0$ if~$y_1\le q_1$, we obtain that
$$ J= \int_{-\infty}^{+\infty} 
\frac{\Lambda(y_1)-v(y_1)}{|y_1-p_1|^{2+2s}}\,dy_1
= \int_{-\infty}^{+\infty} 
\frac{v(p_1)-v(y_1)}{|y_1-p_1|^{2+2s}}\,dy_1,$$
where we have used~\eqref{DEF:v2} in the last identity and the
integrals are taken in the principal value sense.
Hence, we use \eqref{DEF:v1} and the substitution in~\eqref{SU:id},
and we conclude that
$$ J= 
\frac{m}{\gamma}
\int_{-\infty}^{+\infty}
\frac{ (p_1+c_m)^\gamma - (y_1+c_m)^\gamma_+}{|y_1-p_1|^{2+2s}}\,dy_1
=
\frac{m}{\gamma (p_1+c_1)^{1+2s-\gamma}}
\int_{-\infty}^{+\infty}
\frac{ 1- (t+1)^\gamma_+}{|t|^{2+2s}}\,dt =
-\frac{C_B\, m}{(p_1+c_m)^{\frac{1}{2}+s-\epsilon_0}},$$
where
$$ C_B:= \int_{-\infty}^{+\infty}
\frac{ (t+1)^\gamma_+-1}{|t|^{2+2s}}\,dt.$$
{F}rom Lemma~\ref{TOR}, we have that~$C_B \ge c_\star \epsilon_0$,
for some~$c_\star>0$. As a consequence,
$$ I = H-J = 
\frac{(C_A+C_B)\, m}{(p_1+c_m)^{\frac{1}{2}+s-\epsilon_0}}
\ge \frac{(C_A+c_\star\epsilon_0)\, m}{(p_1+c_m)^{\frac{1}{2}+s-\epsilon_0}},$$
and so, by~\eqref{Y67P}
\begin{equation*}
\int_{B} \frac{dy}{|y-p|^{2+2s}} \ge
\frac{(1-Cm)(C_A+c_\star\epsilon_0)\, m}{(p_1+c_m)^{\frac{1}{2}+s-\epsilon_0}}
.\end{equation*}
Putting together this and~\eqref{JKUFG}, we obtain that
$$
\int_{B} \frac{dy}{|y-p|^{2+2s}} - \int_{A} \frac{dy}{|y-p|^{2+2s}}\ge
\frac{\big[ (1-Cm)(C_A+c_\star\epsilon_0)-C_A\big]\, m}{(p_1+c_m)^{\frac{1}{2}+s-\epsilon_0}},$$
which implies the desired result.
\end{proof}

\begin{figure}
    \centering
    \includegraphics[height=5.5cm]{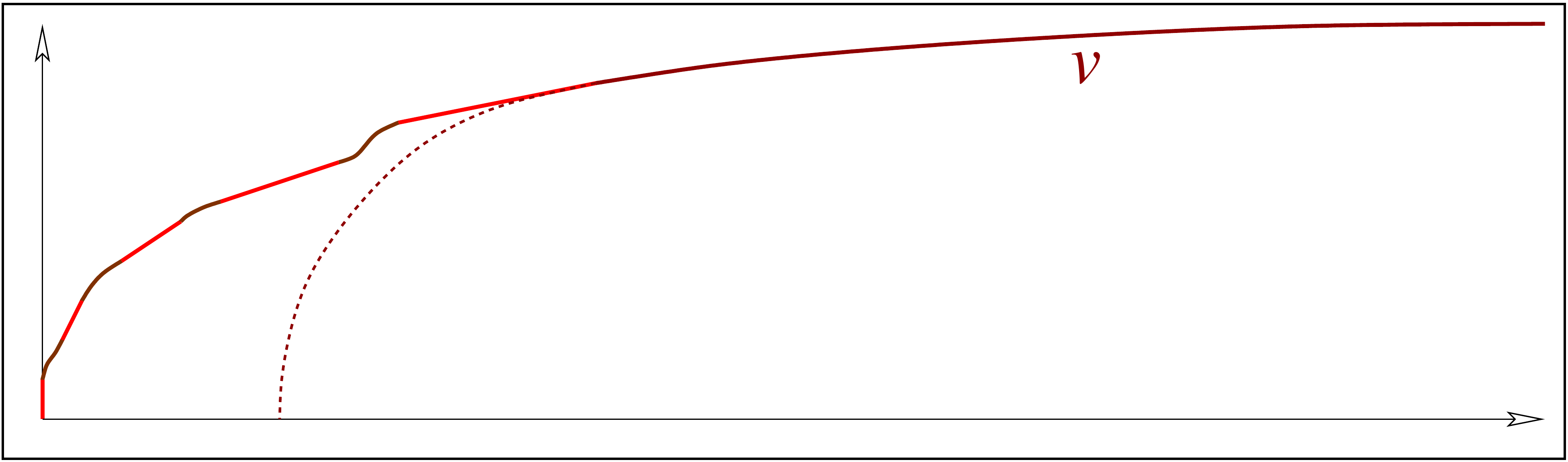}
    \caption{The barrier constructed in Proposition~\ref{BAR-gr}.}
    \label{SK}
\end{figure}

Now we are in the position of
improving the behavior at infinity of the barrier constructed in
Proposition~\ref{IT:BARR}. The idea is to ``glue'' the barrier
of Proposition~\ref{IT:BARR} with the graph of
the ``right'' power function at infinity. The construction
is sketched in Figure~\ref{SK} and the precise result obtained
is the following:

\begin{proposition}\label{BAR-gr}
Let~$\epsilon_0\in(0,1-\gamma_0)$.
There exists~$c>0$ such that if~$m\in (0,c\epsilon_0]$,
then the following statement holds.

There exist~$a_m>0$, $d_m>1>\alpha_m>0$, $c_m\in\R$
and a set~$E_m\subset\R^2$
with $(\partial E_m)\cap \{x_2>0\}$
of class~$C^{1,1}$
and such that:
\begin{eqnarray*}
&& E_m \cap \{x_1 < 0\} = (-\infty,0)\times(-\infty,0), \\
&& E_m\supseteq \R\times (-\infty,0),\\
&& E_m\supseteq (0,+\infty)\times (-\infty,a_m],\\
&& E_m\cap \{\alpha_m\le x_1 \le d_m\} =
\{ x_2<v'(d_m)(x_1-d_m)+v(d_m),\;\alpha_m\le x_1 \le d_m\}\\
{\mbox{and }}&& E_m\cap \{x_1 >d_m\} = \{x_2<v(x_1), \;x_1 >d_m\}
,\end{eqnarray*}
where~$v$ was introduced in~\eqref{DEF:v1}.
Moreover, there exist~$c'\in(0,1)$ and~$N>1$ such that
\begin{equation}\label{PLFG}
\int_{\R^2} \frac{\chi_{E_m}(y)-\chi_{E_m^c}(y)}{|y-p|^{2+2s}}\,dy
\ge \frac{c'\,\epsilon_0\,m}{|p|^{\frac{1}{2}+s-\epsilon_0}},\end{equation}
for any~$p\in (\partial E_m)\cap \{x_1>\frac{d_m}{N}\}$, and
\begin{equation}\label{PLFGbar}
\int_{\R^2} \frac{\chi_{E_m}(y)-\chi_{E_m^c}(y)}{|y-p|^{2+2s}}\,dy
\ge \frac{c'\,m}{d_m^{1-\gamma}\, |p|^{2s}},\end{equation}
for any~$p\in (\partial E_m)\cap \{x_1\in\left(0,\frac{d_m}{N}\right]\}$.
\end{proposition}

\begin{proof} We use Proposition~\ref{IT:BARR} with a large~$K$.
In this way, we may suppose that~$\ell_K\ge K$ is as large as
we wish, while~$q_K\le K^{-1}$ is as small as we wish.
We fix~$N>0$, to be chosen
appropriately large (independently on~$K$) and we set
\begin{equation}\label{UIj:M}
\begin{split}
& d_m:=N^2\\
{\mbox{and }}\quad& m:= 
\ell_K^{-1} \big( \gamma\,(d_m+q_K)\big)^{1-\gamma}
.\end{split}\end{equation}
We stress that~$m>0$ is small when~$K$ is large, since
$$ m\le K^{-1} \big( \gamma\,(N^2+K^{-1})\big)^{1-\gamma},$$
that is small when~$K$ is large (much larger than the fixed~$N$).
Hence Proposition~\ref{IT:BARR} provides a set, say~$F_m$,
whose boundary agrees with a straight line~$\lambda_m$
of the form~$x_2=\ell_K^{-1}(x_1+q_K)$
when~$x_1\ge \alpha_m$, for suitable~$q_K\in[0,K^{-1}]$ and~$\alpha_m>0$.

Now we join such a straight line with the function~$v$ defined
in~\eqref{DEF:v1}, at the point~$(d_m,v(d_m))$, with~$\beta_m:=d_m-\alpha_m$
suitably large. To this goal, we define
\begin{equation}\label{J56GFJJ}
c_m:=(\gamma-1)d_m+\gamma q_K.\end{equation}
Notice that
\begin{equation}\label{HJ:LKJ}
d_m+c_m =\gamma(d_m+q_K).
\end{equation}
This and~\eqref{UIj:M} give that
$$ v(d_m)=\frac{m\,(d_m+c_m)^\gamma_+ }{\gamma}
= \frac{m\,\big(\gamma (d_m+q_K)\big)^\gamma }{\gamma}
=
\ell_K^{-1} \big( \gamma\,(d_m+q_K)\big)^{1-\gamma}
\cdot\frac{\big(\gamma (d_m+q_K)\big)^\gamma }{\gamma}
=\ell_K^{-1}(d_m+q_K),$$
which says that~$v$ meets the straight line~$\lambda_m$
at the point~$(d_m,v(d_m))$. 

Also, by~\eqref{UIj:M} and~\eqref{HJ:LKJ}, we see that
$$ v'(d_m)= m(d_m+c_m)^{\gamma-1} = 
\ell_K^{-1} \big( \gamma\,(d_m+q_K)\big)^{1-\gamma}\cdot
\big( \gamma\,(d_m+q_K)\big)^{\gamma-1} =\ell_K^{-1},$$
therefore $v$ and~$\lambda_m$ have the same slope at
the meeting point~$(d_m,v(d_m))$.
Therefore, the set~$E_m$ which coincides with~$F_m$
when~$\{x_1\le d_m\}$ and with the subgraph of~$v$ when~$\{x_1>d_m\}$
satisfy the geometric properties listed in the statement of
Proposition~\ref{BAR-gr}, and it only remains to prove~\eqref{PLFG}
and~\eqref{PLFGbar}.

For this scope, we first consider the case in which~$p_1\ge d_m$.
Then, we take~$\Lambda$ as in~\eqref{DEF:v2}
and $A$ and~$B$ as in~\eqref{A:Bdef}. Let also~$T$ be the subgraph
of~$\Lambda$. Then, by symmetry
$$ \int_{\R^2} \frac{\chi_{T}(y)-\chi_{T^c}(y)}{|y-p|^{2+2s}}\,dy=0.$$
Notice that~$T\setminus E_m\subseteq A$
and~$E_m\setminus T \supseteq B$, therefore
\begin{eqnarray*}
&& \int_{\R^2} \frac{\chi_{E_m}(y)-\chi_{E_m^c}(y)}{|y-p|^{2+2s}}\,dy\\
&=& 
\int_{\R^2} \frac{\chi_{E_m}(y)-\chi_{E_m^c}(y)-\chi_{T}(y)+\chi_{T^c}(y)
}{|y-p|^{2+2s}}\,dy\\
&=&2
\int_{\R^2} \frac{\chi_{E_m\setminus T}(y) - \chi_{T\setminus E_m}(y)
}{|y-p|^{2+2s}}\,dy\\
&\ge& 2
\int_{\R^2} \frac{\chi_{B}(y) - \chi_{A}(y)
}{|y-p|^{2+2s}}\,dy\\
&=& 2\left(
\int_{B} \frac{dy}{|y-p|^{2+2s}}-\int_{A} \frac{dy}{|y-p|^{2+2s}}
\right).
\end{eqnarray*}
Notice also that
\begin{equation}\label{HJLKK9}
1-c_m = 1 -(\gamma-1)d_m-\gamma q_K\le 1-\gamma d_m +d_m\le d_m\end{equation}
thanks to \eqref{J56GFJJ}
and \eqref{UIj:M}. Hence, in this case, $p_1\ge d_m\ge 1-c_m$,
and so the assumptions of Lemma~\ref{TECH-LM1} are fulfilled.
Therefore, by~\eqref{T:SPP:1},
\begin{equation}\label{JK:PLL}
\int_{\R^2} \frac{\chi_{E_m}(y)-\chi_{E_m^c}(y)}{|y-p|^{2+2s}}\,dy
\ge
\frac{c' \,\epsilon_0\, m}{(p_1+c_m)^{\frac{1}{2}+s-\epsilon_0}},
\end{equation}
for some~$c'>0$. 
Now we notice that, by~\eqref{J56GFJJ} and \eqref{UIj:M},
$$ p_1+c_m = p_1+ (\gamma-1)d_m+\gamma q_K\le 2 p_1 \le 2|p|.$$
Using this and~\eqref{JK:PLL}, we see that~\eqref{PLFG}
holds true in this case.

Hence, it remains to prove~\eqref{PLFG} and~\eqref{PLFGbar}
when~$p_1\in(0,d_m)$.
In this case, we use that, by
Proposition~\ref{IT:BARR},
$$ \int_{\R^2} \frac{\chi_{F_m}(y)-\chi_{F_m^c}(y)}{|y-p|^{2+2s}}\,dy
\ge \frac{\bar c}{\ell_K\,|p|^{2s}},$$
for some~$\bar c>0$.
Also~$F_m\setminus E_m$ coincides with the portion comprised above
the graph of~$v$ and below the straight line~$\lambda_m$, that is 
$$ G:=\{ x_1> d_m,\; v(x_1)< x_2< v'(d_m)(x_1-d_m)+v(d_m)\},$$
while~$E_m\setminus F_m$ is empty. Therefore
\begin{equation}\label{PL:PL1}
\begin{split}
& \frac{\bar c}{\ell_K\,|p|^{2s}}
-
\int_{\R^2} \frac{\chi_{E_m}(y)-\chi_{E_m^c}(y)}{|y-p|^{2+2s}}\,dy
\le
\int_{\R^2} \frac{\chi_{F_m}(y)-\chi_{F_m^c}(y)-
\chi_{E_m}(y)+\chi_{E_m^c}(y)}{|y-p|^{2+2s}}\,dy
\\ &\quad=
2\int_{G} \frac{dy}{|y-p|^{2+2s}}
\le  2\int_{G} \frac{dy}{|y_1-p_1|^{2+2s}}
\\ &\quad= 2\int_{d_m}^{+\infty} \frac{v'(d_m)(y_1-d_m)+v(d_m)-v(y_1)
}{ |y_1-p_1|^{2+2s} }\,dy_1.\end{split}\end{equation}
Now, we distinguish the cases~$p_1\in \left(0,\frac{d_m}{N}\right)$
and~$p_1\in\left[\frac{d_m}{N},d_m\right)$.

If~$p_1\in \left(0,\frac{d_m}{N}\right)$, we use~\eqref{PL:PL1}
and observe that~$v(y_1)\ge v(d_m)$ if~$y_1\ge d_m$, to conclude that
\begin{eqnarray*}
&&\frac{\bar c}{\ell_K\,|p|^{2s}}
-
\int_{\R^2} \frac{\chi_{E_m}(y)-\chi_{E_m^c}(y)}{|y-p|^{2+2s}}\,dy
\le
2 v'(d_m) \int_{d_m}^{+\infty} \frac{y_1-d_m}{ |y_1-p_1|^{2+2s} }\,dy_1
\\ &&\quad\le
2 v'(d_m) \int_{d_m}^{+\infty} \frac{dy_1}{ (y_1-p_1)^{1+2s} }
\le\frac{Cm\, (d_m+c_m)^{\gamma-1} }{(d_m-p_1)^{2s}}
\\ &&\quad\le \frac{Cm\, (d_m+c_m)^{\gamma-1} }{d_m^{2s}},\end{eqnarray*}
up to renaming constants. Therefore, recalling~\eqref{UIj:M}
and~\eqref{HJ:LKJ},
\begin{equation}\label{JK:PJgH}
\begin{split}
& \int_{\R^2} \frac{\chi_{E_m}(y)-\chi_{E_m^c}(y)}{|y-p|^{2+2s}}\,dy\ge
\frac{\bar c\,m}{ \big( \gamma\,(d_m+q_K)\big)^{1-\gamma}\,|p|^{2s}} - 
\frac{Cm }{ (d_m+c_m)^{1-\gamma}\,
d_m^{2s}} \\&\qquad= \frac{m}{ (d_m+c_m)^{1-\gamma} }\left(\frac{\bar c}{|p|^{2s}}-
\frac{C}{d_m^{2s}}
\right).\end{split}\end{equation}
Now we observe that, when~$p_1\le \frac{d_m}{N}$, we have
that~$p_2\le 1 +\ell_K^{-1}\left( \frac{d_m}{N}+q_K\right)\le
2+\frac{d_m}{N}\le \frac{d_m}{N^{1/2}}$, and so~$|p|\le \frac{d_m}{N^{1/4}}$.
Therefore
$$ \frac{C}{d_m^{2s}} \le \frac{C}{N^{s/2}\,|p|^{2s}}\le
\frac{\bar c}{2\,|p|^{2s}},$$
if~$N$ is large enough (independently on~$m$ and~$K$).
This and~\eqref{JK:PJgH} imply that
\begin{equation*} \int_{\R^2} \frac{\chi_{E_m}(y)-\chi_{E_m^c}(y)}{|y-p|^{2+2s}}\,dy\ge
\frac{m\,\bar c}{ 2(d_m+c_m)^{1-\gamma} |p|^{2s}}.\end{equation*}
By recalling~\eqref{HJ:LKJ},
we see that the latter estimate implies~\eqref{PLFGbar}
in this case.

\begin{figure}
    \centering
    \includegraphics[height=5.9cm]{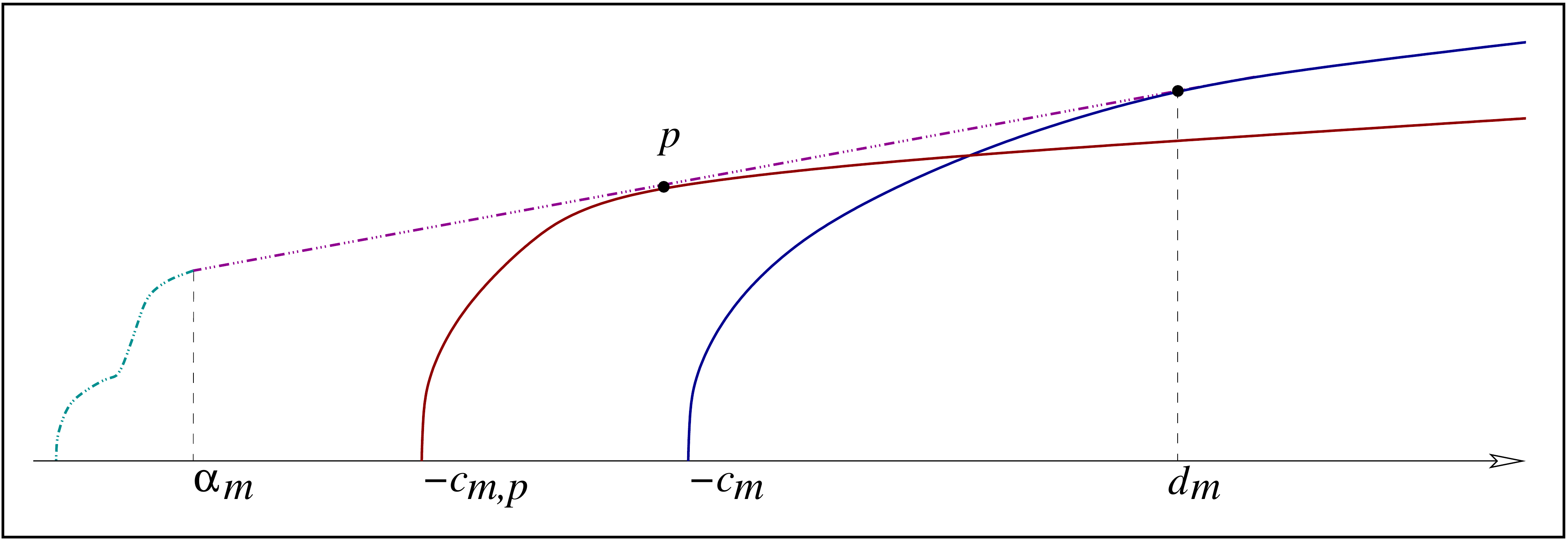}
    \caption{A power-like function tangent at~$p\in \partial E_m$,
with $p_1\in\left[\frac{d_m}{N},d_m\right)$.}
    \label{VTAN}
\end{figure}

It remains to prove~\eqref{PLFG} when~$p_1\in\left[\frac{d_m}{N},d_m\right)$.
In this case, we argue like this.
For any $p=(p_1,p_2)\in \partial F_m$
with~$p_1\in\left[\frac{d_m}{N},d_m\right)$, we have~$p_2=
v'(d_m)(p_1-d_m)+v(d_m)$,
and we
define~$v_p$
the power function whose graph passes
through~$p$ and tangent to the line~$\{ x_2 = v'(d_m)(x_1-d_m)+v(d_m)\}$
at~$p$, see Figure~\ref{VTAN}.
Explicitly, we define
\begin{eqnarray*}
&& v_p(x_1):=\frac{m_p \,(x_1+c_{m,p} )^\gamma_+ }{\gamma},\\
{\mbox{with }}&& m_p:=(\gamma p_2)^{1-\gamma} m^\gamma
(d_m+c_m)^{\gamma(\gamma-1)}\\
{\mbox{and }}&& c_{m,p}:=\frac{\gamma p_2}{m\,(d_m+c_m)^{\gamma-1}}-p_1.
\end{eqnarray*}
We remark that
$$ v_p (p_1)=p_2 {\mbox{ and }}
v_p'(p_1) = v'(d_m).$$
Since~$p_2< v(d_m)=m\,\gamma^{-1}(d_m+c_{m} )^\gamma$, we have that
\begin{equation}\label{LP:mp small}
m_p<\big(
m\,(d_m+c_{m} )^\gamma \big)^{1-\gamma} m^\gamma
(d_m+c_m)^{\gamma(\gamma-1)}=m.
\end{equation}
Moreover~$p_2=v'(d_m)(p_1-d_m)+v(d_m)=
m\,(d_m+c_m)^{\gamma-1} (p_1-d_m) + m\,\gamma^{-1} (d_m+c_m)^{\gamma}$,
therefore
\begin{equation}\label{HJLKK9-cmp:PRE}
\begin{split} c_{m,p}\;&=\,\frac{\gamma 
m\,(d_m+c_m)^{\gamma-1} (p_1-d_m) + m\,(d_m+c_m)^{\gamma}
}{m\,(d_m+c_m)^{\gamma-1}}-p_1
\\&= \,\gamma(p_1-d_m)+d_m+c_m-p_1= (1-\gamma)(d_m-p_1)+c_m.\end{split}\end{equation}
Hence, since~$p_1<d_m$,
\begin{equation}\label{HJLKK9-cmp}
c_{m,p}> \,c_m
.\end{equation}
Also, from~\eqref{J56GFJJ} and~\eqref{HJLKK9-cmp:PRE},
\begin{equation} \label{KL:PP-2}
c_{m,p} = (1-\gamma)(d_m-p_1)+
(\gamma-1)d_m+\gamma q_K = -(1-\gamma)\,p_1+\gamma q_K.\end{equation}
Therefore, since~$p_1\ge\frac{d_m}{N}=N$,
\begin{equation}\label{HJ:pl}
c_{m,p} \le -(1-\gamma)\,N+\gamma q_K\le -(1-\gamma)\,N+1
< -1 \le -\alpha_m,
\end{equation}
provided that~$N$ is large enough.

Furthermore, using again~\eqref{KL:PP-2},
\begin{equation}\label{PKFDc}
p_1+c_{m,p} = \gamma p_1+\gamma q_K\ge \frac{\gamma d_m}{N}=\gamma N\ge 1.
\end{equation}
In addition, 
$$ m_p (p_1+c_{m,p})_+^{\gamma-1}=
v_p'(p_1) = v'(d_m) = m (d_m+c_m)_+^{\gamma-1},$$
therefore, by~\eqref{HJ:LKJ}
and~\eqref{KL:PP-2},
\begin{equation}\label{SRTG-9}
\begin{split}
& \frac{m_p}{m} =\frac{(p_1+c_{m,p})_+^{1-\gamma} }{ (d_m+c_m)_+^{1-\gamma} }
=\frac{(\gamma p_1+\gamma q_K)_+^{1-\gamma} }{ (\gamma(d_m+q_K))_+^{1-\gamma} }
= \frac{(p_1+q_K)_+^{1-\gamma} }{ (d_m+q_K)_+^{1-\gamma} }
\\ 
&\qquad\ge \frac{\left(\frac{d_m}{N}\right)_+^{1-\gamma} }{ 
(2d_m)_+^{1-\gamma} }
= \frac{1}{(2N)^{1-\gamma}}.
\end{split}\end{equation}
Now we claim that
\begin{equation}\label{OILJ:889}
{\mbox{if~$x_1\ge d_m$, then $v_p(x_1)\le v(x_1)$.}}\end{equation}
To prove this,
we use~\eqref{HJLKK9}
and~\eqref{HJLKK9-cmp} to see that
$$ x_1+c_{m,p}\ge x_1+c_m\ge d_m+c_m\ge1,$$ therefore
$$ \psi(x_1):=\gamma\,\big(v_p(x_1)-v(x_1)\big)=
{m_p \,(x_1+c_{m,p} )^\gamma }-
{m \,(x_1+c_{m} )^\gamma }.$$
Also, $v_p$ is concave, therefore
\begin{eqnarray*}
&& v_p(d_m) \le v_p (p_1) + v_p'(p_1)(d_m-p_1)
=p_2+v'(d_m)(d_m-p_1)\\
&&\qquad=v'(d_m)(p_1-d_m)+v(d_m)+v'(d_m)(d_m-p_1)=v(d_m).\end{eqnarray*}
As a consequence, $ \psi(d_m)\le 0$. Moreover, for any~$x_1\ge d_m$,
$$ \psi'(x_1)=
{m_p \gamma \,(x_1+c_{m,p} )^{\gamma-1} }-
{m \gamma \,(x_1+c_{m} )^{\gamma-1} }
\le m\gamma\,\big[ {(x_1+c_{m,p} )^{\gamma-1} }-
{(x_1+c_{m} )^{\gamma-1} } \big]\le0,$$
thanks to~\eqref{LP:mp small} and~\eqref{HJLKK9-cmp}. {F}rom
these considerations, we obtain that~$\psi\le0$
in~$[d_m,+\infty)$, which proves~\eqref{OILJ:889}.

Also, by concavity,
\begin{equation}\label{COPP}
\begin{split}
& {\mbox{if $x_1\in[-c_{m,p},d_m]$, then}}\\
&\qquad v_p(x_1)\le
v_p'(p_1)(x_1-p_1)+v_p(p_1)=
v'(d_m)(x_1-p_1)+p_2
=v'(d_m)(x_1-d_m)+v(d_m).\end{split}\end{equation}
Now we claim that
\begin{equation}\label{is co}
{\mbox{the subgraph of~$v_p$ is contained in~$E_m$.}}
\end{equation}
To check this, let~$x=(x_1,x_2)$ be such that~$x_2< v_p(x_1)$.
Then, if~$x_1 < -c_{m,p}$ then~$v_p(x_1)=0$ and so~\eqref{is co}
plainly follows. If~$x_1\in [-c_{m,p},d_m]$,
then~\eqref{is co} is implied by~\eqref{HJ:pl} and \eqref{COPP}.
Finally, if~$x_1>d_m$, then~\eqref{is co}
is a consequence of~\eqref{OILJ:889}.

Hence, we define~$S:= \{x_2<v_p(x_1)\}$,
we use~\eqref{is co}
and Lemma~\ref{TECH-LM1} (which can be exploited in this
framework with the power-like function~$v_p$,
thanks to~\eqref{PKFDc}) and we obtain that
\begin{equation}\label{LGH:LK}
\int_{\R^2} \frac{\chi_{E_m}(y)-\chi_{E_m^c}(y)}{|y-p|^{2+2s}}\,dy
\ge
\int_{\R^2} \frac{\chi_{S}(y)-\chi_{S^c}(y)}{|y-p|^{2+2s}}\,dy
\ge \frac{c' \,\epsilon_0\, m_p}{(p_1+c_{m,p})^{\frac{1}{2}+s-\epsilon_0}},\end{equation}
for some~$c'>0$. Now we recall~\eqref{KL:PP-2} and we see that~$
p_1+c_{m,p}\le p_1\le |p|$. Using this and~\eqref{SRTG-9}
(notice that~$N$ has now been fixed),
we obtain~\eqref{PLFG} if~$p_1\in\left[\frac{d_m}{N},d_m\right)$
as a consequence of~\eqref{LGH:LK}.

This completes the proof of \eqref{PLFG} in all cases and finishes the
proof of Proposition~\ref{BAR-gr}.
\end{proof}

\section{Construction of compactly supported barriers}\label{SEC:CMP:SUPP}

In this section, we construct a suitable barrier for the fractional
mean curvature equation in the plane which is flat and horizontal
outside a vertical slab, and whose geometric properties
inside the slab are under control. Roughly speaking,
we will take the barrier constructed in Proposition~\ref{BAR-gr}
and a reflected version of it and join it smoothly in the middle.
The effect of this surgery is negligible at
the points of the barrier that are near the horizontal part,
and give a bounded contribution in the middle.

\begin{figure}
    \centering
    \includegraphics[height=3.2cm]{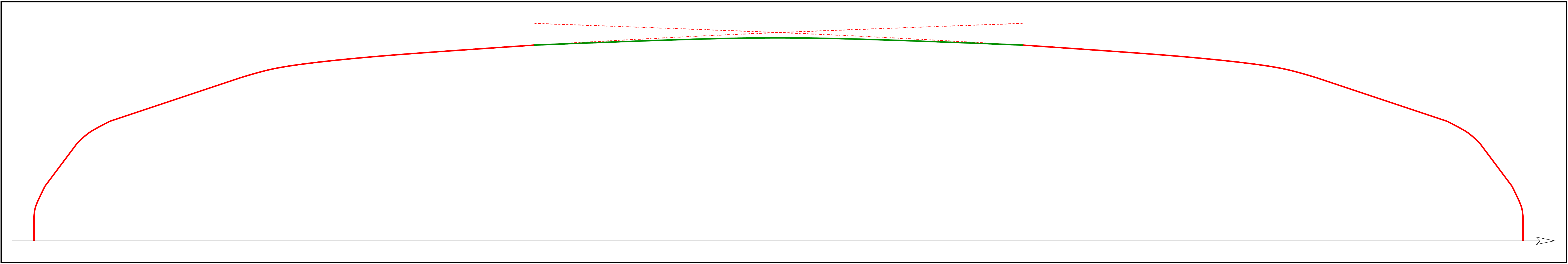}
    \caption{The barrier constructed in Proposition~\ref{LK:PO}.}
    \label{LLM}
\end{figure}

This barrier is described in Figure~\ref{LLM}
and the precise result obtained is the following:

\begin{proposition}\label{LK:PO}
Let~$\epsilon_0\in(0,1-\gamma_0)$.
There exists~$m_{\epsilon_0}>0$ such that if~$m\in (0,\,m_{\epsilon_0}]$
then the following statement holds.

There exist~$a_m>0$, $L_m>A_m>d_m>1$, $c_m\in\R$, $C_\star>0$
and a set~$F_m\subset\R^2$
with $(\partial F_m)\cap \{x_2>0\}$
of class~$C^{1,1}$ and
such that:
\begin{eqnarray*}
&& F_m \cap \{x_1 < 0\} = (-\infty,0)\times(-\infty,0), \\
&& F_m\supseteq \R\times (-\infty,0),\\
&& F_m\supseteq (0,L_m+1)\times (-\infty,a_m],\\
&& F_m \subseteq \{ x_2\le C_\star m\, L_m^{\frac{1}{2}+s+\epsilon_0} \}\\
{\mbox{and }}&& 
F_m\cap \{d_m<x_1<L_m\} = \{x_2<v(x_1), \;d_m<x_1 <L_m\}
,\end{eqnarray*}
where~$v$ was introduced in~\eqref{DEF:v1}.
In addition, one can suppose that
\begin{equation}\label{L emme}
L_m = 10 A_m \ge 2+m^{-1} + e^{\frac{1}{a_m}}.
\end{equation}
Moreover, the set~$F_m$ is even symmetric with respect to
the vertical axis~$\{x_1=L_m+1\}$, 
and
there exists $C'>0$ such that
\begin{equation}\label{PLFG-SYM}
\int_{\R^2} \frac{\chi_{F_m}(y)-\chi_{F_m^c}(y)}{|y-p|^{2+2s}}\,dy
\ge 0,\end{equation}
for any~$p\in (\partial F_m)\cap \{x_1\in (0,A_m)\}$,
and
\begin{equation}\label{PLFG-SYM:2}
\int_{\R^2} \frac{\chi_{F_m}(y)-\chi_{F_m^c}(y)}{|y-p|^{2+2s}}\,dy
\ge -\frac{C' m^{2s}}{L_m^{\frac{1}{2}+s-\epsilon_0}},\end{equation}
for any~$p\in (\partial F_m)\cap \{x_1\in\left[A_m, L_m+1\right]\}$.
\end{proposition}

\begin{proof} We let~$E_m$ be the set constructed in
Proposition~\ref{BAR-gr}. Let~$E'_m$
be the even reflection
of~$E_m$ with respect to
the vertical axis~$\{x_1=L_m+1\}$.
We take a smooth function~$w:[L_m,\,L_m+2]\to [v(L_m),\,Cm \,L_m^\gamma]$
that is even with respect to~$\{x_1=L_m+1\}$, with~$w(L_m)=v(L_m)$
and such that
its derivatives agree with the ones of~$v$ at the point~$L_m$.
The set~$F_m$ is then defined as
$$ \big( E_m\cap \{x_1 \le L_m\} \big)
\;\cup\; 
\big\{ x_2<w(x_1),\; x_1\in (L_m,L_m+2)\big\}
\;\cup\; 
\big( E_m'\cap \{x_1 \ge L_m+2\} \big).$$
For completeness, let us describe the above function~$w$ explicitly.
One takes an odd function~$\tau\in C^\infty(\R,[-1,1])$ such that~$\tau=-1$
in~$(-\infty,-1]$ and~$\tau=1$ in~$[1,+\infty)$ and defines~$w$ by
$$ w(x_1):=\frac{\big(1-\tau(x_1-L_m-1)\big)\,v(x_1)+
\big(1+\tau(x_1-L_m-1)\big)\,v(2L_m+2-x_1)
}{2}.$$
Then~$w(L_m+1+x_1)=w(L_m+1-x_1)$,
hence~$w$ is even with respect to~$\{x_1=L_m+1\}$.
The set~$F_m$ has the desired geometric properties, so
it remains to prove~\eqref{PLFG-SYM}
and~\eqref{PLFG-SYM:2}. For this, we take~$L_m=10 A_m$
appropriately large. In particular, we suppose that~$L_m\ge c_m+2A_m$,
and therefore, for any~$y_1\in [L_m,+\infty)$ and~$p_1\in(0,A_m)$
we have that~$ y_1+c_m\le 2(y_1-p_1)$, and so, by~\eqref{DEF:v1},
$$ v(y_1)=\frac{m\,(y_1+c_m)^\gamma_+ }{\gamma}\le
\frac{2^\gamma m\,(y_1-p_1)^\gamma }{\gamma}.$$
We also notice that~$E_m\setminus F_m\subseteq\{x_1>L_m,\;0<x_2<v(x_1)\}$.
Therefore, for every~$p\in (\partial F_m)\cap \{x_1\in (0,A_m)\}$,
\begin{equation}\label{CVXKK}
\begin{split} &
\int_{\R^2} \frac{\chi_{E_m}(y)-\chi_{E_m^c}(y)}{|y-p|^{2+2s}}\,dy
-\int_{\R^2} \frac{\chi_{F_m}(y)-\chi_{F_m^c}(y)}{|y-p|^{2+2s}}\,dy\\
&\qquad=
2\int_{\R^2} \frac{
\chi_{E_m\setminus F_m}(y)-
\chi_{F_m\setminus E_m}(y)
}{|y-p|^{2+2s}}\,dy
\le
2\int_{ E_m\setminus F_m }\frac{dy}{ |y-p|^{2+2s} }
\\ &\qquad\le 2\int_{L_m}^{+\infty} \frac{v(y_1)\,dy_1}{ (y_1-p_1)^{2+2s} }
\le Cm \int_{L_m}^{+\infty} (y_1-p_1)^{\gamma-2-2s}
= \frac{Cm}{(L_m-p_1)^{1+2s-\gamma}}\\
&\qquad\le \frac{Cm}{L_m^{1+2s-\gamma}} =
\frac{Cm}{L_m^{\frac{1}{2}+s-\epsilon_0}},
\end{split}\end{equation}
up to changing the names of the constant~$C>0$ line after line.
Hence, recalling~\eqref{PLFG},
$$ \int_{\R^2} \frac{\chi_{F_m}(y)-\chi_{F_m^c}(y)}{|y-p|^{2+2s}}\,dy\ge
\frac{c'\,\epsilon_0\,m}{|p|^{\frac{1}{2}+s-\epsilon_0}}
-\frac{Cm}{L_m^{\frac{1}{2}+s-\epsilon_0}}
\ge \frac{c'\,\epsilon_0\,m}{2\,|p|^{\frac{1}{2}+s-\epsilon_0}}$$
for any~$p\in (\partial F_m)\cap \{x_1\in \left(\frac{d_m}{N},A_m\right)\}$,
as long as~$L_m$ is large enough (possibly in dependence of~$
\sup_{q_1\in(0,A_m)} |q|$). This establishes~\eqref{PLFG-SYM}
if~$p\in (\partial F_m)\cap \{x_1\in \left(\frac{d_m}{N},A_m\right)\}$.

If instead~$p\in (\partial F_m)\cap \{x_1\in \left(0,\frac{d_m}{N}\right]\}$,
we use~\eqref{CVXKK}
and~\eqref{PLFGbar} to obtain that
$$ \int_{\R^2} \frac{\chi_{F_m}(y)-\chi_{F_m^c}(y)}{|y-p|^{2+2s}}\,dy
\ge \frac{c'\,m}{d_m^{1-\gamma}\, |p|^{2s}} - 
\frac{Cm}{L_m^{\frac{1}{2}+s-\epsilon_0}}
\ge 
\frac{c'\,N^{2s}\,m}{d_m^{1-\gamma}\, d_m^{2s}} -
\frac{Cm}{L_m^{\frac{1}{2}+s-\epsilon_0}}
=\frac{c'\,N^{2s}\,m}{d_m^{ {\frac{1}{2}+s-\epsilon_0} }} -
\frac{Cm}{L_m^{\frac{1}{2}+s-\epsilon_0}}\ge0
,$$
as long as~$L_m$ is large enough, and this proves~\eqref{PLFG-SYM}
also in this case.

Now we prove~\eqref{PLFG-SYM:2}.
For this, we take~$p\in (\partial F_m)\cap \{x_1\in\left[A_m, L_m+1\right)\}$.
By~\eqref{DEF:v1}, the curvature of~$F_m$ at~$p$
is bounded (in absolute value) by~$C m L_m^{\gamma-2}$.
Hence (see Lemma~3.1 in~\cite{nostro}, applied 
here with~$\lambda:=L_m^{\gamma-1}$
and~$R:=m^{-1} L^{2-\gamma}$, so that~$\lambda R= \frac{L_m}{m}$,
and canceling the contribution coming from
the tangent line) one obtains that
\begin{equation}\label{POlfG-1}
\left|\int_{B_{\frac{L_m}{m}} (p)} \frac{\chi_{F_m}(y)-
\chi_{F_m^c}(y)}{|y-p|^{2+2s}}\,dy\right|\le C \big( L_m^{\gamma-1}\big)^{1-2s}
\big( m^{-1} L_m^{2-\gamma}\big)^{-2s}
= C m^{2s} L_m^{\gamma-1-2s}=\frac{C m^{2s}}{ L_m^{\frac12+s-\epsilon_0}}
,\end{equation}
for some~$C>0$, possibly varying from step to step. 

Moreover, to compute the contribution coming from
outside~$B_{\frac{L_m}{m}}(p)$, we can compare the set~$F_m$
with the horizontal line passing through~$p$.
Notice indeed that~$F_m\setminus B_{\frac{L_m}{m}}(p)=
\{x_2<0\}\setminus B_{\frac{L_m}{m}}(p)$.
Thus, since~$p_2$ is controlled by~$C m L_m^\gamma$
\begin{equation*}\begin{split}
& \left|\int_{\R^2\setminus B_{ \frac{L_m}{m} }(p)} \frac{\chi_{F_m}(y)-
\chi_{F_m^c}(y)}{|y-p|^{2+2s}}\,dy\right|
\le 2 \int_{\{0<y_2<C m L_m^\gamma\}\setminus B_{ \frac{L_m}{m} (p)}} \frac{dy}{|y-p|^{2+2s}}
\\ &\qquad\le C m L_m^\gamma \int_{ \{ |y_1-p_1|\ge L_m\} } \frac{dy_1}{|y_1-p_1|^{2+2s}}
= C m L_m^{\gamma-1-2s} =
\frac{C m}{ L_m^{\frac12+s-\epsilon_0}}
.\end{split}\end{equation*}
up to renaming~$C>0$. This and~\eqref{POlfG-1} imply~\eqref{PLFG-SYM:2},
as desired.
\end{proof}

By scaling Proposition~\ref{LK:PO},
one obtains the following result:

\begin{corollary}\label{COR 2.2-BAR}
Fix~$\epsilon_0>0$ arbitrarily small.
There exist an infinitesimal sequence of
positive~$\delta$'s and sets~$H_\delta\subseteq\R^2$,
with $(\partial H_\delta)\cap \{x_2>0\}$
of class~$C^{1,1}$,
that are even symmetric with respect to the axis~$\{x_1=0\}$ and
satisfy the following properties:
\begin{eqnarray*}
&& H_\delta \cap \{x_1 < -1\} = (-\infty,-1)\times(-\infty,0), \\
&& H_\delta \supseteq \R\times (-\infty,0),\\
&& H_\delta \supseteq (-1,1)\times (-\infty,\delta^{\frac{2+\epsilon_0}{1-2s}}]\\
{\mbox{and }}&& H_\delta \subseteq \{ x_2\le \delta \} 
.\end{eqnarray*}
Moreover,
\begin{equation}\label{PLFG-SYM:cor}
\int_{\R^2} \frac{\chi_{H_\delta}(y)-\chi_{H_\delta^c}(y)}{|y-p|^{2+2s}}\,dy
\ge 0,\end{equation}
for any~$p\in (\partial H_\delta)\cap \{x_1\in \left(-1,-1+\frac{1}{100}\right)\}$
and
\begin{equation}\label{PLFG-SYM:2:cor}
\int_{\R^2} \frac{\chi_{H_\delta}(y)-\chi_{H_\delta^c}(y)}{|y-p|^{2+2s}}\,dy
\ge -\delta,\end{equation}
for any~$p\in (\partial H_\delta)\cap \{x_1\in\left[-1+\frac{1}{100},\;0\right]\}$.
\end{corollary}

\begin{proof} We scale the set $F_m$ constructed in
Proposition~\ref{LK:PO} by a factor of order~$\frac{1}{L_m}$
(then we also translate to the left
by a horizontal vector of length~$1$)
and take~$\delta:= \frac{1}{L_m^{\frac12-s-\epsilon_0}}$.
Notice that~$\delta$ is infinitesimal, due to~\eqref{L emme}.
Also, the estimates in~\eqref{PLFG-SYM:cor}
and~\eqref{PLFG-SYM:2:cor} follow from the
ones in~\eqref{PLFG-SYM}
and~\eqref{PLFG-SYM:2}, since the fractional curvature scales
by a factor proportional to~$L_m^{2s}$.

We also remark that the vertical stickiness
of~$F_m$ in Proposition~\ref{LK:PO} was bounded from below by~$a_m$,
and~$L_m \ge e^{\frac{1}{a_m}}$, by~\eqref{L emme}.
As a consequence, by scaling, the vertical stickiness of~$H_\delta$
here is bounded by
an order of~$\frac{a_m}{L_m}\ge \frac{1}{L_m\log L_m}$.
This quantity is in turn bounded by an order of~$ 
\frac{\delta^{\frac{2}{1-2s-2\epsilon_0}} }{|\log\delta|}$,
which we can bound by~$\delta^{\frac{2+\epsilon_0}{1-2s}}$,
up to renaming~$\epsilon_0$.
\end{proof}

We observe that while in~\eqref{PLFG-SYM:cor}
we obtained that the fractional mean curvature of the set is
nonnegative near~$\{x_1=\pm1\}$, from~\eqref{PLFG-SYM:2:cor}
we can only say that the fractional mean curvature of the set
near~$\{x_1=0\}$ is controlled by a small negative quantity
(and this cannot be improved, since at the points in which
the set reaches its highest level the fractional mean
curvature must be negative). By adding an additional small
contribution to the set in~$\{|x_1|\in (2,3)\}$, we can
obtain a complete subsolution, i.e. a set whose
fractional mean curvature is nonnegative.
Such subsolution has the important geometric feature
that the points along~$\{x_1=0\}$ detach from~$\{x_2=0\}$,
see Figure~\ref{FARA}.
The precise statement goes as follows:

\begin{figure}
    \centering
    \includegraphics[height=3.6cm]{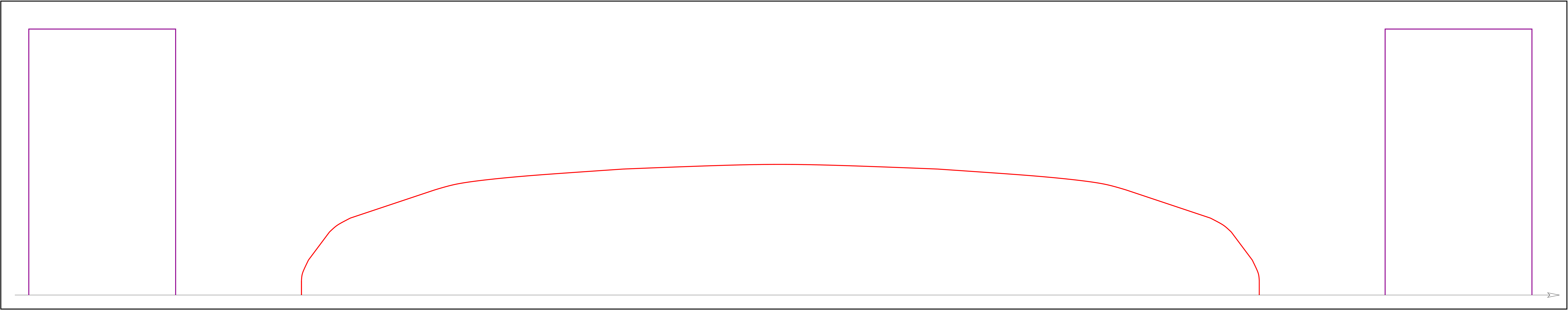}
    \caption{The barrier constructed in Proposition~\ref{LK:FARA}.}
    \label{FARA}
\end{figure}

\begin{proposition}\label{LK:FARA}
Fix~$\epsilon_0>0$ arbitrarily small.
There exist~$C>0$,
an infinitesimal sequence of
positive~$\delta$'s and sets~$E_\delta\subseteq\R^2$,
with $(\partial E_\delta)\cap \big((-\frac32,\frac32)\times(0,+\infty)\big)$
of class~$C^{1,1}$,
that are even symmetric with respect to the axis~$\{x_1=0\}$ and
satisfy the following properties:
\begin{eqnarray*}
&& E_\delta \cap \{x_1 \in(-\infty,-3)\cup(-2,-1)\} = 
\big( (-\infty,-3)\cup(-2,-1)\big)\times(-\infty,0), \\
&& E_\delta \cap \{x_1 \in[-3,-2]\} = [-3,-2]\times (-\infty, C\delta),\\
&& E_\delta \supseteq \R\times (-\infty,0),\\
&& E_\delta \supseteq (-1,1)\times (-\infty,\delta^{\frac{2+\epsilon_0}{1-2s}}]\\
{\mbox{and }}&& E_\delta \cap \{ |x_1|\le 1\}\subseteq \{ x_2\le \delta \} 
.\end{eqnarray*}
Moreover, for any~$p\in(\partial E_\delta)\cap \{|x_1|<1\}$, 
\begin{equation}\label{PLFG-SYM:FIN}
\int_{\R^2} \frac{\chi_{E_\delta}(y)-\chi_{E_\delta^c}(y)}{|y-p|^{2+2s}}\,dy
\ge 0.\end{equation}
\end{proposition}

\begin{proof} Let~$H_\delta$ be as in
Corollary~\ref{COR 2.2-BAR}.
We define~$E_\delta:= H_\delta\cup F_-\cup F_+$, where
$F_-:= (-3,-2)\times [0,C\delta)$ and
$F_+:= (2,3)\times [0,C\delta)$.
Then~$E_\delta$ satisfies all the desired geometric properties,
and~$E_\delta\supset H_\delta$. Therefore,
when~$p\in (\partial E_\delta)
\cap \{|x_1|\in \left(1-\frac{1}{100},\,1\right)\}$,
we have that~\eqref{PLFG-SYM:FIN}
follows from~\eqref{PLFG-SYM:cor}.
Moreover, 
when~$p\in (\partial E_\delta)
\cap \{|x_1|\le1-\frac{1}{100}\}$,
we have that~\eqref{PLFG-SYM:FIN} follows from~\eqref{PLFG-SYM:2:cor}
and the fact that~$|F_+|=|F_-|=C\delta$
(and one can choose $C>0$ conveniently large).
\end{proof}

\begin{remark}{\rm Concerning the statement of Proposition~\ref{LK:FARA},
by~\eqref{PLFG-SYM:FIN} (see in addition
Lemma~3.3 in~\cite{nostro}),
we also obtain that
\begin{equation}\label{PLFG-SYM:FIN:c}
\int_{\R^2} \frac{\chi_{E_\delta}(y)-\chi_{E_\delta^c}(y)}{|y-p|^{2+2s}}\,dy
\ge 0\end{equation}
for any~$p\in\overline{(\partial E_\delta)\cap \{|x_1|<1\}}$.
}\end{remark}

\section{Instability of the flat fractional minimal surfaces}\label{INST:SEC}

With the barrier constructed in Proposition~\ref{LK:FARA}
we are now in the position of proving Theorem~\ref{UNS}.
For this, we will take~$E$ and~$F$ as in the statement of
Theorem~\ref{UNS}.

\begin{proof}[Proof of Theorem~\ref{UNS}] Let~$E_\delta$ be as in
Proposition~\ref{LK:FARA}. The idea is to slide~$E_\delta$ (or, more precisely,
$E_{\frac{\delta}{C}}$)
from below. Namely, for any~$t\ge0$
we consider the set~$E(t):=E_{\frac{\delta}{C}} - te_2$.
For large~$t$, we have that~$E(t)\subseteq E$.
So we take the smallest~$t\ge0$ for which
such inclusion holds. We observe that Theorem~\ref{UNS}
would be proved if we show that such~$t$ equals to~$0$.

Then suppose, by contradiction, that
\begin{equation}\label{KJrf78}
t>0.\end{equation}
By construction, 
\begin{equation}\label{JG5rt}
E(t)\subseteq E\end{equation}
and
there exists a contact point between the two sets.
{F}rom the data outside~$[-1,1]\times\R$,
we have that all the contact points must lie
in~$[-1,1]\times\R$. 

Furthermore, 
\begin{equation}\label{no co}
{\mbox{no contact point can occur in~$(-1,1)\times\R$.}}\end{equation}
To check this, suppose that there exists~$p=(p_1,p_2)\in (\partial E(t))\cap
(\partial E)$ with~$|p_1|<1$. Then, using the
Euler-Lagrange equation in the viscosity sense for~$E$
(see Theorem~5.1 in~\cite{CRS})
and~\eqref{PLFG-SYM:FIN} we have that
$$ \int_{\R^2} \frac{\chi_{E}(y)-\chi_{E^c}(y)}{|y-p|^{2+2s}}\,dy
\le0\le
\int_{\R^2} \frac{\chi_{E(t)}(y)-\chi_{E^c(t)}(y)}{|y-p|^{2+2s}}\,dy.$$
Also, the opposite inequality holds, thanks to~\eqref{JG5rt},
and therefore~$E(t)$ and~$E$ must coincide. This would give that~$t=0$,
against our assumption. This proves~\eqref{no co}.

As a consequence, we have that all the contact points lie
on~$\{\pm1\}\times\R$. Since both~$\partial E(t)$
and~$\partial E$ are closed set, we can take the contact point
with lower vertical coordinate along~$\{x_1=\pm1\}$, and we denote it
by~$x_{o}^\pm =(\pm1,x_{o,2}^\pm)$.

Now, for any~$k\in\N$ (to be taken as large as we wish)
and any~$h\in [0,\,1/k]$
we consider the ball of small radius~$r>0$ (smaller than
the radius of curvature of~$E(t)$)
centered on the line~$\{x_2=x_{o,2}^\pm+h\}$
and we slide such ball to the left (towards~$\{x_1=-1\}$)
or to the right (towards~$\{x_1=1\}$) till it touches
either~$\partial E\cap \{ |x_1|<1\}$ or~$\{x_1=\pm 1\}$,
see Figure~\ref{ARCHI}.

\begin{figure}
    \centering
    \includegraphics[height=5.9cm]{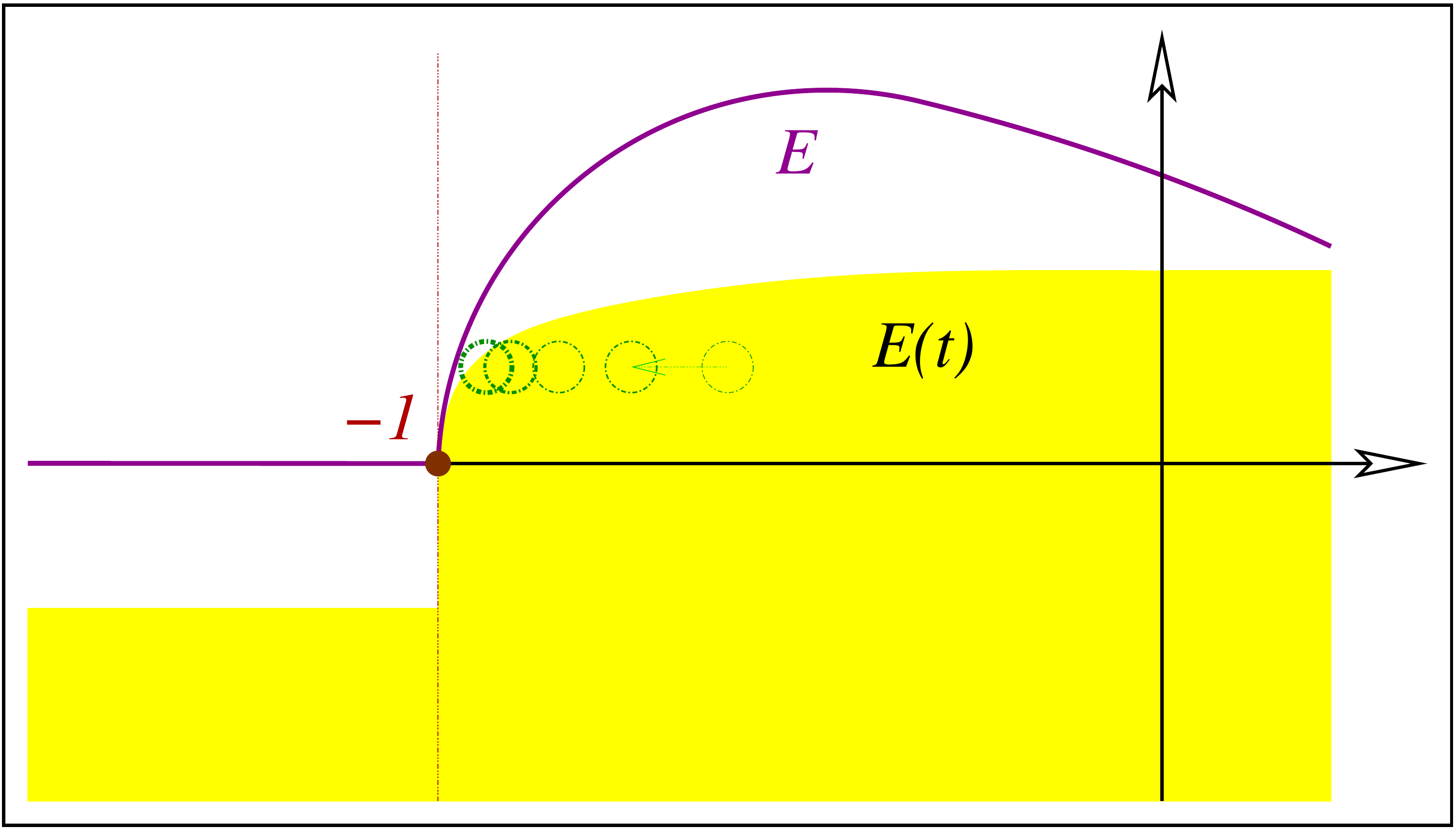}
    \caption{Sliding the balls from the barriers towards~$\partial E\cap \{ |x_1|<1\}$.}
    \label{ARCHI}
\end{figure}

We claim that there exists a sequence~$k\to+\infty$ for which
there exists~$h_k\in 
[0,\,1/k]$ such that the sliding of this ball (either to the right or to
the left)
touches~$\partial E\cap \{ |x_1|<1\}$. Indeed, if not,
we have that~$\partial E$, near~$\{x_1=\pm1\}$, stays 
above~$\{x_2=x_{o,2}^\pm+\alpha\}$,
for some~$\alpha>0$. But this would imply that we can keep
sliding~$E(t)$ a little more upwards, in contradiction
with the minimality of~$t$.

Therefore, we can assume that, for a suitable sequence~$k\to+\infty$,
we have that there exist points~$x_k=(x_{k,1},x_{k,2})
\in (\partial E)\cap\{|x_1|<1\}$
with~$x_{k,2}=x_{o,2}^\pm + h_k$ and~$h_k\in[0,\,1/k]$. By construction,
the points~$x_k$ must lie outside~$E(t)$, hence,
if~$r$ is small enough,
we have that~$|x_{k,1}|\to1$
as~$k\to+\infty$.

Hence, we assume that~$x_k\in (\partial E)\cap\{|x_1|<1\}$
and~$x_k\to x_o:=x_o^-$ as~$k\to+\infty$ (the case in which~$x_k\to x_o^+$
is completely analogous).
Then, by the Euler-Lagrange equation at the points~$x_k$ (see
Lemma~3.4 in~\cite{nostro}),  
we obtain that
\begin{equation}\label{PLFG-SYM:FIN:c:NINI}
\int_{\R^n} \frac{\chi_{E}(y)- \chi_{E^c}(y)}{|
x_o-y|^{n+2s}}\,dy\le0.\end{equation}
On the other hand, by~\eqref{PLFG-SYM:FIN:c}, 
\begin{equation}\label{PLFG-SYM:FIN:c:NI}
\int_{\R^2} \frac{\chi_{E(t)}(y)-\chi_{E^c(t)}(y)}{|x_o-y|^{2+2s}}\,dy
\ge 0.\end{equation}
Combining~\eqref{JG5rt}, \eqref{PLFG-SYM:FIN:c:NINI}
and~\eqref{PLFG-SYM:FIN:c:NI}, it follows that~$E(t)=E$.
Thus, from the values of~$E_\delta$ and~$E$ outside~$\{|x_1|\le 1\}$,
we conclude that~$t=0$.
This is in contradiction with~\eqref{KJrf78}
and so the desired result is proved.
\end{proof}

\begin{appendix}

\section{Symmetry properties and a variation on the proof
of Lemma \ref{TOR}}

Here we prove that the minimizers inherit the symmetry
properties of the boundary data:

\begin{lemma}\label{SYMM-LEMMA}
Let~$T:\R^n\to\R^n$ be an isometry, 
with~$T(\Omega)=\Omega$.
Assume that there exists~$N\in\N$ such that~$T^N(x)=x$ for every~$x\in
\Omega$.

Let~$E\subseteq\R^n$ be such that~$T(E)=E$.
Let~$E_*$ be the $s$-minimal set in a domain~$\Omega$
among all the sets~$F$ such that~$F\setminus\Omega=E\setminus\Omega$.
Then~$T(E_*)=E_*$.
\end{lemma}

\begin{proof} We let
$$ {\mathcal{F}}(u):=\frac12\iint_{\R^{2n}\setminus (\Omega^c)^2}
\frac{|u(x)-u(y)|^2}{|x-y|^{n+2s}}\,dx\,dy.$$
We observe that
$$ {\mathcal{F}}(\chi_E)=\per(E,\Omega).$$
Moreover, by Lemma~3 on page~685
in~\cite{PSV-AnnMat}, we have that
$$ {\mathcal{F}}(\min\{u, v\} ) + {\mathcal{F}}(\max\{u,v\}) \le
{\mathcal{F}} (u ) + {\mathcal{F}} (v),$$
and the equality holds if and only if
either~$u(x) \le v(x)$ or $v(x) \le u(x)$ for any~$ x\in\Omega$.

We use the observations above with~$u:=\chi_{E_*}$ and~$v:=\chi_{T(E_*)}$.
Notice that, in this case, $\min\{u, v\}=\chi_{E_*\cap T(E_*)}$
and~$\max\{u, v\}=\chi_{E_*\cup T(E_*)}$. Hence, we obtain
\begin{equation}\label{89-987}
\per\big( E_*\cap T(E_*),\Omega\big)+
\per\big( E_*\cup T(E_*),\Omega\big)\le
\per( E_*,\Omega)+\per( T(E_*),\Omega),
\end{equation}
and the equality holds if and only if
either~$\chi_{E_*}(x) \le \chi_{T(E_*)}(x)$ or $\chi_{T(E_*)}(x) \le \chi_{E_*}(x)$
for any~$ x\in\Omega$, that is,
if and only if 
\begin{equation}\label{89j67-0}
{\mbox{either~$E_*\cap\Omega\subseteq T(E_*)\cap\Omega$
or~$T(E_*)\cap\Omega\subseteq E_*\cap\Omega$.}}\end{equation}
Now we observe that
\begin{eqnarray*}
\per\big(T(E_*),\Omega\big) &=& L\big(
T(E_*)\cap\Omega, \R^n\setminus T(E_*)\big)+
L\big(\Omega\setminus T(E_*),
T(E_*)\setminus\Omega\big) \\
&=& L\big(T(E_*)\cap T(\Omega), \R^n\setminus T(E_*)\big)+
L\big(T(\Omega)\setminus T(E_*),
T(E_*)\setminus T(\Omega)\big)\\
&=&
L\big(T(E_*\cap\Omega), T(\R^n\setminus E_*)\big)+
L\big(T(\Omega\setminus E_*),
T(E_*\setminus \Omega)\big) \\
&=&
L(E_*\cap\Omega,\R^n\setminus E_*)+
L(\Omega\setminus E_*, E_*\setminus \Omega)
\\ &=& \per(E_*,\Omega).
\end{eqnarray*}
Substituting this in~\eqref{89-987}, we obtain that
\begin{equation}\label{89-987-2}
\per\big( E_*\cap T(E_*),\Omega\big)+
\per\big( E_*\cup T(E_*),\Omega\big)\le
2\per( E_*,\Omega).\end{equation}
On the other hand, 
\begin{equation}\label{0989uo}
T(E_*)\setminus\Omega=
T(E_*)\setminus T(\Omega)= T(E_*\setminus\Omega)
=T(E\setminus\Omega)=T(E)\setminus\Omega=E\setminus \Omega.\end{equation}
This says that~$E_*\cap T(E_*)$ and~$E_*\cup T(E_*)$
are admissible competitors for~$E_*$ and therefore
$$ \per( E_*,\Omega)\le\per\big( E_*\cap T(E_*),\Omega\big)
\;{\mbox{ and }}\;
\per( E_*,\Omega)\le\per\big( E_*\cup T(E_*),\Omega\big).$$
This implies that the equality holds in~\eqref{89-987-2},
and so in~\eqref{89-987}.

Therefore, \eqref{89j67-0} holds true.
So we suppose that~$E_*\cap\Omega\subseteq T(E_*)\cap\Omega$
(the case in which~$T(E_*)\cap\Omega\subseteq E_*\cap\Omega$
can be dealt with in a similar way).
Then we have that~$E_*\cap\Omega\subseteq T(E_*\cap\Omega)$.
By applying~$T$, we obtain~$T(E_*\cap\Omega)\subseteq T^2(E_*\cap\Omega)$,
and so, iterating the procedure
$$ E_*\cap\Omega\subseteq T(E_*\cap\Omega)\subseteq\dots
\subseteq T^{N-1}(E_*\cap\Omega)
\subseteq T^N(E_*\cap\Omega)=E_*\cap\Omega.$$
This shows that~$E_*\cap\Omega=T(E_*\cap\Omega)$, that is~$E_*\cap\Omega=
T(E_*)\cap\Omega$.

Also, by~\eqref{0989uo}, $E_*\setminus\Omega=
T(E_*)\setminus\Omega$. Therefore~$E_*=T(E_*)$, as desired.
\end{proof}

Now we give a different (and more general)
proof of Lemma~\ref{TOR}, according to the following result:

\begin{lemma}\label{lemma:trasformata}
Let $\sigma$, $\sigma_0\in (0,1)$, with $\sigma<2\sigma_0$. Then, for any $t>0$, we have 
\begin{equation}\label{prima112}
(-\Delta)^{\sigma_0} t_+^{\sigma} =-4\,\Gamma(1+\sigma)\,\Gamma(2\sigma_0-\sigma)\,
\sin\big(\pi(\sigma-\sigma_0)\big)t^{\sigma-2\sigma_0},\end{equation}
where $\Gamma$ is the gamma function. 

In particular, 
\begin{itemize}
\item if $\sigma=\sigma_0$, then, for any $t>0$, 
$$ (-\Delta)^{\sigma} t_+^{\sigma}=0,$$ 
\item if $\sigma>\sigma_0$, then for any $t>0$, 
$$ (-\Delta)^{\sigma_0} t_+^{\sigma}<0, $$
\item if $\sigma<\sigma_0$, then for any $t>0$, 
$$ (-\Delta)^{\sigma_0} t_+^{\sigma}>0. $$
\end{itemize}
\end{lemma}

\begin{proof} The proof is a modification of an argument
given in~\cite{claudia}.
In order to prove Lemma \ref{lemma:trasformata}, 
we will use the Fourier transform of~$|t|^q$
in the sense of distribution, where~$q\in\C\setminus\Z$.
Namely (see e.g. 
Lemma~2.23 on page~38 of~\cite{KOLDO})
\begin{equation}\label{KSp}
{\mathcal{F}} (|t|^q) = C_q\,|\xi|^{-1-q},
\end{equation}
with
\begin{equation}\label{J5K:1} C_q:=-2\Gamma(1+q)\,\sin\frac{\pi q}{2}.
\end{equation}
Notice that the map~$\R\ni t\mapsto|t|^q$
is even, and so we can rewrite~\eqref{KSp} as
\begin{equation}\label{KSp2}
{\mathcal{F}}^{-1} (|\xi|^q) =(2\pi)^{-1} C_q\,|t|^{-1-q}.
\end{equation}
Moreover,
$$ |t|^\sigma +\frac{1}{\sigma+1} \partial_t |t|^{\sigma +1} = 2t_+^\sigma.$$
Therefore, taking the Fourier transform and using~\eqref{KSp} with~$q:=\sigma$
and~$q:=\sigma+1$, we obtain that
\begin{eqnarray*}
2{\mathcal{F}}(t_+^\sigma) &=&
{\mathcal{F}}(|t|^\sigma) +\frac{1}{\sigma +1} {\mathcal{F}}\big(\partial_t |t|^{\sigma +1}\big)\\
&=&
{\mathcal{F}}(|t|^\sigma) +\frac{2 i\xi}{\sigma +1} {\mathcal{F}}(|t|^{\sigma +1})
\\ &=& C_\sigma \,|\xi|^{-1-\sigma}+\frac{2 i\xi}{\sigma +1} C_{\sigma +1}\,|\xi|^{-2-\sigma}.
\end{eqnarray*}
So, multiplying the equality above by $|\xi|^{2\sigma_0}$, we obtain that
$$ 2 |\xi|^{2\sigma_0} {\mathcal{F}}(t_+^\sigma)
= C_\sigma\,|\xi|^{2\sigma_0 -\sigma -1}+\frac{2 i\xi}{\sigma +1} 
C_{\sigma +1}\,|\xi|^{2\sigma_0-\sigma-2},$$
and so 
\begin{equation}\label{senzapi}
2 {\mathcal{F}}^{-1}\Big( |\xi|^{2\sigma_0} {\mathcal{F}}(t_+^\sigma)\Big)
=
C_\sigma \,{\mathcal{F}}^{-1}(|\xi|^{2\sigma_0-\sigma-1})
+\frac{2 C_{\sigma+1}\,i}{\sigma+1} 
{\mathcal{F}}^{-1}(\xi) * {\mathcal{F}}^{-1}(|\xi|^{2\sigma_0-\sigma-2})
\end{equation}
Now we claim that, for any test function~$g$, 
\begin{equation}\label{spero1534}
\left({\mathcal{F}}^{-1}(\xi) * g\right)(t) = -i \partial_t g(t).
\end{equation}
Indeed, 
\begin{eqnarray*}
&& \left({\mathcal{F}}^{-1}(\xi) * g\right)(t) = 
{\mathcal{F}}^{-1}\left(\xi {\mathcal{F}}g(\xi)\right)(t)\\
&&\qquad =
\frac{1}{2\pi}\int_{\R}dy\int_{\R}d\xi\,e^{iy\cdot(t-\xi)}\,y\,g(\xi)
= -
\frac{1}{2\pi i}\int_{\R}dy\int_{\R}d\xi\,\partial_\xi e^{iy\cdot(t-\xi)} \,g(\xi)\\&&\qquad =
\frac{1}{2\pi i}\int_{\R}dy\int_{\R}d\xi\, e^{iy\cdot(t-\xi)}\,\partial_\xi g(\xi)
=
\frac{1}{2\pi i}\int_{\R}dy\, e^{iy\cdot t}\,{\mathcal{F}}\big(\partial_\xi g\big)(y)\\ &&\qquad 
= \frac{1}{i}{\mathcal{F}}^{-1}\left({\mathcal{F}}(\partial_\xi g)\right)(t)
= -i\, \partial_\xi g(t),
\end{eqnarray*}
which shows~\eqref{spero1534}. 

Using~\eqref{spero1534} into~\eqref{senzapi}, we obtain that
\begin{eqnarray*}
2 {\mathcal{F}}^{-1}\Big( |\xi|^{2\sigma_0} {\mathcal{F}}(t_+^\sigma)\Big)
&=&C_\sigma\,{\mathcal{F}}^{-1}(|\xi|^{2\sigma_0-\sigma-1})
-\frac{C_{\sigma+1}\,i}{\sigma+1}\cdot i\, \partial_t
{\mathcal{F}}^{-1}(|\xi|^{2\sigma_0-\sigma-2})
\\ &=&C_\sigma\,{\mathcal{F}}^{-1}(|\xi|^{2\sigma_0-\sigma-1})
+\frac{C_{\sigma+1}}{\sigma+1}\partial_t {\mathcal{F}}^{-1}(|\xi|^{2\sigma_0-\sigma-2}).
\end{eqnarray*}
As a consequence, exploiting~\eqref{KSp2} with~$q:=2\sigma_0-\sigma-1$ and~$q:=2\sigma_0-\sigma-2$,
we have that
\begin{eqnarray*}
2 {\mathcal{F}}^{-1}\Big( |\xi|^{2\sigma_0} {\mathcal{F}}(t_+^\sigma)\Big)
&=& C_\sigma \,C_{2\sigma_0-\sigma-1}\,|t|^{\sigma-2\sigma_0}
+\frac{C_{\sigma+1}\,C_{2\sigma_0-\sigma-2}}{\sigma+1}
\partial_t |t|^{\sigma-2\sigma_0+1}\\
&=&C_\sigma\,C_{2\sigma_0-\sigma-1}\,|t|^{\sigma-2\sigma_0}
+\frac{\sigma-2\sigma_0+1}{\sigma+1}\cdot C_{\sigma+1}\,C_{2\sigma_0-\sigma-2} 
\,t\,|t|^{\sigma-2\sigma_0-1}.
\end{eqnarray*}
This gives that, for $t>0$, 
$$ 2 {\mathcal{F}}^{-1}\Big( |\xi|^{2\sigma_0} {\mathcal{F}}(t_+^\sigma)\Big)
= \left( C_\sigma\,C_{2\sigma_0-\sigma-1} + 
\frac{\sigma-2\sigma_0+1}{\sigma+1}\cdot C_{\sigma+1}\,C_{2\sigma_0-\sigma-2}\right) 
t^{\sigma-2\sigma_0}.$$
So we obtain that, up to a dimensional constant, for any $t>0$, 
\begin{equation}\label{plug}
(-\Delta)^{\sigma_0} (t_+^\sigma) = \left( C_\sigma\,C_{2\sigma_0-\sigma-1} + 
\frac{\sigma-2\sigma_0+1}{\sigma+1}\cdot C_{\sigma+1}\,C_{2\sigma_0-\sigma-2}\right) 
t^{\sigma-2\sigma_0}.\end{equation}

Now, we observe that
\begin{equation}\label{C sigma}
C_\sigma\,C_{2\sigma_0-\sigma-1}= 4 \,\Gamma(1+\sigma)\,\Gamma(2\sigma_0-\sigma)\,
\sin\left(\frac{\pi}{2}\sigma\right)\,\sin\left(\frac{\pi}{2}(2\sigma_0-\sigma-1)\right).
\end{equation}
Moreover, 
$$ \Gamma(2+\sigma)=(1+\sigma) \Gamma(1+\sigma)\;{\mbox{ and }}\;
\Gamma(2\sigma_0-\sigma)=(2\sigma_0-\sigma-1) \Gamma(2\sigma_0-\sigma-1).$$
As a consequence, recalling \eqref{J5K:1} and \eqref{C sigma}, 
\begin{eqnarray*}
&&\frac{\sigma-2\sigma_0+1}{\sigma+1}\cdot C_{\sigma+1}\,C_{2\sigma_0-\sigma-2}
\\& =& \frac{\sigma-2\sigma_0+1}{\sigma+1}\, 4\, \Gamma(2+\sigma)\, \Gamma(2\sigma_0-\sigma-1)
\,\sin\left( \frac{\pi}{2}(\sigma+1)\right)\,\sin\left(\frac{\pi}{2}(2\sigma_0-\sigma-2)\right) \\
&=& -4\, \Gamma(1+\sigma)\, \Gamma(2\sigma_0-\sigma)\,
\sin\left( \frac{\pi}{2}(\sigma+1)\right)\,\sin\left(\frac{\pi}{2}(2\sigma_0-\sigma-2)\right)\\
&=& - C_\sigma\,C_{2\sigma_0-\sigma-1}\cdot 
\frac{\sin\left( \frac{\pi}{2}(\sigma+1)\right)}{\sin\left(\frac{\pi}{2}\sigma\right)}\cdot 
\frac{\sin\left(\frac{\pi}{2}(2\sigma_0-\sigma-2)\right)}{
\sin\left(\frac{\pi}{2}(2\sigma_0-\sigma-1)\right)}.\end{eqnarray*}
Plugging this into \eqref{plug}, we get
\begin{equation*}
(-\Delta)^{\sigma_0} (t_+^\sigma) = C_\sigma\,C_{2\sigma_0-\sigma-1}\left(1 - 
\frac{\sin\left( \frac{\pi}{2}(\sigma+1)\right)}{\sin\left(\frac{\pi}{2}\sigma\right)}\cdot 
\frac{\sin\left(\frac{\pi}{2}(2\sigma_0-\sigma-2)\right)}{
\sin\left(\frac{\pi}{2}(2\sigma_0-\sigma-1)\right)}
\right) 
t^{\sigma-2\sigma_0}.\end{equation*}
Now, by elementary trigonometry, we see that
$$ \sin\left( \frac{\pi}{2}(\sigma+1)\right) = \cos\left(\frac{\pi}{2}\sigma\right)
\;{\mbox{ and }}\;
\sin\left(\frac{\pi}{2}(2\sigma_0-\sigma-2)\right) = - 
\cos\left(\frac{\pi}{2}(2\sigma_0-\sigma-1)\right). 
$$
Therefore, 
\begin{eqnarray*}
&& 1 - 
\frac{\sin\left( \frac{\pi}{2}(\sigma+1)\right)}{\sin\left(\frac{\pi}{2}\sigma\right)}\cdot 
\frac{\sin\left(\frac{\pi}{2}(2\sigma_0-\sigma-2)\right)}{
\sin\left(\frac{\pi}{2}(2\sigma_0-\sigma-1)\right)}\\
&=& 1 +
\frac{\cos\left( \frac{\pi}{2}\sigma\right)}{\sin\left(\frac{\pi}{2}\sigma\right)}\cdot 
\frac{\cos\left(\frac{\pi}{2}(2\sigma_0-\sigma-1)\right)}{
\sin\left(\frac{\pi}{2}(2\sigma_0-\sigma-1)\right)}\\
&=& \frac{\cos\left(\frac{\pi}{2}\sigma\right)}{\sin\left(\frac{\pi}{2}\sigma\right)}
\left[ \frac{\sin\left(\frac{\pi}{2}\sigma\right)}{\cos\left(\frac{\pi}{2}\sigma\right)}
+ \frac{\cos\left(\frac{\pi}{2}(2\sigma_0-\sigma-1)\right)}{
\sin\left(\frac{\pi}{2}(2\sigma_0-\sigma-1)\right)}\right]\\
&=& \frac{\cos\left(\frac{\pi}{2}\sigma\right)}{\sin\left(\frac{\pi}{2}\sigma\right)}\cdot 
\frac{\sin\left(\frac{\pi}{2}\sigma\right)\,\sin\left(\frac{\pi}{2}(2\sigma_0-\sigma-1)\right) 
+\cos\left(\frac{\pi}{2}\sigma\right)\,\cos\left(\frac{\pi}{2}(2\sigma_0-\sigma-1)\right)}{
\cos\left(\frac{\pi}{2}\sigma\right)\,\sin\left(\frac{\pi}{2}(2\sigma_0-\sigma-1)\right)}\\
&=& \frac{\cos\left(\frac{\pi}{2}\sigma\right)}{\sin\left(\frac{\pi}{2}\sigma\right)}\cdot 
\frac{\cos\left( \pi(\sigma-\sigma_0) +\frac{\pi}{2} \right)}{
\cos\left(\frac{\pi}{2}\sigma\right)\,\sin\left(\frac{\pi}{2}(2\sigma_0-\sigma-1)\right)}\\
&=& - \frac{\cos\left(\frac{\pi}{2}\sigma\right)}{\sin\left(\frac{\pi}{2}\sigma\right)}\cdot 
\frac{\sin\left( \pi(\sigma-\sigma_0)\right)}{
\cos\left(\frac{\pi}{2}\sigma\right)\,\sin\left(\frac{\pi}{2}(2\sigma_0-\sigma-1)\right)}\\
&=& - \frac{\sin\left( \pi(\sigma-\sigma_0)\right)}{
\sin\left(\frac{\pi}{2}\sigma\right)\, \sin\left(\frac{\pi}{2}(2\sigma_0-\sigma-1)\right)}.
\end{eqnarray*}
Accordingly, up to a dimensional constant,
$$ (-\Delta)^{\sigma_0} (t_+^\sigma) = -C_\sigma\,C_{2\sigma_0-\sigma-1} 
\frac{\sin\left( \pi(\sigma-\sigma_0)\right)}{
\sin\left(\frac{\pi}{2}\sigma\right)\, \sin\left(\frac{\pi}{2}(2\sigma_0-\sigma-1)\right)} 
t^{\sigma-2\sigma_0}.$$ 
So, recalling \eqref{C sigma}, we obtain that, for any $t>0$, 
$$ (-\Delta)^{\sigma_0} (t_+^\sigma) =-4\,\Gamma(1+\sigma)\,\Gamma(2\sigma_0-\sigma)\, 
\sin\left(\pi(\sigma-\sigma_0)\right),$$ 
which shows \eqref{prima112}. 

We finish the proof of Lemma \ref{lemma:trasformata} by noticing that 
\begin{itemize}
\item if $\sigma=\sigma_0$, then $\sin\left(\pi(\sigma-\sigma_0)\right)=0$,   
\item if $\sigma>\sigma_0$, then $\sin\left(\pi(\sigma-\sigma_0)\right)>0$,
\item if $\sigma<\sigma_0$, then $\sin\left(\pi(\sigma-\sigma_0)\right)<0$.
\end{itemize}
This implies the desired result. 
\end{proof}

\end{appendix}

\end{document}